\documentclass[a4paper,10pt,leqno]{amsart}
\title[Gromov-Lawson-Rosenberg Conjecture]
{The  Gromov-Lawson-Rosenberg Conjecture for  the  group $\mathbb{Z}/4\times \mathbb{Z}/4$ }

\author{No\'{e} B\'{a}rcenas }
                \email{barcenas@matmor.unam.mx}
         \urladdr{https://www.matmor.unam.mx/~barcenas/}

 \address{Centro de Ciencias Matem\'aticas. UNAM \\Ap.Postal 61-3 Xangari. Morelia, Michoac\'an MEXICO 58089}

\author{Luis Eduardo Garc\'ia-Hern\'andez}
	         \urladdr{https://mathlegh.com/}
\address{Luis Eduardo Garc\'ia-Hern\'andez, Universidad Nacional Aut\'onoma de M\'exico.}
\email{legh@ciencias.unam.mx}

\author{Raphael Reinauer}
\thanks{R. Reinauer is employed by Amazon; this work was carried out
independently and does not relate to the author's position at Amazon.}
\address{Independent researcher}
\email{reinauerr@gmail.com}

         \date{\today}

\usepackage{multicol}

\usepackage{float}
\usepackage[bookmarks=false,hidelinks]{hyperref}

\usepackage{calc}
\usepackage{enumerate,amssymb}
\usepackage[arrow,curve,matrix,tips,2cell]{xy}
  \SelectTips{eu}{10} \UseTips
  \UseAllTwocells
\usepackage{graphicx}
\usepackage[section]{placeins}
\usepackage{pdflscape}
\usepackage{tikz}
\usepackage{spectralsequences}
\DeclareMathAlphabet\EuR{U}{eur}{m}{n}
\SetMathAlphabet\EuR{bold}{U}{eur}{b}{n}

\makeindex             % used for the subject index
\theoremstyle{plain}
\newtheorem{theorem}{Theorem}[section]
\newtheorem{lemma}[theorem]{Lemma}
\newtheorem{proposition}[theorem]{Proposition}
\newtheorem{corollary}[theorem]{Corollary}

\theoremstyle{definition}
\newtheorem{definition}[theorem]{Definition}
\newtheorem{example}[theorem]{Example}

\newtheorem{remark}[theorem]{Remark}

{\catcode`@=11\global\let\c@equation=\c@theorem}

\numberwithin{figure}{section}

\newcommand{\comsquare}[8]                   % Produces a commutative square
{\begin{CD}
#1 @>#2>> #3\\
@V{#4}VV @V{#5}VV\\
#6 @>#7>> #8
\end{CD}
}

\newcommand{\xycomsquare}[8]                   % kommutatives Quadrat (xy-Version)
{\xymatrix
{#1 \ar[r]^{#2} \ar[d]^{#4} &
#3 \ar[d]^{#5}  \\
#6\ar[r]^{#7} &
#8
}
}

\newcommand{\Ext}{\operatorname{Ext}}

\newcommand{\im}{\operatorname{im}}

\newcommand{\Spin}{\operatorname{Spin}}

\newcommand{\higherlim}[3]{{\setbox1=\hbox{\rm lim}
        \setbox2=\hbox to \wd1{\leftarrowfill} \ht2=0pt \dp2=-1pt
        \mathop{\vtop{\baselineskip=5pt\box1\box2}}
        _{#1}}^{#2}#3}

\newcommand{\version}[1]                       %marks the date of last editing and compilation
{\begin{center} last edited on #1\\
last compiled on \today\\
name of texfile: \jobname
\end{center}
}

\newcounter{commentcounter}

\usepackage[T1]{fontenc}
\usepackage[utf8]{inputenc}
\usepackage{longtable}
\providecommand{\F}{\mathbb{F}}
\providecommand{\Hom}{\operatorname{Hom}}
\begin{document}
\begin{abstract}
We prove the Gromov--Lawson--Rosenberg conjecture for
$G=C_4\times C_4$ in every dimension at least five.
Correctly normalized eta invariants give the odd-dimensional geometric
lower bounds.  Actual Wood-sequence maps determine sufficient Adams
differentials, and the odd bounds force collapse.  In even dimensions,
explicit quotient manifolds, spin changes, and a connective Gysin
calculation using spin-Dirac indices generate the kernel row by row.
An affine torus-bundle representative supplies the class missing from
cyclic induction in dimension five; dimension six is treated separately.
\end{abstract}
\maketitle
\setcounter{tocdepth}{1}
\tableofcontents

\section{Introduction}\label{sec:introduction}
The Gromov--Lawson--Rosenberg problem asks whether the real index obstruction
is sufficient for a closed connected spin manifold of dimension at least five
to admit a metric of positive scalar curvature. For a finite group $G$, the
index is the composite
\[
\Omega_n^{\Spin}(BG)\xrightarrow{D}ko_n(BG)
\xrightarrow{\mathrm{per}}KO_n(BG)
\xrightarrow{A}KO_n(C_r^*(G)),
\]
where $D$ is the Atiyah--Bott--Shapiro orientation, $\mathrm{per}$ is
periodization, and $A$ is real assembly. The index obstruction and the
bordism-and-surgery approach are developed in
\cite{gromovlawsonsurgery,rosenbergalpha,stolzspin,stolzannals,botvinnikgilkeystolz}.
Throughout this paper $C_4=\mathbb Z/4$ and $G=C_4\times C_4$.

The conclusion is unstable: no product with a Bott manifold is required.
The stable theorem of Rosenberg--Stolz \cite{RosenbergStolz} does not
by itself give this assertion.  Nor can it be inferred from an
unrestricted index-sufficiency principle: Schick's counterexample
\cite{schick} shows that the unstable conjecture fails for other
fundamental groups.  The relevant positive input here is the cyclic
theorem of Botvinnik--Gilkey--Stolz \cite{botvinnikgilkeystolz} and
the bordism formulation reviewed in \cite{StolzSurvey2022}.

\begin{comment}This is a revision of the authors' earlier preprint \cite{BGHR}.
Its asserted rank-two conclusion is not used as an input.  The
specific work required beyond the cyclic case is to produce actual
PSC representatives for the smash summand, including the asymmetric
degree-five class and the higher even-filtration rows.
\end{comment}

\begin{theorem}[Main result]\label{theo:gromovz42}
Let $M$ be a closed connected smooth spin manifold of dimension $n\geq5$
with fundamental group $G=C_4\times C_4$, and let $f:M\to BG$ classify its
universal cover. Then $M$ admits a metric of positive scalar curvature if
and only if $(A\circ\mathrm{per}\circ D)[M,f]=0$.
\end{theorem}

The argument combines the cyclic input of
\cite{botvinnikgilkeystolz,siegemeyer,reinauer} with a calculation on the smash
summand of $BG$. Section~\ref{section:adams} fixes the spectral-sequence
conventions and Wood sequence. Section~\ref{section:splitting} identifies
the stable splitting and Steenrod modules. Section~\ref{sec:cyclic}
treats the cyclic factor, and Section~\ref{sec:smash} treats the complex
and real smash calculations. Section~\ref{section:eta} develops eta and
homological detection of the geometric generators.
Section~\ref{sec:low} supplies the exceptional cases: in dimension five,
cyclic induction misses a coset represented by an affine torus bundle;
in dimension six, product classes detect the remaining coefficient class.
Section~\ref{sec:conclusion} identifies the connective assembly kernel,
compares the Adams upper bounds with these detection results, and proves
the main theorem.

The even-dimensional proof does not infer orders from independent
filtration-zero edges.  It constructs generators in every surviving
row; the first mixed row is determined by an actual connective Gysin
map and the spin-Dirac index theorem.  The five-dimensional step uses the low-degree isomorphism of the spin
orientation together with spin bordism and surgery. Thus it does not simply
extend the $n>5$ reduction stated in Lemma~\ref{lem:stolz-bgs-reduction}
beyond its stated range.
\subsection{Acknowledgements}
The  first  author   thanks  support  of DGAPA-UNAM  grant IN101423. The  first  two authors  acknowledge  support  of   CONAHCYT Grant CF-2019 217392.  The first and  third  author thank Michael Joachim for the  generous introduction  to  the  topic,  and several  discussions  along the  years.
\subsection{AI tool disclosure}
The present article originated in the third author's Ph.D.\ thesis
\cite{reinauer}, defended in 2020. The first draft of the manuscript
was posted on arXiv in 2024 without AI assistance. All figures in
both that draft and the present article were prepared manually
using the \texttt{spectralsequences} package.

In June 2026, AI tools (GPT-5.6) were used to develop the five-dimensional
case in Section~\ref{sec:low}, and to prepare
the exact-arithmetic reproducibility checks listed in
Appendix~\ref{app:code}, to assist with revising mathematical arguments
and identifying possible gaps, and to improve grammar, spelling,
and phrasing. GitHub Copilot was used to assist in
developing the accompanying Python verification scripts, and  to prepare their documentation and test runner.

All mathematical content was carefully reviewed, verified, and
substantially revised by the authors, who take full responsibility
for the final manuscript and the accompanying reproducibility checks.

\section{Notation, spectral sequences, and the Wood sequence}\label{section:adams}
\subsection{Global conventions}\label{sec:notation}
Set $Y=BC_4$ and $X=Y\wedge Y$. Ordinary homology and cohomology are
unreduced unless a tilde is displayed. Throughout, $\widetilde{ko}_*(Z)$
and $\widetilde{ku}_*(Z)$ denote reduced theories, while $ko_*(BG)$ is
unreduced. In every Adams Ext expression, $H^*(Z;\F_2)$ means the
positive-degree, reduced cohomology module; a displayed cohomology ring
includes its unit. This convention also applies to $M=H^*(Z;\F_2)$.

We use $(s,t)$ for Ext bidegrees and $(n,s)=(t-s,s)$ for
stem--filtration coordinates. Thus $d_r$ has degree $(r,r-1)$ in
$(s,t)$ and $(-1,r)$ in $(n,s)$. The class $h_0$ detects multiplication
by $2$, and $h_1$ detects $\eta$; on coefficient modules we also write
$\eta$ for its detecting class. The symbols $v$ (in $ku$) and
$\omega,\mu$ (in $ko$) denote the
coefficient classes in degrees $2,4,8$, with
$c(\omega)=2v^2$ and $c(\mu)=v^4$.

The map $R$ in the Wood sequence lowers homological degree by two;
it is not the degree-preserving realification map. The Greek letter
$\eta$ in this sequence denotes the coefficient operation. Eta invariants
of manifolds are distinguished by their manifold and representation arguments.
Local bases in the geometric constructions are declared where used.
In the Adams calculations, ordinary cohomology monomials are also used
as labels for their dual homology/Ext basis vectors; products displayed
inside an Ext class refer to these labels, not to a cohomology ring
structure on the manifold.  Integral $ku$ generators $e_j$ are defined
separately by the connective Gysin sequence.
The mod-two generators are $x,z$ in degrees one and two, with pullbacks
$x_0,z_0,x_1,z_1$. Mod-four generators $T,u$ and their integral Bockstein
are defined separately in Section~\ref{sec:low}.

The algebra $C_r^*(G)$ in this paper is the reduced \emph{real} group
$C^*$-algebra; all assembly maps and their coefficient groups are real.
All Adams spectral sequences are at the prime $2$ and converge to the
$2$-completed theory. The reduced groups under consideration are finite
$2$-primary groups, so completion does not change them.
\subsection{Adams $E_2$--identifications and hypotheses}

Let $\mathcal{A}$ be the mod-2 Steenrod algebra, and write
\[
\mathcal{A}(1) := \langle Sq^1, Sq^2\rangle \subset \mathcal{A},\qquad E(1) := \Lambda(Q_0,Q_1)\subset \mathcal{A}
\]
with
\[
Q_0 = Sq^1,\qquad Q_1 = [Sq^2,Sq^1]=Sq^2Sq^1+Sq^1Sq^2.
\]
We use the notation $\mathcal{A}//B := \mathcal{A}\otimes_B \mathbb{F}_2$ for the quotient $\mathcal{A}$--module induced from the trivial $B$--module
$\mathbb{F}_2$.

\begin{lemma}[Adams $E_2$-identifications]\label{lem:adams-e2-identifications}
For any based CW complex \(Z\) of finite type, with the reduced convention above,
\[
E^{s,t}_2(ko\wedge Z) \cong \operatorname{Ext}^{s,t}_{\mathcal{A}(1)}\bigl(H^*(Z;\mathbb{F}_2),\mathbb{F}_2\bigr),\qquad
E^{s,t}_2(ku\wedge Z) \cong \operatorname{Ext}^{s,t}_{E(1)}\bigl(H^*(Z;\mathbb{F}_2),\mathbb{F}_2\bigr).
\]
\end{lemma}

\begin{proof}
By definition, the classical Adams spectral sequence at the prime \(2\) for a spectrum $Z$ has
\[
E^{s,t}_2(Z) = \operatorname{Ext}^{s,t}_{\mathcal{A}}\bigl(H^*(Z;\mathbb{F}_2),\mathbb{F}_2\bigr).
\]
For connective $ko$ and $ku$, one uses the standard $A$--module descriptions
\[
H^*(ko;\mathbb{F}_2)\cong {\mathcal{A}}//\mathcal{A}(1),\qquad H^*(ku;\mathbb{F}_2)\cong \mathcal{A}//E(1).
\]
Hence
\[
E^{s,t}_2(ko\wedge Z)\cong \operatorname{Ext}^{s,t}_{\mathcal A}\bigl((\mathcal{A}//\mathcal{A}(1))\otimes H^*(Z),\mathbb{F}_2\bigr),
\]
and similarly with $E(1)$ for $ku$. Since $\mathcal{A}$ is free as a right $\mathcal{A}(1)$--module and as a right $E(1)$-module,
change-of-rings yields
\[
\operatorname{Ext}^{s,t}_\mathcal{A}\bigl((\mathcal{A}//\mathcal{A}(1))\otimes H^*(Z),\mathbb{F}_2\bigr)\cong
\operatorname{Ext}^{s,t}_{\mathcal{A}(1)}\bigl(H^*(Z),\mathbb{F}_2\bigr),
\]
and similarly
\[
\operatorname{Ext}^{s,t}_{\mathcal{A}}\bigl((\mathcal A//E(1))\otimes H^*(Z),\mathbb{F}_2\bigr)\cong
\operatorname{Ext}^{s,t}_{E(1)}\bigl(H^*(Z),\mathbb{F}_2\bigr).
\]
\end{proof}

\subsection[The eta-c-R cofiber sequence and the E2 long exact sequence]{The $\eta$--$c$--$R$ cofiber sequence and the $E_2$ long exact sequence}

There is a cofiber sequence of spectra
\[
\Sigma ko \xrightarrow{\ \eta\ } ko \xrightarrow{\ c\ } ku \xrightarrow{\ R\ } \Sigma^2 ko,
\]
where $c$ is complexification and $R$ is (shifted) realification.

We need the existence of a short exact sequence of $\mathcal{A}$--modules expressing $\mathcal{A}//E(1)$ as an extension
of $\mathcal{A}//\mathcal{A}(1)$ by a $2$--shifted copy.

\begin{lemma}[The algebraic SES underlying $\eta$-$c$-$R$]\label{lem:algebraic-ses-eta-c-r}
There is a short exact sequence of graded left $\mathcal{A}$--modules
\[
0 \to \Sigma^2(\mathcal{A}//\mathcal{A}(1)) \xrightarrow{\ \iota\ } \mathcal{A}//E(1) \xrightarrow{\ p\ } \mathcal{A}//\mathcal{A}(1) \to 0.
\]
Consequently, for any based finite-type CW complex \(Z\), put \(M=H^*(Z;\mathbb F_2)\). Applying
$\operatorname{Ext}_\mathcal{A}(-\otimes M,\mathbb{F}_2)$ and using Lemma~\ref{lem:adams-e2-identifications} yields a long exact sequence
\[
\cdots \xrightarrow{h_1} \operatorname{Ext}^{s,t}_{\mathcal{A}(1)}(M,\mathbb{F}_2)
\xrightarrow{c} \operatorname{Ext}^{s,t}_{E(1)}(M,\mathbb{F}_2)
\xrightarrow{R} \operatorname{Ext}^{s,t-2}_{\mathcal{A}(1)}(M,\mathbb{F}_2)
\xrightarrow{h_1} \cdots,
\]
where the internal shift $t\mapsto t-2$ comes from the suspension $\Sigma^2$ in the SES.
\end{lemma}

\begin{proof}
As a right $E(1)$--module, $\mathcal{A}(1)$ is free with basis $\{1,Sq^2\}$, hence $\mathcal{A}(1)//E(1)$ has a two-step
filtration with subquotients $\mathbb{F}_2$ in degrees $0$ and $2$. Inducing up to $\mathcal{A}$ gives a filtration of
$\mathcal{A}//E(1)$ with subquotients $\mathcal{A}//\mathcal{A}(1)$ and $\Sigma^2(\mathcal{A}//\mathcal{A}(1))$, yielding the stated SES. (The inclusion
$\iota$ is induced by the degree-2 basis class $Sq^2$, and $p$ is the quotient killing that summand.)

Applying $\operatorname{Ext}_\mathcal{A}(-\otimes M,\mathbb{F}_2)$ yields the associated long exact sequence in Ext.
Finally, Lemma~\ref{lem:adams-e2-identifications} identifies these Ext$_\mathcal{A}$ groups with Ext$_{\mathcal{A}(1)}(M,\mathbb{F}_2)$ and Ext$_{E(1)}(M,\mathbb{F}_2)$
by change-of-rings, and the shift $\Sigma^2$ becomes $t\mapsto t-2$ in Ext.
\end{proof}

\subsubsection*{Coefficient behavior}

We record the coefficient facts used later. Write $v\in \pi_2(ku)\cong \mathbb{Z}$ for the connective Bott generator and
choose $\mu\in \pi_8(ko)\cong \mathbb{Z}$ so that $c(\mu)=v^4$. (With this convention, $\mu$ is a Bott element with
$\eta\mu\neq 0$.) Let $\omega\in \pi_4(ko)\cong \mathbb{Z}$ be the standard generator.

\begin{lemma}[Coefficients for $c$ and $R$]\label{lem:coeff-c-r}
In homotopy groups, the maps in the cofiber sequence satisfy:
\[
c(\eta)=0,\qquad c(\omega)=2v^2,\qquad c(\mu)=v^4,
\]
and
\[
R(v)=2,\qquad R(v^2)=\eta^2,\qquad R(v^3)=\omega.
\]
\end{lemma}

\begin{proof}
These follow from the long exact coefficient sequence of the cofiber sequence $\Sigma ko\to ko\to ku\to \Sigma^2 ko$,
together with the known low-degree groups $\pi_*(ko)$ and $\pi_*(ku)$ and the fact $\pi_1(ku)=0$.

For instance, $R(v)=2$ is forced because $\eta:\pi_0(ko)=\mathbb{Z}\to \pi_1(ko)=\mathbb{Z}/2$ is reduction mod $2$,
hence $\ker(\eta)=2\mathbb{Z}$ and $\operatorname{im}(R)=2\mathbb{Z}$. The remaining statements are forced similarly by
exactness and degree considerations.
\end{proof}

The coefficient charts are displayed in Figure~\ref{fig:coefficients}.
The horizontal coordinate is the stem $t-s$ and the vertical coordinate is
the Adams filtration $s$; vertical lines represent multiplication by $h_0$.

\begin{figure}[htbp]
\centering
\begin{minipage}[t]{.48\linewidth}\centering
\includegraphics[width=\linewidth]{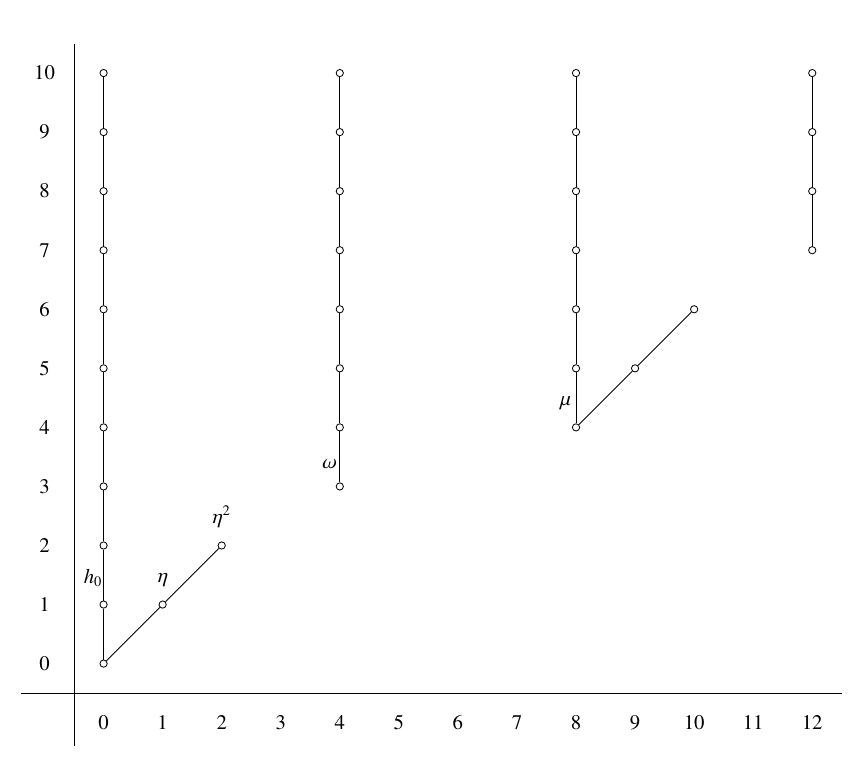}\\
\small Real coefficients: $ko$.
\end{minipage}\hfill
\begin{minipage}[t]{.48\linewidth}\centering
\includegraphics[width=\linewidth]{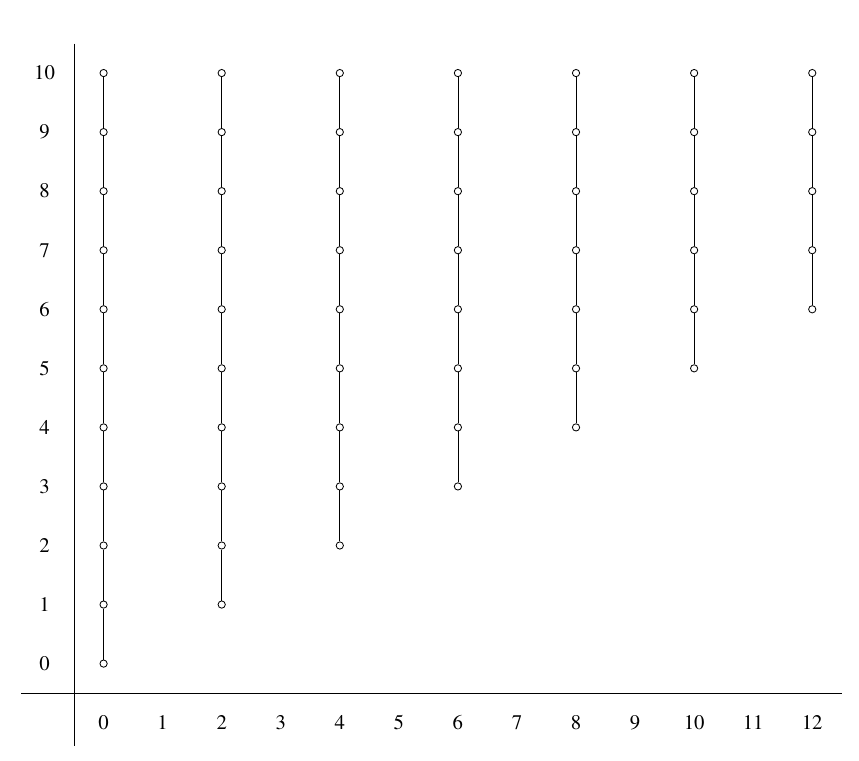}\\
\small Complex coefficients: $ku$.
\end{minipage}
\caption{The mod-two Adams coefficient charts. The real generators
$\eta$, $\omega$, and $\mu$ occur at $(1,1)$, $(4,3)$, and $(8,4)$;
the complex Bott generator $v$ occurs at $(2,1)$. Only a finite range is shown.}
\label{fig:coefficients}
\end{figure}

\subsection{The Atiyah--Hirzebruch spectral sequence}\label{sec:ahss}
The homological Atiyah--Hirzebruch spectral sequence of a connective spectrum
$E$ has
\[
E^2_{p,q}=H_p(Z;\pi_qE)\Longrightarrow E_{p+q}(Z),\qquad
 d_r:E^r_{p,q}\longrightarrow E^r_{p-r,q+r-1}.
\]
For reduced homology one uses the reduced version. Its cohomological
counterpart has $E_2^{p,q}=H^p(Z;E^q(*))$ and total degree $p+q$.
In the low-degree homological $ko$ calculation, $d_2$ on the $q=0$ row is
dual to $Sq^2$ after reduction modulo two, and on the $q=1$ row is dual
to $Sq^2$. These are the operations used in
Lemma~\ref{lem:cyclic-low-relation}. We distinguish the homological AHSS
edge from the mod-two Adams edge throughout the geometric argument.

\section{Stable splitting, cohomology, and Steenrod modules}\label{section:splitting}
\begin{lemma}[The stable product splitting]\label{lem:stable-product-splitting}\label{theo:splitting}
For every spectrum \(E\), reduced \(E\)-homology satisfies the natural splitting
\[
\widetilde E_n(BG)\cong
\widetilde E_n(Y)\oplus\widetilde E_n(Y)\oplus\widetilde E_n(X).
\]
In particular,
\[
ko_n(BG)\cong ko_n(*)\oplus \widetilde{ko}_n(Y)\oplus \widetilde{ko}_n(Y)
\oplus \widetilde{ko}_n(X).
\]
\end{lemma}

\begin{proof}
For based spaces \(Y\) there is a natural stable splitting
\[
\Sigma(Y\times Y)\simeq \Sigma Y\vee \Sigma Y\vee \Sigma(Y\wedge Y).
\]
Since \(BG\simeq BC_4\times BC_4=Y\times Y\), applying reduced \(E\)-homology gives the displayed direct sum.
The second display follows by adding the basepoint summand.
\end{proof}
\subsection{$H^*(B\mathbb{Z}/4;\mathbb{F}_2)$ and Steenrod operations}

Since we repeatedly reduce to summands built from $B\mathbb{Z}/4$,
we record the (co)homological input carefully.

\begin{lemma}[Cohomology of $B\mathbb{Z}/4$ and Steenrod action]\label{lem:cohomology-bz4}
Let $Y=B\mathbb{Z}/4$. Then
\[
H^*(Y;\mathbb{F}_2)\cong \mathbb{F}_2[x,z]/(x^2),\qquad |x|=1,\ |z|=2.
\]
The Steenrod operations are
\[
Sq^1(x)=x^2=0,\qquad Sq^1(z)=0,\qquad Sq^2(x)=0,\qquad Sq^2(z)=z^2.
\]
\end{lemma}

\begin{proof}
The ring structure is standard for cyclic 2--groups: $H^1(Y)$ is one-dimensional with generator $x$,
and $H^2(Y)$ is one-dimensional with generator $z$.

For $Sq^1(x)=x^2=0$: the operation $Sq^1$ equals the mod-2 Bockstein for the short exact sequence
$0\to \mathbb{Z}/2\to \mathbb{Z}/4\to \mathbb{Z}/2\to 0$. The class $x\in H^1(B\mathbb{Z}/4;\mathbb{Z}/2)\cong
\mathrm{Hom}(\mathbb{Z}/4,\mathbb{Z}/2)$ lifts to a class in $H^1(B\mathbb{Z}/4;\mathbb{Z}/4)\cong
\mathrm{Hom}(\mathbb{Z}/4,\mathbb{Z}/4)$ (send $1\mapsto 1$). Therefore the Bockstein, hence $Sq^1$, on $x$ is $0$.
By the unstable axiom, $Sq^1(x)=x^2$, so $x^2=0$.

The unstable axiom gives $Sq^i(x)=0$ for $i>|x|=1$, so $Sq^2(x)=0$. The top-square axiom gives $Sq^2(z)=z^2$ for
$|z|=2$.

For $Sq^1(z)$: the generator $z\in H^2(B\mathbb{Z}/4;\mathbb{F}_2)$ is the mod-2 reduction of the integral generator in
$H^2(B\mathbb{Z}/4;\mathbb{Z})\cong \mathbb{Z}/4$. Since $Sq^1$ is the mod-2 Bockstein associated to
$0\to \mathbb{Z}/2\to \mathbb{Z}/4\to \mathbb{Z}/2\to 0$, and $z$ lifts to an integral class, we have $Sq^1(z)=0$.
\end{proof}

\begin{corollary}[Useful Cartan consequences]\label{cor:useful-cartan-consequences}
In $H^*(B\mathbb{Z}/4;\mathbb{F}_2)$ one has, for $n\ge 0$,
\[
Sq^2(z^n)=(n \bmod 2)\, z^{n+1},\qquad Sq^2(xz^n)=x\,Sq^2(z^n)=(n \bmod 2)\, xz^{n+1}.
\]
In particular, for $k>0$,
\[
Sq^2(z^{2k})=0,\quad Sq^2(z^{2k-1})=z^{2k},\quad Sq^2(xz^{2k})=0,\quad Sq^2(xz^{2k-1})=xz^{2k}.
\]
\end{corollary}

\begin{proof}
Apply the Cartan formula for $Sq^2$ and use Lemma~\ref{lem:cohomology-bz4}. The binomial coefficient computation for $Sq^2(z^n)$ is standard:
$Sq^2(z^n)=\binom{n}{1} z^{n+1}$ for $|z|=2$. The formula for $Sq^2(xz^n)$ uses $Sq^2(x)=Sq^1(x)=0$, hence
$Sq^2(xz^n)=x\,Sq^2(z^n)$.
\end{proof}

\subsection{The $E(1)$--module structure on $M$ and $\operatorname{Ext}^0_{E(1)}$ generators}

For $M=H^*(B\mathbb{Z}/4;\mathbb{F}_2)$, the $E(1)$--action via $Q_0=Sq^1$ and
$Q_1=[Sq^2,Sq^1]$ is trivial. Indeed, by Lemma~\ref{lem:cohomology-bz4}, $Sq^1(x)=0$ and $Sq^1(z)=0$, so $Q_0=Sq^1=0$ on $M$.
For $Q_1=[Sq^2,Sq^1]=Sq^2Sq^1+Sq^1Sq^2$, we have $Sq^1=0$ on $M$, hence $Q_1=0$ on $M$ as well.

Consequently,
\[
\operatorname{Ext}^{*,*}_{E(1)}(M,\mathbb{F}_2)\cong \mathrm{Hom}_{\mathbb{F}_2}(M,\mathbb{F}_2)\otimes
\operatorname{Ext}^{*,*}_{E(1)}(\mathbb{F}_2,\mathbb{F}_2)\cong \mathrm{Hom}(M,\mathbb{F}_2)\otimes \mathbb{F}_2[h_0,v],
\]
with $|h_0|=(1,1)$ and $|v|=(1,3)$.

In particular, every monomial $xz^k$ gives a filtration-0 class in
$\operatorname{Ext}^{0,1+2k}_{E(1)}(M,\mathbb{F}_2)$, and similarly $z^k$ gives a filtration-0 class in
$\operatorname{Ext}^{0,2k}_{E(1)}(M,\mathbb{F}_2)$. We use the same notation $xz^k$ and $z^k$ for these
$\operatorname{Ext}^0$ classes.

\subsection{$\mathcal A(1)$--module decomposition for $B\mathbb Z/4\wedge B\mathbb Z/4$}

The purpose of this section is to record the $\mathcal A(1)$--module structure on
\[
M:=H^*(B\mathbb Z/4\wedge B\mathbb Z/4;\F_2)
\]
in a form that is usable for the $\eta$--$c$--$R$ sequence. The calculation is elementary: the standard computation
of the mod-$2$ cohomology of a cyclic group gives
\[
H^*(B\mathbb Z/4;\F_2)\cong \F_2[x,z]/(x^2),\qquad |x|=1,\quad |z|=2,
\]
and the Steenrod action is, as in Corollary~\ref{cor:useful-cartan-consequences},
\[
Sq^1=0,\qquad Sq^2(z^n)=(n\bmod 2)z^{n+1},\qquad Sq^2(xz^n)=(n\bmod 2)xz^{n+1}.
\]
For the cohomology of cyclic groups and the Steenrod operations used here, see for example
\cite{browncohomologygroups,ademmilgram,moshertangora}. Since $Sq^1=0$, the $\mathcal A(1)$--module structure in
this section is determined by $Sq^2$.

\paragraph{Basic blocks.}
We use the following $\mathcal A(1)$--modules, always with trivial $Sq^1$.
\begin{itemize}
\item $P$ is one-dimensional, generated by $p$, with $Sq^2(p)=0$.
\item $M_B$ is the two-cell bow module generated by $u,v$, with $|v|=|u|+2$ and
\[
Sq^2(u)=v,\qquad Sq^2(v)=0.
\]
Equivalently, $M_B\cong \mathcal A(1)//E(1)$.
\item $M_{SB}$ is the four-cell summed bow generated by $d_0,d_{20},d_{21},d_4$, with
\[
Sq^2(d_0)=d_{20}+d_{21},\qquad Sq^2(d_{20})=Sq^2(d_{21})=d_4,\qquad Sq^2(d_4)=0.
\]
\end{itemize}

The notation $M_{SB}$ stands for ``summed bow'' and records the basis naturally obtained from the product of two
bow modules. It should not be read as an indecomposable summand. Indeed, if
\[
e=d_{20}+d_{21},\qquad f=d_{20},
\]
then
\[
Sq^2(d_0)=e,\qquad Sq^2(e)=0,\qquad Sq^2(f)=d_4,\qquad Sq^2(d_4)=0,
\]
so
\[
M_{SB}\cong M_B\oplus \Sigma^2M_B
\]
as an $\mathcal A(1)$--module. We nevertheless keep the $M_{SB}$ basis because the maps $c$ and $R$ below are
most transparent in the original monomial basis.

Figure~\ref{fig:basic-modules} depicts these two blocks. Figures~\ref{fig:block-ext} and~\ref{fig:summed-block-ext}
show their Ext groups in stem--filtration coordinates. In these charts,
$w_0$ denotes the bottom Ext generator of the first bow, and $w'_2$ the bottom
Ext generator of the shifted bow. The labels $v^jw_0$ and $v^jw'_2$ record
the successive bow generators, with the lifting convention used for $c$
and $R$ below; they do not assert that $v\in\pi_2(ko)$.
\begin{figure}[htbp]
\centering
\begin{minipage}{.2\linewidth}\centering
\includegraphics[width=\linewidth,trim={69.68bp 30.27bp 30.86bp 63.58bp},clip]{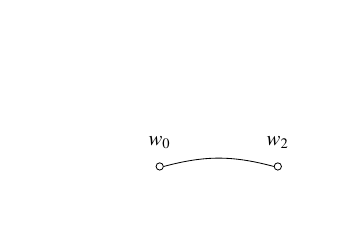}\\
\small $M_B$.
\end{minipage}\hfill
\begin{minipage}{.4\linewidth}\centering
\includegraphics[width=\linewidth,trim={69.69bp 20.16bp 30.59bp 63.31bp},clip]{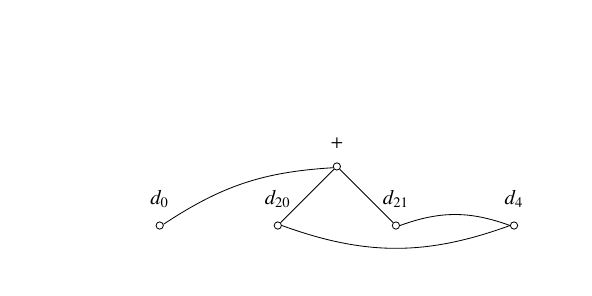}\\
\small $M_{SB}$.
\end{minipage}
\caption{The bow and summed-bow modules. In the left diagram
$w_0=u$ and $w_2=v$ in the local two-cell notation. The curves encode
$Sq^2$; in the right diagram the upper vertex represents $d_{20}+d_{21}$.}
\label{fig:basic-modules}
\end{figure}
\begin{figure}[htbp]
\centering
\includegraphics[width=.85\linewidth,clip]{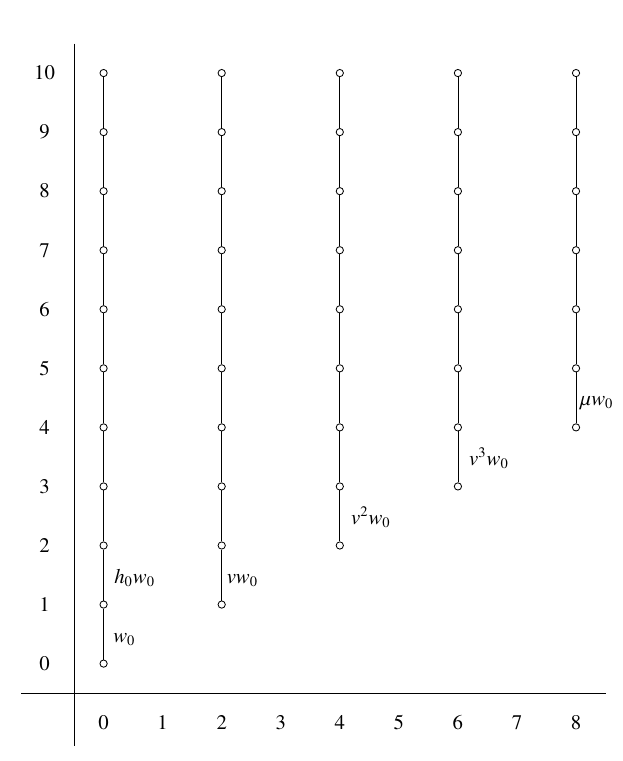}
\caption{A finite window of $\Ext_{\mathcal A(1)}(M_B,\F_2)$ in
stem--filtration coordinates. Vertical lines denote $h_0$.}
\label{fig:block-ext}
\end{figure}
\begin{figure}[htbp]
\centering
\includegraphics[width=.85\linewidth,clip]{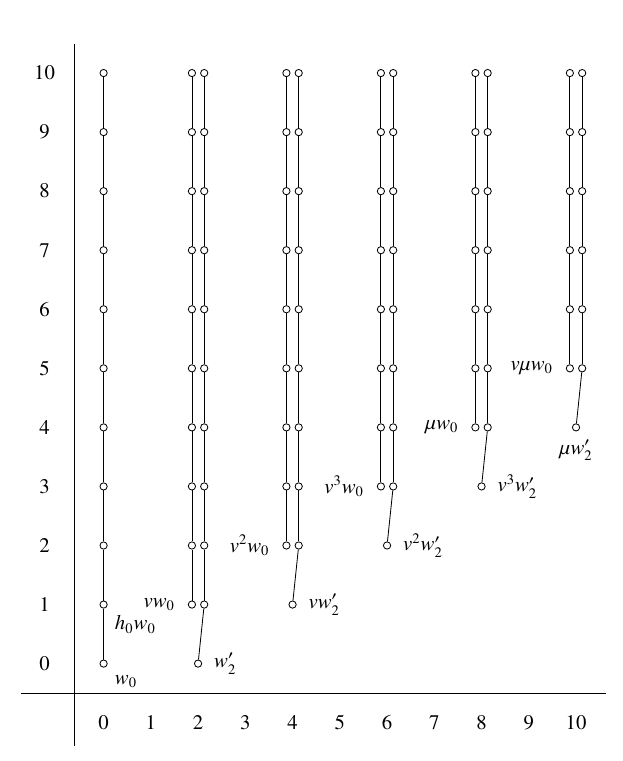}
\caption{A finite window of $\Ext_{\mathcal A(1)}(M_{SB},\F_2)$:
the bow chart and its shift by two in the stem. Vertical lines denote $h_0$.}
\label{fig:summed-block-ext}
\end{figure}

\paragraph{The one-factor decomposition.}
Let $M_i$ denote the cohomology of the $i$th factor, with generators $x_i,z_i$. The basis of
$M_i$ is
\[
\{x_i\}\cup\{z_i^n\mid n\geq 1\}\cup\{x_i z_i^n\mid n\geq 1\}.
\]
Define submodules
\[
P_i:=\F_2\{x_i\},
\]
and, for $k\geq 1$ and $\ell\geq 0$,
\[
C_{i,k}:=\F_2\{z_i^{2k-1},z_i^{2k}\},\qquad
D_{i,\ell}:=\F_2\{x_i z_i^{2\ell+1},x_i z_i^{2\ell+2}\}.
\]
Then
\[
M_i=P_i\oplus\bigoplus_{k\geq 1}C_{i,k}\oplus\bigoplus_{\ell\geq 0}D_{i,\ell},
\]
where $P_i\cong \Sigma^1P$, $C_{i,k}\cong \Sigma^{4k-2}M_B$, and
$D_{i,\ell}\cong \Sigma^{4\ell+3}M_B$.

By the reduced K\"unneth isomorphism,
\[
M\cong M_0\otimes M_1.
\]
For a monomial
\[
m(\epsilon_0,a_0;\epsilon_1,a_1):=
x_0^{\epsilon_0}z_0^{a_0}x_1^{\epsilon_1}z_1^{a_1},
\]
where $\epsilon_i\in\{0,1\}$, $a_i\geq 1$ if $\epsilon_i=0$, and $a_i\geq 0$ if $\epsilon_i=1$, Cartan gives
\begin{equation}\label{eq:Sq2-smash-monomial}
\begin{split}
Sq^2\big(m(\epsilon_0,a_0;\epsilon_1,a_1)\big)
&=(a_0\bmod 2)m(\epsilon_0,a_0+1;\epsilon_1,a_1)\\
&\quad +(a_1\bmod 2)m(\epsilon_0,a_0;\epsilon_1,a_1+1).\end{split}
\end{equation}

\begin{lemma}[Block decomposition of $H^*(B\mathbb Z/4\wedge B\mathbb Z/4;\F_2)$]
\label{lemma:A1decomp-smash-firstprinciples}
As a graded $\mathcal A(1)$--module, $M$ is the direct sum
\begin{align*}
M\cong{}&
P_0\otimes P_1\\
&\oplus\bigoplus_{\ell\geq 1}P_0\otimes C_{1,\ell}
\oplus\bigoplus_{\ell\geq 0}P_0\otimes D_{1,\ell}
\oplus\bigoplus_{k\geq 1}C_{0,k}\otimes P_1
\oplus\bigoplus_{k\geq 0}D_{0,k}\otimes P_1\\
&\oplus\bigoplus_{k,\ell\geq 1}C_{0,k}\otimes C_{1,\ell}
\oplus\bigoplus_{\substack{k\geq 1\\ \ell\geq 0}}C_{0,k}\otimes D_{1,\ell}
\oplus\bigoplus_{\substack{k\geq 0\\ \ell\geq 1}}D_{0,k}\otimes C_{1,\ell}
\oplus\bigoplus_{k,\ell\geq 0}D_{0,k}\otimes D_{1,\ell}.
\end{align*}
The first summand is the copy of $\Sigma^2P$ generated by $x_0x_1$. The four summand families in the second
line are copies of suspended $M_B$. The four summand families in the third line are summed bow blocks in the
monomial tensor-product basis.

Explicitly, the four $M_B$ families are
\begin{align*}
x_0 z_1^{2\ell-1} &\xrightarrow{Sq^2} x_0 z_1^{2\ell} &&(\ell\geq 1),\\
x_0x_1 z_1^{2\ell+1} &\xrightarrow{Sq^2} x_0x_1 z_1^{2\ell+2} &&(\ell\geq 0),\\
z_0^{2k-1}x_1 &\xrightarrow{Sq^2} z_0^{2k}x_1 &&(k\geq 1),\\
x_0 z_0^{2k+1}x_1 &\xrightarrow{Sq^2} x_0 z_0^{2k+2}x_1 &&(k\geq 0).
\end{align*}
The four $M_{SB}$ families are
\begin{align*}
z_0^{2k-1}z_1^{2\ell-1}
&\xrightarrow{Sq^2}\{z_0^{2k}z_1^{2\ell-1},\,z_0^{2k-1}z_1^{2\ell}\}\\
&\xrightarrow{Sq^2}z_0^{2k}z_1^{2\ell} &&(k,\ell\geq 1),\\
z_0^{2k-1}x_1z_1^{2\ell+1}
&\xrightarrow{Sq^2}\{z_0^{2k}x_1z_1^{2\ell+1},\,z_0^{2k-1}x_1z_1^{2\ell+2}\}\\
&\xrightarrow{Sq^2}z_0^{2k}x_1z_1^{2\ell+2} &&(k\geq 1,\ell\geq 0),\\
x_0z_0^{2k+1}z_1^{2\ell-1}
&\xrightarrow{Sq^2}\{x_0z_0^{2k+2}z_1^{2\ell-1},\,x_0z_0^{2k+1}z_1^{2\ell}\}\\
&\xrightarrow{Sq^2}x_0z_0^{2k+2}z_1^{2\ell} &&(k\geq 0,\ell\geq 1),\\
x_0x_1z_0^{2k+1}z_1^{2\ell+1}
&\xrightarrow{Sq^2}\{x_0x_1z_0^{2k+2}z_1^{2\ell+1},\,x_0x_1z_0^{2k+1}z_1^{2\ell+2}\}\\
&\xrightarrow{Sq^2}x_0x_1z_0^{2k+2}z_1^{2\ell+2} &&(k,\ell\geq 0).
\end{align*}
\end{lemma}

\begin{proof}
The one-factor decomposition above is a direct sum of $Sq^2$--orbits, because an even power of $z_i$ is hit by
the preceding odd power and an odd power starts a unique two-cell orbit. Tensoring the two one-factor
decompositions gives the displayed direct sum of tensor products. It remains only to identify the tensor
products as the announced blocks.

Since $Sq^1=0$ on both factors, the Cartan formula reduces to
\[
Sq^2(a\otimes b)=Sq^2(a)\otimes b+a\otimes Sq^2(b).
\]
Thus $P\otimes P\cong P$, $P\otimes M_B\cong M_B\cong M_B\otimes P$, and $M_B\otimes M_B$ has the $M_{SB}$
pattern. For the last assertion, if $M_B=\F_2\{u,v\}$ and $M_B'=\F_2\{u',v'\}$ with
$Sq^2(u)=v$, $Sq^2(u')=v'$, define
\[
d_0=u\otimes u',\qquad d_{20}=v\otimes u',\qquad d_{21}=u\otimes v',\qquad d_4=v\otimes v'.
\]
Then Cartan gives
\[
Sq^2(d_0)=d_{20}+d_{21},\qquad Sq^2(d_{20})=Sq^2(d_{21})=d_4,\qquad Sq^2(d_4)=0.
\]
Substituting the explicit generators of the $C_{i,k}$ and $D_{i,\ell}$ chains gives the four bow families and
the four summed bow families listed in the statement. The supports of these families are disjoint subsets of
the monomial basis, and every monomial belongs to exactly one of them, according to whether
each factor is $x_i$, an odd-to-even $z_i$ chain, or an odd-to-even $x_i z_i$ chain. Hence the sum is direct
and exhaustive.
\end{proof}

\subsection{Filtration-$0$ formulas for $c$ and $R$ on generators}

\begin{lemma}[Filtration-$0$ descriptions of $c$ and $R$]
\label{lemma:filtration0-c-R}
The operations $Q_0=Sq^1$ and $Q_1=[Sq^2,Sq^1]$ act trivially on $M$, so $M$ is a trivial $E(1)$--module and
\[
\Ext^{0,*}_{E(1)}(M,\F_2)=\Hom_{\F_2}(M,\F_2).
\]
If $m^\vee$ denotes the dual of a monomial $m\in M$, then the connecting morphism
\[
R:\Ext^{0,*}_{E(1)}(M,\F_2)\longrightarrow
\Ext^{0,*-2}_{\mathcal A(1)}(M,\F_2)
\]
is given by
\[
R(\phi)=\phi\circ Sq^2.
\]
Equivalently,
\begin{align*}
R\big(m(\epsilon_0,a_0;\epsilon_1,a_1)^\vee\big)
={}&
\begin{cases}
m(\epsilon_0,a_0-1;\epsilon_1,a_1)^\vee,& a_0>0\text{ even},\\
0,& a_0\text{ odd or }a_0=0,
\end{cases}\\
&+
\begin{cases}
m(\epsilon_0,a_0;\epsilon_1,a_1-1)^\vee,& a_1>0\text{ even},\\
0,& a_1\text{ odd or }a_1=0,
\end{cases}
\end{align*}
where any term that is not a monomial is omitted.

In the block notation of Lemma~\ref{lemma:A1decomp-smash-firstprinciples}, this says
\[
\begin{array}{ll}
\text{on }M_B=\langle u\to v\rangle:& R(v^\vee)=u^\vee,\quad R(u^\vee)=0,\\[2pt]
\text{on }M_{SB}=\langle d_0;\,d_{20},d_{21};\,d_4\rangle:&
R(d_{20}^\vee)=d_0^\vee,\quad R(d_{21}^\vee)=d_0^\vee,\\
&R(d_4^\vee)=d_{20}^\vee+d_{21}^\vee,\quad R(d_0^\vee)=0,\\[2pt]
\text{on }\Sigma^2P\langle x_0x_1\rangle:& R\big((x_0x_1)^\vee\big)=0.
\end{array}
\]

The map
\[
c:\Ext^{0,*}_{\mathcal A(1)}(M,\F_2)\longrightarrow \Ext^{0,*}_{E(1)}(M,\F_2)
\]
is the inclusion of the $\mathcal A(1)$--linear functionals into all $\F_2$--linear functionals. Since $Sq^1=0$,
a functional is $\mathcal A(1)$--linear if and only if it vanishes on $\operatorname{im}(Sq^2)$. Thus
\[
\begin{array}{ll}
\text{on }M_B=\langle u\to v\rangle:&
\Hom_{\mathcal A(1)}(M_B,\F_2)=\F_2\{u^\vee\},\quad c(u^\vee)=u^\vee,\\[2pt]
\text{on }M_{SB}:&
\Hom_{\mathcal A(1)}(M_{SB},\F_2)=
\F_2\{d_0^\vee,\,d_{20}^\vee+d_{21}^\vee\},\\
&c(d_0^\vee)=d_0^\vee,\quad
c(d_{20}^\vee+d_{21}^\vee)=d_{20}^\vee+d_{21}^\vee,\\[2pt]
\text{on }\Sigma^2P\langle x_0x_1\rangle:&
\Hom_{\mathcal A(1)}(\Sigma^2P,\F_2)=\F_2\{(x_0x_1)^\vee\},\\
&c\big((x_0x_1)^\vee\big)=(x_0x_1)^\vee.
\end{array}
\]
\end{lemma}

\begin{proof}
The equality $Q_0=0$ is $Sq^1=0$. The equality $Q_1=0$ follows because
$Q_1=Sq^2Sq^1+Sq^1Sq^2$, and $Sq^1$ vanishes on all of $M$. Hence
$\Ext^{0,*}_{E(1)}(M,\F_2)=\Hom_{E(1)}(M,\F_2)=\Hom_{\F_2}(M,\F_2)$.

The bow module $M_B\cong\mathcal A(1)//E(1)$ fits into the short exact sequence
\[
0\to \Sigma^2\F_2 \xrightarrow{\ \iota\ } M_B \xrightarrow{\ p\ } \F_2\to 0,
\qquad \iota(1)=v.
\]
After tensoring with $M$, the connecting map in the resulting long exact Ext sequence is the map denoted $R$.
Let $\phi\in\Hom_{\F_2}(M,\F_2)$ and let
\[
f_\phi:M_B\otimes M\longrightarrow \F_2
\]
be the corresponding $\mathcal A(1)$--linear map under the adjunction
\[
\Hom_{\mathcal A(1)}(M_B\otimes M,\F_2)\cong \Hom_{E(1)}(M,\F_2),
\]
normalized by $f_\phi(u\otimes m)=\phi(m)$, where $u$ is the lower generator of $M_B$. Since
$Sq^2(u\otimes m)=v\otimes m+u\otimes Sq^2m$ and $Sq^2$ acts trivially on $\F_2$, we get
\[
0=f_\phi(Sq^2(u\otimes m))=f_\phi(v\otimes m)+\phi(Sq^2m).
\]
Over $\F_2$ this gives $f_\phi(v\otimes m)=\phi(Sq^2m)$. Since $R$ is induced by the inclusion
$\iota\otimes 1:\Sigma^2M\to M_B\otimes M$, this is precisely
\[
(R\phi)(m)=\phi(Sq^2m).
\]
The monomial formula and the displayed formulas on $P$, $M_B$, and $M_{SB}$ are obtained by dualizing
\eqref{eq:Sq2-smash-monomial}.

Finally,
\[
\Ext^{0,*}_{\mathcal A(1)}(M,\F_2)=\Hom_{\mathcal A(1)}(M,\F_2)
\]
consists of those linear functionals that kill $Sq^1M$ and $Sq^2M$. Since $Sq^1M=0$, this is exactly the
annihilator of $\operatorname{im}(Sq^2)$. For a bow, $\operatorname{im}(Sq^2)=\F_2\{v\}$, so only $u^\vee$
survives. For an $M_{SB}$ block,
\[
\operatorname{im}(Sq^2)=\F_2\{d_{20}+d_{21},d_4\},
\]
so a functional must vanish on $d_4$ and take the same value on $d_{20}$ and $d_{21}$; this gives the basis
$d_0^\vee,d_{20}^\vee+d_{21}^\vee$. The one-cell case is immediate. The map $c$ is restriction of scalars
from $\mathcal A(1)$ to $E(1)$, hence the evident inclusion in filtration $0$.
\end{proof}

\section{The cyclic factor}\label{sec:cyclic}
\subsection[The eta-c-R sequence for BZ/4]{The $\eta$--$c$--$R$ sequence for $B\mathbb{Z}/4$}

Let
\[
M=H^*(B\mathbb Z/4;\mathbb F_2).
\]
As above, a monomial written in an Ext group denotes the dual of the corresponding cohomology monomial. The
following lemma records the part of the $\eta$--$c$--$R$ sequence used in the differential calculation.

\begin{lemma}[Basic $c$ and $R$ formulas for $B\mathbb{Z}/4$]
\label{lem:basic-c-r-bz4}
Consider the long exact sequence of Lemma~\ref{lem:algebraic-ses-eta-c-r} for this module $M$.
\begin{enumerate}
\item[(i)] On $\operatorname{Ext}^0$, the map $c$ is the inclusion
\[
\Hom_{\mathcal A(1)}(M,\mathbb F_2)\hookrightarrow \Hom_{E(1)}(M,\mathbb F_2).
\]
Its source has basis
\[
x,\ z^{2k+1}\ (k\ge 0),\ xz^{2k+1}\ (k\ge 0),
\]
and $c$ sends each of these classes to the class with the same name in
$\Ext^0_{E(1)}(M,\mathbb F_2)$.

\item[(ii)] On $\operatorname{Ext}^0_{E(1)}$ classes,
\[
R(z^{2k+1})=0,\qquad R(z^{2k+2})=z^{2k+1},
\]
and
\[
R(xz^{2k+1})=0,\qquad R(xz^{2k+2})=xz^{2k+1}.
\]

\item[(iii)] The maps are compatible with the coefficient formulas in Lemma~\ref{lem:coeff-c-r}. Thus, for every
filtration-$0$ class $y$ in the image of $c$,
\[
c(\omega y)=h_0v^2c(y),\qquad c(\mu y)=v^4c(y).
\]

\item[(iv)] Let $a,b\in M$ span a two-cell $Sq^2$--chain, so $Sq^2a=b$ and $Sq^2b=0$. In the corresponding
$E(1)$--Ext group,
\[
R(va)=R(h_0b)=h_0a\neq 0,\qquad
\ker(R)=\mathbb F_2\{va+h_0b\}
\]
in that bidegree. Hence exactness gives a real class $\alpha_a$, unique modulo $\operatorname{im}(h_1)$, with
\[
c(\alpha_a)=va+h_0b.
\]
In particular this applies to the chains
\[
z^{2k+1}\xrightarrow{Sq^2}z^{2k+2},
\qquad
xz^{2k+1}\xrightarrow{Sq^2}xz^{2k+2}.
\]
Multiplication by $\mu^m$ gives the corresponding $v^{4m}$--periodic lifts.
\end{enumerate}

No formula for arbitrary $v$--multiples is asserted here: $R$ is $h_0$--linear and natural, but it is not
$v$--linear. The $v^{4m}$--periodicity supplied by $\mu^m$ is the part used below.
\end{lemma}

\begin{proof}
For (i), the map $c$ is restriction of scalars along $E(1)\subset\mathcal A(1)$. Since $Sq^1=0$ on $M$, an
$\mathcal A(1)$--linear functional is precisely a functional vanishing on $\operatorname{im}(Sq^2)$. By
Corollary~\ref{cor:useful-cartan-consequences},
\[
\operatorname{im}(Sq^2)=\mathbb F_2\{z^{2k},xz^{2k}\mid k\geq 1\}.
\]
This gives the stated basis and the formula for $c$.

We next identify $R$ on filtration $0$. Let
\[
M_B=\mathcal A(1)//E(1)=\mathbb F_2\{u,w\},\qquad Sq^2u=w,\quad Sq^2w=0.
\]
The exact sequence
\[
0\to \Sigma^2\mathbb F_2\xrightarrow{\ \iota\ }M_B\xrightarrow{\ p\ }\mathbb F_2\to 0,\qquad \iota(1)=w,
\]
after tensoring with $M$, is the algebraic source of the map $R$. If
\(\phi\in\Hom_{E(1)}(M,\mathbb F_2)\) (equivalently, \(\phi\in\Hom_{\mathbb F_2}(M,\mathbb F_2)\)), let
$f_\phi:M_B\otimes M\to\mathbb F_2$ be the corresponding $\mathcal A(1)$--linear map, normalized by
$f_\phi(u\otimes m)=\phi(m)$. Since $Sq^1=0$ on $M_B$ and on $M$,
\[
Sq^2(u\otimes m)=w\otimes m+u\otimes Sq^2m.
\]
Applying $f_\phi$ gives
\[
0=f_\phi(Sq^2(u\otimes m))=f_\phi(w\otimes m)+\phi(Sq^2m),
\]
and hence
\[
(R\phi)(m)=f_\phi(w\otimes m)=\phi(Sq^2m).
\]
Dualizing the formulas of Corollary~\ref{cor:useful-cartan-consequences} gives (ii).

Part (iii) is Lemma~\ref{lem:coeff-c-r}, with the integer $2$ detected by $h_0$ on the Adams $E_2$--page.

For (iv), let $N=\langle a,b\rangle$ be the two-cell submodule. Then
$N\cong \Sigma^{|a|}M_B$. The quotient
\[
q:N\longrightarrow \Sigma^{|a|}\mathbb F_2,\qquad q(a)=1,\quad q(b)=0,
\]
is $\mathcal A(1)$--linear. Naturality gives
\[
\begin{array}{ccc}
\Ext^{1,*}_{E(1)}(\Sigma^{|a|}\mathbb F_2,\mathbb F_2) & \xrightarrow{R} &
\Ext^{1,*-2}_{\mathcal A(1)}(\Sigma^{|a|}\mathbb F_2,\mathbb F_2) \\
\downarrow q^* & & \downarrow q^* \\
\Ext^{1,*}_{E(1)}(N,\mathbb F_2) & \xrightarrow{R} &
\Ext^{1,*-2}_{\mathcal A(1)}(N,\mathbb F_2).
\end{array}
\]
The coefficient calculation gives $R(v)=h_0$, so
\[
R(va)=h_0a.
\]
The class $h_0a$ is nonzero: indeed, since $N\cong \Sigma^{|a|}\mathcal A(1)//E(1)$, change of rings identifies
$\Ext_{\mathcal A(1)}(N,\mathbb F_2)$ with the corresponding suspension of
$\Ext_{E(1)}(\mathbb F_2,\mathbb F_2)$.

On the other hand, the filtration-$0$ formula for $R$ gives $R(b)=a$, and $R$ is $h_0$--linear. Therefore
\[
R(h_0b)=h_0a.
\]
Thus
\[
R(va+h_0b)=0.
\]
In the source bidegree, $\Ext_{E(1)}(N,\mathbb F_2)$ has basis $\{va,h_0b\}$, because $E(1)$ acts trivially on
$N$. The target bidegree has basis $\{h_0a\}$ by the change-of-rings identification above. Since $R$ maps both
basis elements to the same nonzero class, its kernel is exactly the line generated by $va+h_0b$. Exactness now
identifies this line with $\operatorname{im}(c)$, and the indeterminacy in its real lift is
$\ker(c)=\operatorname{im}(h_1)$.
\end{proof}

\subsection[The d2 differential for the Adams spectral sequence on BZ/4]{The $d_2$ differential for the Adams spectral sequence on $B\mathbb{Z}/4$}

\begin{definition}[Hashimoto generators]\label{def:hashimoto}
For \(n\geq1\), let
\[
B_0^{(n)},\ldots,B_{n-1}^{(n)}
\]
denote Hashimoto's standard generators of
\(\widetilde{ku}_{2n-1}(B\mathbb Z/4)\).  We suppress the superscript when the
degree \(2n-1\) is fixed.
\end{definition}

We use the following \(C_4\)-presentation together with its natural
Gysin construction.  Let
\[
\mathcal B_n=\mathbb Z\langle B_0^{(n)},\ldots,B_{n-1}^{(n)}\rangle,\qquad
\mathcal R_n=\mathbb Z\langle R_0,\ldots,R_{n-1}\rangle ,
\]
and define \(\partial:\mathcal R_n\to\mathcal B_n\) by
\[
\partial(R_j)=4B_j^{(n)}+6B_{j-1}^{(n)}+4B_{j-2}^{(n)}+B_{j-3}^{(n)},
\]
with the convention that \(B_i^{(n)}=0\) for \(i<0\) or \(i\geq n\).
Hashimoto's presentation gives
\[
\widetilde{ku}_{2n-1}(B\mathbb Z/4)\cong \operatorname{coker}(\partial)
\]
in this basis; see \cite[Lemma~2.6]{hashimotogenerators}.  Its generators
come from the connective circle-bundle Gysin sequence
\cite[Proposition~2.1, Lemmas~2.2 and~2.5, and diagram~(2.4)]{hashimotogenerators}.
The general elementary-divisor basis is not needed for the Adams differential.

The indexing is chosen so that complex Bott multiplication preserves the
Hashimoto index:
\[
vB_i^{(n)}=B_i^{(n+1)}.
\]
Thus multiplication by \(v\) preserves the Hashimoto index \(i\) and raises the homological degree by \(2\).
Equivalently, write $e_j=B_j^{(j+1)}\in\widetilde{ku}_{2j+1}(Y)$;
then $v^{n-1-j}e_j=B_j^{(n)}$ for $0\leq j<n$.

We use this presentation to fix the complex Adams \(d_2^{ku}\) on \(B\mathbb Z/4\).

\begin{theorem}\label{theo:differentials.adams.complex.summand}
The only nonzero differential in the Adams spectral sequence for \(\widetilde{ku}_*(B\mathbb Z/4)\) is \(d_2^{ku}\).  It is given by
\[
d_2^{ku}(z^k)=h_0^2xz^{k-1}+h_0v\,xz^{k-2}\quad(k\geq2),
\]
and
\[
d_2^{ku}(z)=h_0^2x,\qquad d_2^{ku}(xz^k)=0\quad(k\geq0).
\]
\end{theorem}
\begin{proof}
Since $Q_0$ and $Q_1$ act trivially on
$\widetilde H^*(Y;\mathbb F_2)$, the complex Adams $E_2$--page is
the free $\mathbb F_2[h_0,v]$--module on
\[
z^k\quad(k\geq1),\qquad xz^j\quad(j\geq0).
\]
As above, the monomials denote the dual homology classes. The
coefficient classes $h_0$ and $v$ are permanent cycles of degrees
$(0,1)$ and $(2,1)$ in stem--filtration coordinates.

We first realize the generators $e_j$ by the Gysin boundary. Let
$\mathcal L\to\mathbb CP^\infty$ be the canonical line bundle, and
choose its complex-oriented Euler class $\xi$ so that
$v\xi=[\mathcal L]-1$ after periodization. The unit circle bundle of $\mathcal L^4$ is a model
for $Y$, with oriented Gysin boundary
\[
\delta:ku_{2j}(\mathbb CP^\infty)
\longrightarrow\widetilde{ku}_{2j+1}(Y).
\]
The even-cell filtration of $\mathbb CP^\infty$ gives a free
$ku_*$--basis $\beta_j$, $j\geq0$, Kronecker-dual to $\xi^j$, with
$\xi\cap\beta_j=\beta_{j-1}$. Here the source is unreduced, so
$\beta_0$ is the point class. We have $e_j=\delta\beta_j$ for the
coherent Hashimoto generators above, and set $\beta_j=e_j=0$ for
negative indices.

The map $ku\to H\mathbb F_2$ preserves the complex Thom orientation
and therefore commutes with the Gysin boundary. Since the ordinary
mod-two Euler class of $\mathcal L^4$ vanishes, the ordinary Gysin
sequence gives an isomorphism
\[
\delta:H_{2j}(\mathbb CP^\infty;\mathbb F_2)
\longrightarrow H_{2j+1}(Y;\mathbb F_2).
\]
Consequently the mod-two Hurewicz image of $e_j$ is dual to $xz^j$.
Thus $e_j$ is detected by $xz^j$ in filtration $0$, and $xz^j$ is a
permanent cycle. In particular, $d_2^{ku}(xz^j)=0$.

In the chosen orientation, the Euler class of $\mathcal L^4$ is
\[
[4](\xi)=4\xi+6v\xi^2+4v^2\xi^3+v^3\xi^4.
\]
By Gysin exactness, $\delta([4](\xi)\cap\beta_{j+1})=0$, so
\[
4e_j+6ve_{j-1}+4v^2e_{j-2}+v^3e_{j-3}=0
\qquad(j\geq0).
\]
This is the homogeneous form of Hashimoto's presentation
\cite[Lemmas~2.5--2.6]{hashimotogenerators}.

Let $F^s\widetilde{ku}_n(Y)$ denote the decreasing mod-two Adams
filtration, with associated graded $E_\infty^{n,s}$. Multiplication
by $2$ and by $v$ raises the filtration by at least one. Hence the
preceding relation gives
\[
4e_j+2ve_{j-1}
=-4ve_{j-1}-4v^2e_{j-2}-v^3e_{j-3}
\in F^3\widetilde{ku}_{2j+1}(Y).
\]
Consider the nonzero class
\[
y_j=h_0^2xz^j+h_0v\,xz^{j-1}\in E_2^{2j+1,2},
\]
where the second term is omitted for $j=0$. Since $h_0$, $v$, and
the classes $xz^j$ are permanent cycles, $y_j$ cannot support a
differential on any page on which it remains nonzero. If $y_j$
survived to $E_\infty$, strong convergence and compatibility with
coefficient products would identify it with the class of
$4e_j+2ve_{j-1}$ in $F^2/F^3$, contradicting the relation above.
Thus $y_j$ must be hit by an incoming differential.

An incoming $d_r$ to bidegree $(2j+1,2)$ has source
$(2j+2,2-r)$. This filtration is negative for $r\geq3$. For $r=2$,
the source is the one-dimensional space generated by $z^{j+1}$.
Therefore
\[
d_2^{ku}(z^{j+1})=h_0^2xz^j+h_0v\,xz^{j-1},
\]
with $d_2^{ku}(z)=h_0^2x$ when $j=0$.

Finally, the differential is $\mathbb F_2[h_0,v]$--linear and is
injective on the even-stem part. Indeed, if
$\sum_{k\geq1}f_k(h_0,v)z^k$ is a nonzero finite sum and $K$ is the
largest index with $f_K\ne0$, then the coefficient of $xz^{K-1}$
in its differential is $h_0^2f_K\ne0$. Hence $E_3$ is concentrated
in odd stems. Every higher Adams differential lowers the stem by
one, so all higher differentials vanish and $E_3=E_\infty$.
\end{proof}

\subsubsection*{Real connective case: $ko$}

We determine the \(d_2^{ko}\) differential in the Adams spectral sequence for
\(\widetilde{ko}_{*}(B\mathbb{Z}/4)\) using the \(\eta\)-\(c\)-\(R\) exact sequence. We first
record the setup used in the proof of Theorem~\ref{theo:differentials.adams.real.summand}.

Fix
\[
Y:=B\mathbb{Z}/4,\qquad M:=H^*(Y;\mathbb{F}_2).
\]

\begin{lemma}[Naturality and $E_2$ exactness for $\eta$-$c$-$R$]
\label{lem:naturality-e2-exactness-eta-c-r}
The maps in the cofiber sequence
\[
\Sigma ko \xrightarrow{\ \eta\ } ko \xrightarrow{\ c\ } ku \xrightarrow{\ R\ } \Sigma^2 ko
\]
induce morphisms of Adams spectral sequences at the prime \(2\)
\[
c_r:E_r^{s,t}(ko\wedge Y)\to E_r^{s,t}(ku\wedge Y),\qquad
R_r:E_r^{s,t}(ku\wedge Y)\to E_r^{s,t-2}(ko\wedge Y),
\]
and these commute with all differentials:
\[
c_r\circ d_r^{ko}=d_r^{ku}\circ c_r,\qquad R_r\circ d_r^{ku}=d_r^{ko}\circ R_r.
\]
At the $E_2$--page there is a long exact sequence
\[
\cdots \to E_2^{s,t}(ko\wedge Y)\xrightarrow{c_2}E_2^{s,t}(ku\wedge Y)
\xrightarrow{R_2}E_2^{s,t-2}(ko\wedge Y)\xrightarrow{h_1}
E_2^{s+1,t}(ko\wedge Y)\to\cdots .
\]
\end{lemma}

\begin{proof}
Naturality of the Adams construction gives maps of Adams resolutions, hence maps of exact couples and therefore
maps of spectral sequences commuting with all $d_r$. Applying this to $c$ and $R$ yields the displayed identities.
At $r=2$ the statement is exactly the long exact Ext sequence from
Lemma~\ref{lem:algebraic-ses-eta-c-r}. We will not use, and in general do not have, exactness after passing to
later pages.
\end{proof}

\begin{definition}[Single-factor real generators and their \(c\)--images]
\label{def:real-bz4-generators-c-images}
On the summand for \(Y=B\mathbb Z/4\), use the coefficient classes
\[
\mu^m x,\qquad \mu^m\omega x\qquad(m\geq0),
\]
with
\[
c(\mu^mx)=v^{4m}x,\qquad c(\mu^m\omega x)=h_0v^{4m+2}x.
\]
For each \(k\geq0\), choose the odd lifts
\[
vz^{2k+1},\ vxz^{2k+1}\in \Ext_{\mathcal{A}(1)}^{1,*}(M,\mathbb{F}_2)
\]
characterized by
\[
c\!\left(vz^{2k+1}\right)=v z^{2k+1}+h_0z^{2k+2},\qquad
c\!\left(vxz^{2k+1}\right)=v xz^{2k+1}+h_0xz^{2k+2}.
\]
More generally, for every two-cell bow
\[
a\xrightarrow{Sq^2}b,\qquad
a\in\{z^{2k+1},\,xz^{2k+1}\},\quad b=Sq^2a,
\]
write
\[
v^{2q}a,\qquad v^{2q+1}a
\]
for the real generators characterized by
\[
c(v^{2q}a)=v^{2q}a,\qquad
c(v^{2q+1}a)=v^{2q+1}a+h_0v^{2q}b.
\]
Lemma~\ref{lem:basic-c-r-bz4}\textup{(iv)} gives the initial odd lift.
The formulas for all exponents follow from the bow realification
calculation in Section~\ref{section:real}, where generators $A_e$ are
constructed for every $e\geq0$ and their complexifications are computed.
They are unique modulo the \(h_1\)-image.
On the \(ku\)-side \(v\) is the Bott class; on the \(ko\)-side \(v^ia\) denotes these chosen real generators.
\end{definition}

\begin{definition}[The single-factor lifting operator]
\label{def:real-bz4-lifting-operator}
Let \(Y=B\mathbb Z/4\). Define
\[
\mathcal L_Y:
\operatorname{im}\!\left(c:E_2^{*,*}(ko\wedge Y)\longrightarrow E_2^{*,*}(ku\wedge Y)\right)
\longrightarrow E_2^{*,*}(ko\wedge Y)
\]
to be the \(\F_2[h_0]\)--linear section of \(c\) determined by
Definition~\ref{def:real-bz4-generators-c-images}.  Explicitly,
\[
\mathcal L_Y(v^{4m}x)=\mu^m x,\qquad
\mathcal L_Y(h_0v^{4m+2}x)=\mu^m\omega x.
\]
For each two-cell bow
\[
a\xrightarrow{Sq^2}b,
\qquad
a\in\{z^{2k+1},\,xz^{2k+1}\},\quad b=Sq^2a,
\]
set
\[
\mathcal L_Y(v^{2q}a)=v^{2q}a,\qquad
\mathcal L_Y(v^{2q+1}a+h_0v^{2q}b)=v^{2q+1}a.
\]
The block decomposition of \(M=H^*(Y;\F_2)\) makes the prescription unambiguous.
\end{definition}

\begin{theorem}[The complete rule for the cyclic real second differential]
\label{theo:differentials.adams.real.summand}
The differential is zero on the one-cell coefficient summand.
On a bow $a\xrightarrow{Sq^2}b$, write $A_e$ for its normalized
real lift, characterized by
\[
 c(A_e)=v^ea+(e\bmod2)h_0v^{e-1}b.
\]
Then $R_2(v^eb)=A_e$ and the actual differential, in every residue,
is
\[
 d_2^{ko}(h_0^jA_e)=h_0^jR_2d_2^{ku}(v^eb).
\]
Here $R_2$ is the actual map on the algebraic Wood sequence.
Equivalently $d_2^{ko}=\mathcal L_Yd_2^{ku}c_2$.
In particular the initial values include
\[
\begin{aligned}
d_2^{ko}(z)&=h_0^2x,\\
d_2^{ko}(z^{2k+1})&=h_0v\,xz^{2k-1} &&(k>0),\\
d_2^{ko}(vz)&=h_0^3xz,\\
d_2^{ko}(vz^{2k+1})&=h_0^3xz^{2k+1}+h_0v^2xz^{2k-1} &&(k>0),\\
d_2^{ko}(vxz^{2k+1})&=0 &&(k\geq0).
\end{aligned}
\]
These initial values are not a claim that $d_2$ is linear over a
real coefficient $v$; no such coefficient exists.  For example,
the full rule also gives $d_2^{ko}(v^2z)=h_0\omega x$.
\end{theorem}
\begin{proof}
A bow is induced from the trivial $E(1)$ module, so its real Ext
group is annihilated by $h_1$.  Exactness of the algebraic Wood
sequence makes $R_2$ onto the bow, with the displayed preimage
formula.  Naturality of the actual spectral-sequence maps gives
$d_2^{ko}R_2=R_2d_2^{ku}$, proving the rule without any uniqueness
assumption about a target bidegree.

Every complex $d_2$ entry is $h_0$-divisible.  Realification is
$h_0$-linear, so every actual real $d_2$ image is $h_0$-divisible.
The kernel of $c_2$ is supported on the one-cell $\eta$ and
$\eta^2$ coefficient strings, none of which is $h_0$-divisible.
Thus the chosen $c$-section gives the same answer.  This is also
the argument, with the explicit residue tables, given independently
for the smash module in Theorem~\ref{theo:real-smash-d2-transfer};
that proof uses only the complex differential and the algebraic
Wood maps, not this cyclic real theorem.

The one-cell class $x$ is represented by the mapped circle.
It and its actual coefficient multiples cannot support an outgoing
differential, by the Adams product pairing.  Substituting the complex
$d_2$ formulas and the bow $R_2$ table gives the displayed values.
The independent complete cyclic group calculation used later is
\cite[Theorem~2.2.40]{siegemeyer}; no incomplete list of real
$v$-residues is used to deduce those group orders.
\end{proof}

\begin{figure}[htbp]
    \centering
    \includegraphics[width=\textwidth]{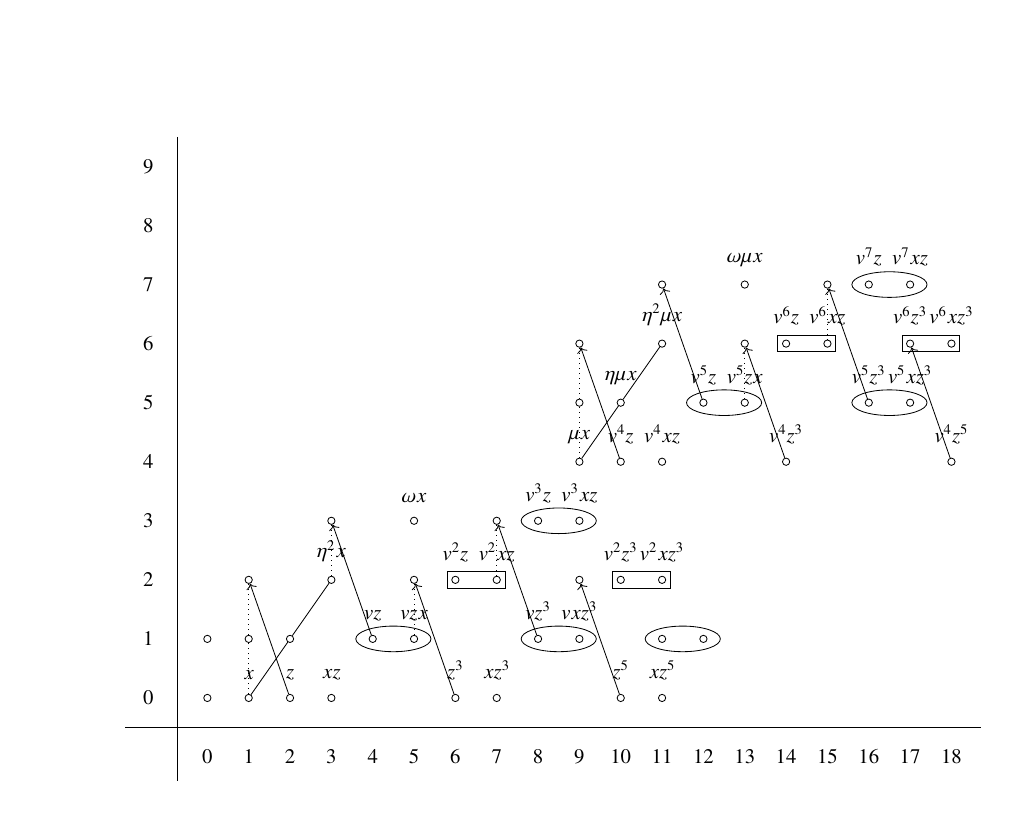}
    \caption{The Adams $E_2$-page for
    $\widetilde{ko}_*(BC_4)$ in stem--filtration
    coordinates $(t-s,s)$. The labels $v^e z^{2k+1}$
    and $v^e xz^{2k+1}$ denote the normalized real
    Ext generators of Definition~4.5, rather than
    multiplication by a real Bott element.
    The indicated differentials are $d_2$-differentials.}
    \label{fig:adams-cyclic-ko}
\end{figure}

\section{Connective K-theory of the smash summand}\label{sec:smash}
\label{section:complex}
\subsection[The complex d2 on BZ/4 smash BZ/4]{The complex $d_2$ on $B\mathbb{Z}/4\wedge B\mathbb{Z}/4$}

Let $x_i,z_i$ denote the copies of the one-factor classes $x,z$ in the $i$th smash factor, $i=0,1$.
Throughout the next two sections, a cohomology monomial written in an Ext group denotes its dual class.
Both factors are based, and put $z_i^m=0$ for $m<0$. Thus the filtration-$0$ monomial basis on the
$ku$--side consists of the following four families, with the indicated ranges:
\[
\begin{array}{c|c}
\text{class} & \text{range}\\ \hline
z_0^kz_1^\ell & k,\ell\geq 1\\
x_0z_0^kz_1^\ell & k\geq 0,\ \ell\geq 1\\
z_0^kx_1z_1^\ell & k\geq 1,\ \ell\geq 0\\
x_0z_0^kx_1z_1^\ell & k,\ell\geq 0.
\end{array}
\]

The one-factor complex differential from Theorem~\ref{theo:differentials.adams.complex.summand} may be written
in the $i$th factor as
\[
d_{2,i}^{ku}(z_i^m)=h_0^2x_iz_i^{m-1}+h_0v\,x_iz_i^{m-2},\qquad
d_{2,i}^{ku}(x_iz_i^m)=0.
\]
The smash-product pairing of Adams spectral sequences induces the external product on the $E_2$--page and is
compatible with differentials. Hence the differential satisfies the Leibniz rule, with no sign over $\F_2$; see
May~\cite[Theorem~4.4]{may_hinfty} and Bousfield--Kan~\cite{bousfieldkanpairings}.

\begin{theorem}\label{theo:diff.adams.ku.smash.low}
In the Adams spectral sequence for \(\widetilde{ku}_*(X)\), the \(d_2^{ku}\) differential on
filtration-$0$ product generators is given by
\begin{align*}
d_2^{ku}(z_0^kz_1^\ell)
&=
\left(h_0^2x_0z_0^{k-1}+h_0v x_0z_0^{k-2}\right)z_1^\ell
 +z_0^k\left(h_0^2x_1z_1^{\ell-1}+h_0v x_1z_1^{\ell-2}\right),\\
d_2^{ku}(x_0z_0^kz_1^\ell)
&=
x_0z_0^k\left(h_0^2x_1z_1^{\ell-1}+h_0v x_1z_1^{\ell-2}\right),\\
d_2^{ku}(z_0^kx_1z_1^\ell)
&=
\left(h_0^2x_0z_0^{k-1}+h_0v x_0z_0^{k-2}\right)x_1z_1^\ell,\\
d_2^{ku}(x_0z_0^kx_1z_1^\ell)
&=0.
\end{align*}
By $\mathbb{F}_2[h_0,v]$--linearity, these formulas determine the \(d_2^{ku}\) differential on all
$h_0$-- and $v$--multiples of these generators.
\end{theorem}

\begin{proof}
For product classes $a_0b_1$ the Leibniz rule gives
\[
d_2^{ku}(a_0b_1)=d_{2,0}^{ku}(a_0)b_1+a_0d_{2,1}^{ku}(b_1).
\]
Substituting the one-factor formulas for \(d_{2,0}^{ku}\) and \(d_{2,1}^{ku}\) gives the four displayed cases. The
coefficient classes $h_0$ and $v$ are permanent cycles on the coefficient $E_2$--page, so the formulas extend
$\mathbb F_2[h_0,v]$--linearly.
\end{proof}

\subsection[The real d2 differential on BZ/4 smash BZ/4]{The real $d_2$ differential on $B\mathbb{Z}/4\wedge B\mathbb{Z}/4$}
\label{section:real}

We now pass from the complex computation to the Adams spectral sequence for
\(\widetilde{ko}_*(X)\). In this section \(c\) and \(R\) denote the \(E_2\)--page maps
$c_2$ and $R_2$. The real $E_2$--term is organized by the $\mathcal A(1)$--decomposition of
Lemma~\ref{lemma:A1decomp-smash-firstprinciples}; in particular, on summed bow summands the natural
filtration-$0$ generators are not simply external products of one-factor generators. Thus the efficient
calculation is to use the complex differential
\(d_2^{ku}\) from Theorem~\ref{theo:diff.adams.ku.smash.low}, and lift \(d_2^{ku}(c(y))\) through the
\(\eta\)--\(c\)--\(R\) sequence.

\begin{definition}[Real generators and their $c$--images]
\label{def:real-smash-generators-c-images}
Let $p=x_0x_1$ be the generator of the one-cell summand. We use the coefficient classes
\[
\mu^m p,\qquad \mu^m\omega p\qquad (m\geq 0),
\]
with
\[
c(\mu^m p)=v^{4m}p,\qquad c(\mu^m\omega p)=h_0v^{4m+2}p.
\]

For a bow summand $M_B=\langle a,b\rangle$, $Sq^2(a)=b$, write
\[
v^{2q}a,\qquad v^{2q+1}a
\]
for the real generators characterized by
\begin{equation}\label{eq:c-image-real-bow}
c(v^{2q}a)=v^{2q}a,\qquad
c(v^{2q+1}a)=v^{2q+1}a+h_0v^{2q}b.
\end{equation}
The symbol $v^i$ on the left is notation for a real $\mathcal A(1)$--Ext generator, not $ku$--Bott multiplication.

For a summed bow block
\[
M_{SB}=\langle d_0;\,d_{20},d_{21};\,d_4\rangle
\]
put $e=d_{20}+d_{21}$. We use the generators
\[
v^{2q}d_0,\quad v^{2q+1}_0d_0,\quad v^{2q+1}_1d_0,\quad
v^{2q}e,\quad v^{2q+1}e,
\]
where
\begin{align}
c(v^{2q}d_0)&=v^{2q}d_0,\label{eq:c-image-msb-bottom-even}\\
c(v^{2q+1}_0d_0)&=v^{2q+1}d_0+h_0v^{2q}d_{20},\label{eq:c-image-msb-bottom-zero}\\
c(v^{2q+1}_1d_0)&=v^{2q+1}d_0+h_0v^{2q}d_{21},\label{eq:c-image-msb-bottom-one}\\
c(v^{2q}e)&=v^{2q}(d_{20}+d_{21}),\label{eq:c-image-msb-middle-even}\\
c(v^{2q+1}e)&=v^{2q+1}(d_{20}+d_{21})+h_0v^{2q}d_4.\label{eq:c-image-msb-middle-odd}
\end{align}
The two odd lifts of $d_0$ differ by the middle generator:
\[
v^{2q+1}_1d_0=v^{2q+1}_0d_0+h_0v^{2q}e.
\]
These formulas are extended $h_0$--linearly. Their existence for every
exponent follows from the
\hyperref[par:explicit-realification-table]{explicit realification table below},
where the maps $R$ and $c$ on bow and summed-bow summands are computed
from the coefficient table, naturality, and exactness of the algebraic
$\eta$--$c$--$R$ sequence. This is the standard computation on these
summands in the minimal $\mathcal A(1)$--resolution
\cite{brunergreenlees,siegemeyer}.
\end{definition}

\begin{definition}[The smash-product lifting operator]
\label{def:real-smash-lifting-operator}
Let $X=B\mathbb Z/4\wedge B\mathbb Z/4$. Define
\[
\mathcal L_X:
\operatorname{im}\!\left(c:E_2^{*,*}(ko\wedge X)\longrightarrow E_2^{*,*}(ku\wedge X)\right)
\longrightarrow E_2^{*,*}(ko\wedge X)
\]
to be the $\F_2[h_0]$--linear section of $c$ determined by the displayed $c$--images, block by block. Explicitly,
\[
\mathcal L_X(v^{4m}p)=\mu^mp,\qquad
\mathcal L_X(h_0v^{4m+2}p)=\mu^m\omega p,
\]
\[
\mathcal L_X(v^{2q}a)=v^{2q}a,\qquad
\mathcal L_X(v^{2q+1}a+h_0v^{2q}b)=v^{2q+1}a
\]
on a bow $a\xrightarrow{Sq^2}b$, and
\begin{align*}
\mathcal L_X(v^{2q}d_0)&=v^{2q}d_0,\\
\mathcal L_X(v^{2q+1}d_0+h_0v^{2q}d_{20})&=v^{2q+1}_0d_0,\\
\mathcal L_X(v^{2q+1}d_0+h_0v^{2q}d_{21})&=v^{2q+1}_1d_0,\\
\mathcal L_X\bigl(v^{2q}(d_{20}+d_{21})\bigr)&=v^{2q}e,\\
\mathcal L_X\bigl(v^{2q+1}(d_{20}+d_{21})+h_0v^{2q}d_4\bigr)&=v^{2q+1}e.
\end{align*}
The direct-sum decomposition of Lemma~\ref{lemma:A1decomp-smash-firstprinciples} makes this prescription
unambiguous.
\end{definition}

\paragraph{Realification lifts and the second differential.}

In this subsection all maps are on the Adams $E_2$ page.  Let
$M=\widetilde H^*(BC_4\wedge BC_4;\mathbb F_2)$, and write
$K(M)=\operatorname{Ext}_{\mathcal A(1)}(M,\mathbb F_2)$ and
$C(M)=\operatorname{Ext}_{E(1)}(M,\mathbb F_2)$.  The maps $c,R,h_1$
are those of the algebraic Wood long exact sequence.  In particular,
$R$ lowers the internal degree by two and does not change filtration.

\begin{lemma}[Realification lifts on the bow part]
Let $M'=\Sigma^d M_B$ or $\Sigma^d M_{SB}$ be one of the
$\mathcal A(1)$-summands of $M$.  Multiplication by $h_1$ is zero on
$K(M')$, and $R:C(M')\longrightarrow K(M')[-2]$ is surjective.
\end{lemma}
\begin{proof}
For $M_B=\mathcal A(1)\otimes_{E(1)}\mathbb F_2$, change of rings
identifies its Ext group with
$\mathbb F_2[h_0,v]$.  The action of coefficient Ext classes under
this identification is restriction from $\mathcal A(1)$ to $E(1)$.
The restriction of $h_1\in\operatorname{Ext}^{1,2}_{\mathcal A(1)}
(\mathbb F_2,\mathbb F_2)$ is zero: the degree $(1,2)$ part of
$\operatorname{Ext}_{E(1)}(\mathbb F_2,\mathbb F_2)$ is zero, since
its degree-one generators have internal degrees one and three.
Thus $h_1 K(M_B)=0$.  Suspension preserves this conclusion, and
$M_{SB}\cong M_B\oplus\Sigma^2M_B$ gives the other case.
Exactness of the Wood sequence says $\operatorname{im}R=\ker h_1$,
which is all of $K(M')$.
\end{proof}

\begin{proposition}[An ambiguity-free rule for the real second differential]
Let $y\in K(M')$ belong to a bow or summed-bow summand.  Choose any
$w\in C(M')$ such that $R(w)=y$.  Then
\[
                 d_2^{ko}(y)=R\bigl(d_2^{ku}(w)\bigr).
\]
The right side is independent of the choice of $w$.  All $d_2$'s on the
one-cell coefficient summand generated by $p=x_0x_1$ are zero.
Consequently these rules determine the actual real second differential
once the complex differential and the actual $E_2$ map $R$ are given.
\end{proposition}
\begin{proof}
Existence of $w$ is the preceding lemma.  Naturality of the map of
spectra $ku\longrightarrow\Sigma^2ko$ gives
$R d_2^{ku}(w)=d_2^{ko}R(w)=d_2^{ko}(y)$.
For completeness, if $R(w-w')=0$, exactness gives $w-w'=c(z)$ for
some real class $z$.  Then
\[
 R d_2^{ku}(w-w')=R d_2^{ku}c(z)
 =Rc\,d_2^{ko}(z)=0.
\]
This uses exactness only on $E_2$, not on later pages.

The class $p$ is the external product of the two classes $x$ in the
one-factor real spectral sequences.  Each $x$ has zero second
differential for degree reasons.  The external-product Leibniz rule
therefore gives $d_2(p)=0$.  The one-cell Ext summand is generated by
this product over coefficient Ext.  The coefficient spectral sequence
for $ko$ has zero differentials.  A second application of the Leibniz
rule gives the assertion for every coefficient multiple of $p$.
\end{proof}

\begin{remark}[What has and has not been replaced]
This proposition replaces the false assertion that all target
bidegrees have trivial $\ker c$.  It is a computational rule, not a
license to use an arbitrary section of $c$: its input is the actual
$E_2$ realification map, calculated from the short exact sequence of
Steenrod modules.  Any displayed table of real differentials must be
checked against this rule.  In particular, no general $v$-linearity of
$R$ is asserted.
\end{remark}

\begin{example}[The first exceptional target bidegree]
Let $y=z_0x_1z_1$ in stem five and filtration zero.  Both factors are
one-factor real bottom generators, so the external-product formula
gives
\[
 d_2^{ko}(y)=d_2(z_0)\,x_1z_1+z_0\,d_2(x_1z_1)
           =h_0^2x_0x_1z_1.
\]
This proves that there is no extra $h_1^2p$ term in this particular
differential, even though $h_1^2p$ exists in the target bidegree.
One must not confuse the proved coefficient zero with the nonexistence
of the possible target class.
\end{example}

\begin{theorem}[The real second differential]\label{theo:real-smash-d2-transfer}
Choose the section $\mathcal L_X$ of $c$ whose image is the direct sum of
the bow and summed-bow Ext summands and the $c$-visible part of the
one-cell coefficient summand.  For every real $E_2$ class $y$,
\[
                  d_2^{ko}(y)=\mathcal L_X\bigl(d_2^{ku}(c(y))\bigr).
\]
This asserts that the eta coefficient of the actual differential is
zero; it does not assert that the target bidegree has no eta class.
\end{theorem}
\begin{proof}
The complex second differential has every matrix entry divisible by
$h_0$.  On a bow or summed-bow source, the preceding proposition and
$h_0$-linearity of $R$ therefore imply
\[
                        d_2^{ko}(y)\in h_0K(M).
\]
On the one-cell source the differential is zero.  Consequently every
actual real second differential belongs to $h_0K(M)$.

The $E_2$ kernel of $c$ is the image of $h_1$.  The bow summands have
zero such image.  On the one-cell summand its only classes are
$h_1\mu^mp$ and $h_1^2\mu^mp$, together with their allowed coefficient
multiples; the coefficient chart says $h_0h_1=0$ and shows that neither
class is itself $h_0$-divisible.  Thus $h_0K(M)$ has zero component in
this kernel in the fixed direct-sum decomposition.

Naturality determines the complexification of $d_2^{ko}(y)$ to be
$d_2^{ku}(c(y))$.  The actual differential and the chosen section lift
both have zero component in $\ker c$, so they are equal.  This also
checks the higher periodic bidegrees in which an eta class coexists
with a bow target.
\end{proof}

\paragraph{An explicit realification table for implementation.}
\phantomsection\label{par:explicit-realification-table}
For $M_B=\langle a,b\rangle$ with $Sq^2a=b$, let $A_e$ be the
real Ext generator in stem $|a|+2e$ and filtration $e$, normalized by
\[
 R(v^eb)=A_e.
\]
Then, with $\varepsilon_e=e\bmod2$,
\[
 R(v^ea)=\varepsilon_e h_0A_{e-1},\qquad
 c(A_e)=v^ea+\varepsilon_eh_0v^{e-1}b,
\]
where terms with index $-1$ are zero.  To see the first formula, the
quotient $M_B\to\mathbb F_2$ taking $a$ to one identifies its induced
Ext map with restriction to $E(1)$.  Apply naturality to the coefficient
table
\[
 R(v^{4m})=0,\quad R(v^{4m+1})=h_0\mu^m,
 \quad R(v^{4m+2})=h_1^2\mu^m,\quad
 R(v^{4m+3})=\omega\mu^m.
\]
Restriction sends $h_1$ to zero, $\mu$ to $v^4$, and $\omega$ to
$h_0v^2$.  This proves the asserted parity formula.  Surjectivity and
the unique target generator prove $R(v^eb)=A_e$; the displayed kernel
line then proves the formula for $c$.

For the summed bow choose the actual cohomology splitting
\[
 \langle d_0,d_{20}+d_{21}\rangle
 \ \oplus\ \langle d_{21},d_4\rangle.
\]
Let $D_e,E_e$ be the corresponding real Ext generators, so that
$D_0=d_0^\vee$ and $E_0=d_{20}^\vee+d_{21}^\vee$.
Dualizing this splitting and applying the bow table gives
\[
\begin{array}{c|c}
\text{complex class}&R\text{ of that class}\\\hline
v^ed_0&\varepsilon_e h_0D_{e-1}\\
v^ed_{20}&D_e\\
v^ed_{21}&D_e+\varepsilon_e h_0E_{e-1}\\
v^ed_4&E_e.
\end{array}
\]
In particular an arbitrary power of $v$ cannot be pulled through $R$.
These tables and the coefficient table suffice to implement the
ambiguity-free rule above in every finite bidegree.

\subsubsection*{Filtration-$0$ form}

The transfer formula is the structural statement. On bottom generators it has the following
compact form.

Let $m=m_0m_1$ be a tensor product of one-factor bottom generators, where each $m_i$ is one of
\[
x_i,\qquad z_i^{2r-1}\ (r\geq 1),\qquad x_i z_i^{2r+1}\ (r\geq 0).
\]
Define $\Theta_i(m)$ as follows. If $m_i=x_i$ or $m_i=x_i z_i^{2r+1}$, then
\[
\Theta_i(m)=0.
\]
If $m_i=z_i$, then
\[
\Theta_i(m)=h_0^2\,x_i\,m_{\hat i},
\]
where $m_{\hat i}$ denotes the other factor. If $m_i=z_i^{2r-1}$ with $r>1$, then
\[
\Theta_i(m)=h_0\,v_i\bigl(x_i z_i^{2r-3}m_{\hat i}\bigr),
\]
where $v_i$ means the directional odd lift whose correction term is obtained by applying $Sq^2$ in the $i$th
factor. Then
\begin{equation}\label{eq:real-smash-d2-bottom-expanded}
d_2^{ko}(m)=\Theta_0(m)+\Theta_1(m).
\end{equation}
For example,
\[
d_2^{ko}(x_0z_1)=d_2^{ko}(z_0x_1)=h_0^2x_0x_1,
\]
and for $\ell>1$ and $k>1$,
\[
d_2^{ko}(x_0z_1^{2\ell-1})=h_0\,v(x_0x_1z_1^{2\ell-3}),\qquad
d_2^{ko}(z_0^{2k-1}x_1)=h_0\,v(x_0z_0^{2k-3}x_1).
\]
Also,
\[
d_2^{ko}(x_0x_1)=0,\qquad
d_2^{ko}(x_0x_1z_1^{2\ell+1})=0,\qquad
d_2^{ko}(x_0z_0^{2k+1}x_1)=0.
\]

For the middle generator and odd lifts in a summed bow block, the corresponding formulas are obtained directly
from the transfer formula:
\begin{equation}\label{eq:real-smash-d2-middle-expanded}
d_2^{ko}(e)=\mathcal L_X\bigl(d_2^{ku}(d_{20}+d_{21})\bigr).
\end{equation}
\begin{align}
d_2^{ko}(v^{2q+1}_0d_0)
&=\mathcal L_X\bigl(d_2^{ku}(v^{2q+1}d_0+h_0v^{2q}d_{20})\bigr),\label{eq:real-smash-d2-odd-bottom-zero}\\
d_2^{ko}(v^{2q+1}_1d_0)
&=\mathcal L_X\bigl(d_2^{ku}(v^{2q+1}d_0+h_0v^{2q}d_{21})\bigr),\label{eq:real-smash-d2-odd-bottom-one}\\
d_2^{ko}(v^{2q+1}e)
&=\mathcal L_X\bigl(d_2^{ku}(v^{2q+1}(d_{20}+d_{21})+h_0v^{2q}d_4)\bigr).\label{eq:real-smash-d2-odd-middle}
\end{align}
These are the substitutions of the $c$--images into Theorem~\ref{theo:real-smash-d2-transfer}.

\subsection[The complex d3 differential]{The complex $d_3$ differential}

Let
\[
X=B\mathbb Z/4\wedge B\mathbb Z/4,\qquad p=x_0x_1.
\]
In this section all Adams spectral sequences are for \(ku\wedge X\). We write the Adams bidegree as
\((n,s)\), where \(n=t-s\). Thus
\[
d_r:E_r^{n,s}\longrightarrow E_r^{n-1,s+r},
\]
and we write
\[
\Delta=d_3^{ku}:E_3^{n,s}(ku\wedge X)\longrightarrow E_3^{n-1,s+3}(ku\wedge X).
\]
The Adams spectral sequence pairing is compatible with external products and satisfies the Leibniz rule for
differentials; see May~\cite[Theorem~4.4]{may_hinfty} and Bousfield--Kan~\cite{bousfieldkanpairings}. Since
we work over \(\F_2\), no signs occur. In particular, the coefficient classes \(h_0\) and \(v\) act by
permanent cycles, and \(\Delta\) is \(h_0\)- and \(v\)-linear wherever the products are defined.

We isolate the finite calculation in the bidegrees where the initial complex $d_3$ would occur. For $a,b\geq 1$,
set
\begin{equation}\label{eq:def-gamma-ab}
\Gamma_{a,b}
=
h_0\bigl(x_0z_0^{a-1}z_1^b+z_0^ax_1z_1^{b-1}\bigr)
+v\bigl(x_0z_0^{a-2}z_1^b+z_0^ax_1z_1^{b-2}\bigr),
\end{equation}
where terms with a negative exponent are omitted. A direct substitution in
Theorem~\ref{theo:diff.adams.ku.smash.low} gives \(d_2^{ku}(\Gamma_{a,b})=0\). We also use the
\(d_2^{ku}\)--cycle
\[
\Phi_5=
\sum_{i=0}^3 x_0z_0^iz_1^{4-i}
+
\sum_{i=1}^4 z_0^ix_1z_1^{4-i}
\in E_2^{9,0}(ku\wedge X).
\]

\begin{lemma}[The source in bidegree $(11,1)$]\label{lem:ku-source-e3-11-1}
In bidegree $(11,1)$ the group $E_3^{11,1}(ku\wedge X)$ has basis
\[
\Gamma_{1,5},\ \Gamma_{2,4},\ \Gamma_{3,3},\ \Gamma_{4,2},\ \Gamma_{5,1},\ v\Phi_5
\]
and the target group $E_3^{10,4}(ku\wedge X)$ is one-dimensional, generated by $v^4p$.
Moreover $h_0v^4p=0$ in $E_3^{10,5}(ku\wedge X)$.
\end{lemma}

\begin{proof}
We spell out the source calculation. Put
\[
A_i=h_0x_0z_0^iz_1^{5-i},\qquad
B_i=h_0z_0^{i+1}x_1z_1^{4-i}\qquad(0\leq i\leq 4),
\]
and
\[
C_i=vx_0z_0^iz_1^{4-i},\qquad
D_i=vz_0^{i+1}x_1z_1^{3-i}\qquad(0\leq i\leq 3).
\]
These form the standard $E_2^{11,1}$ basis. In the target of \(d_2^{ku}\), write
\[
P_i=h_0^3pz_0^iz_1^{4-i}\quad(0\leq i\leq 4),\qquad
Q_i=h_0^2vpz_0^iz_1^{3-i}\quad(0\leq i\leq 3),
\]
\[
R_i=h_0v^2pz_0^iz_1^{2-i}\quad(0\leq i\leq 2).
\]
Terms outside these ranges are omitted. The formula of Theorem~\ref{theo:diff.adams.ku.smash.low} gives
\[
d_2^{ku}(A_i)=P_i+Q_i,\qquad d_2^{ku}(B_i)=P_i+Q_{i-1},
\]
\[
d_2^{ku}(C_i)=Q_i+R_i,\qquad d_2^{ku}(D_i)=Q_i+R_{i-1}.
\]
Let
\[
y=\sum_{i=0}^4 a_iA_i+\sum_{i=0}^4 b_iB_i+\sum_{i=0}^3 c_iC_i+\sum_{i=0}^3 d_iD_i.
\]
The equations \(d_2^{ku}(y)=0\) are
\[
a_i+b_i=0\quad(0\leq i\leq 4),
\]
\[
c_i+d_{i+1}=0\quad(0\leq i\leq 2),
\]
and
\[
a_i+a_{i+1}+c_i+d_i=0\quad(0\leq i\leq 3).
\]
Thus $b_i=a_i$,
\[
d_i=a_0+a_i+d_0\quad(1\leq i\leq 3),\qquad
c_i=a_0+a_{i+1}+d_0\quad(0\leq i\leq 3).
\]
The kernel is therefore six-dimensional, with free parameters
\[
a_0,a_1,a_2,a_3,a_4,d_0.
\]
The five classes $\Gamma_{1,5},\ldots,\Gamma_{5,1}$ and the class $v\Phi_5$ correspond respectively to the
following six choices of free coordinates:
\[
\begin{array}{c|cccccc}
&a_0&a_1&a_2&a_3&a_4&d_0\\ \hline
\Gamma_{1,5}&1&0&0&0&0&1\\
\Gamma_{2,4}&0&1&0&0&0&0\\
\Gamma_{3,3}&0&0&1&0&0&0\\
\Gamma_{4,2}&0&0&0&1&0&0\\
\Gamma_{5,1}&0&0&0&0&1&0\\
v\Phi_5&0&0&0&0&0&1
\end{array}
\]
These six vectors form a basis of the kernel. Since there are no incoming \(d_2^{ku}\)--boundaries into filtration \(1\),
this kernel is $E_3^{11,1}$.

The same \(d_2^{ku}\) row-reduction in bidegree $(10,4)$ gives
\[
E_3^{10,4}(ku\wedge X)=\F_2\{v^4p\}.
\]
In bidegree $(10,5)$ the class $h_0v^4p$ is explicitly a boundary:
\[
d_2^{ku}\bigl(v^3(z_0x_1z_1+x_0z_1^2)\bigr)=h_0v^4p.
\]
Hence $h_0v^4p=0$ on the $E_3$--page.
\end{proof}

% Genuine replacement for the primitive complex d3 identification.
% It does not use May--Milgram or a claimed unique AHSS representative.

\paragraph{The primitive differential from naturality.}

We use the complex $d_2$ calculation, and the following finite consequences
of it (all coordinates below are stem--filtration coordinates):
\[
 E_3^{10,4}(ku\wedge BG_+)=\mathbb F_2\{v^4p\},
 \qquad p=x_0x_1.
\]
The group $E_3^{11,1}$ has basis
\[
 g_1=\Gamma_{1,5},\ g_2=\Gamma_{2,4},\
 g_3=\Gamma_{3,3},\ g_4=\Gamma_{4,2},\
 g_5=\Gamma_{5,1},\ f=v\Phi_5,
\]
together with the four axis classes
\[
 a_0=h_0x_0z_0^5,\quad a_1=h_0x_1z_1^5,
 \quad b_0=vx_0z_0^4,\quad b_1=vx_1z_1^4.
\]
Finally $E_3^{11,0}$ is spanned by
$x_0z_0^5,x_1z_1^5,\Phi_6$, where
\[
 \Phi_6=\sum_{i=0}^4x_0z_0^iz_1^{5-i}
       +\sum_{i=1}^5z_0^ix_1z_1^{5-i}.
\]
These are statements about the full product, not a supposed action of every
group automorphism on an individually preserved smash summand.  They follow
by taking the kernel and image of the printed complex $d_2$ matrix; an exact
finite verifier is supplied with this replacement.

\begin{lemma}[Cyclic classes eliminate the possible fourth differential]
Every class in $E_3^{11,0}$ is a permanent cycle.  Moreover $d_3(f)=0$ and
$d_3$ is zero on $a_0,a_1,b_0,b_1$.
\end{lemma}
\begin{proof}
The axis classes come from the two factors.  The mixed class $\Phi_6$
is the smash part of the image of $xz^5$ under the diagonal
$BC_4\to BC_4\times BC_4$.  Similarly $\Phi_5$ is the mixed part of
the diagonal image of $xz^4$.  This is checked by pulling back the
monomial basis in ordinary mod-two cohomology: each displayed monomial
pulls back to the indicated $xz$ power.  These cyclic classes are
permanent cycles because the cyclic complex $E_3$ page is concentrated
in odd stems.  Naturality and multiplication by the permanent coefficient
classes $h_0,v$ prove all assertions.
\end{proof}

\begin{lemma}[The primitive third differential is nonzero]
The map $d_3:E_3^{11,1}\to E_3^{10,4}$ is nonzero.
\end{lemma}
\begin{proof}
The bottom tensor class $p$ is a permanent cycle.  Indeed its stem is
two, the bottom degree of the smash spectrum, and its filtration is
zero, so no differential can enter or leave it.  The coefficient class
$v$ is also a permanent cycle.  The Leibniz rule therefore implies
that $v^4p$ cannot support an outgoing differential on any page on
which it is still nonzero.

Its actual product in $\widetilde{ku}_{10}(X)$ is zero.  The external
product of the canonical generators of $\widetilde{ku}_1(BC_4)$ is the
bottom tensor class in Davis's module.  Davis's Theorem~1.6 with $e=2$
states that its annihilator is exactly $(4,2v,v^4)$; hence multiplication
by $v^4$ is zero.  The multiplicative associated-graded class $v^4p$
must therefore disappear before $E_\infty$.

It is nonzero on $E_3$.  An incoming $d_r$, for $r\geq3$, must have
source $(11,4-r)$.  For $r=4$ the preceding lemma proves that the entire
source in filtration zero is permanent, so that differential is zero.
For $r>4$ the source has negative filtration and is zero.  Thus the
class must be hit by $d_3$.
\end{proof}

\begin{theorem}[The primitive complex $d_3$]\label{lem:ku-d3-spike-identification}
The actual Adams third differential satisfies
\[
 d_3(g_2)=d_3(g_4)=v^4p,\qquad
 d_3(g_1)=d_3(g_3)=d_3(g_5)=d_3(f)=0.
\]
\end{theorem}
\begin{proof}
Write $d_3(y)=\lambda(y)v^4p$.
The two automorphisms of $C_4^2$ whose pullbacks in mod-two cohomology are
\[
 S:(x_0,z_0;x_1,z_1)\mapsto(x_1,z_1;x_0,z_0),
\]
\[
 T:(x_0,z_0;x_1,z_1)\mapsto
       (x_0+x_1,z_0+z_1;x_1,z_1)
\]
are induced by matrices invertible over $\mathbb Z/4$.
Their actions on the target class are trivial: the bottom two-circle
class transforms by the determinant modulo two, which is one, and
the coefficient $v$ is fixed.  Consequently $\lambda$ is invariant
under both automorphisms.  It vanishes on the four axis classes and on
$f$ by the preceding lemma.

For completeness the needed source transformations are given here in
full; they are computed by dualizing the simultaneous substitutions in
$x_0,x_1,z_0,z_1$:
\[
 S(g_j)=g_{6-j}\ (1\leq j\leq5),\quad S(f)=f,
 \quad S(a_0)=a_1,\quad S(b_0)=b_1.
\]
\[
\begin{array}{c|l}
y&T(y)\\\hline
g_1&g_1+g_3+g_5\\
g_2&g_2+g_3+a_0+b_0\\
g_3&g_3\\
g_4&g_4+g_5+a_0+b_0\\
g_5&g_5\\
f&f\\
a_0&a_0\\
a_1&g_1+g_2+g_3+g_4+g_5+f+a_0+a_1\\
b_0&b_0\\
b_1&f+b_0+b_1.
\end{array}
\]
These are just expansions of binomial powers over $\mathbb F_2$;
coefficient classes $h_0,v$ are unchanged.

Put $\lambda_j=\lambda(g_j)$.  Invariance on $g_2$ gives
$\lambda_3=0$.  Invariance on $g_4$ gives $\lambda_5=0$.
The swap gives $\lambda_1=\lambda_5=0$ and
$\lambda_2=\lambda_4$.  Hence the functional is either zero or
$(0,1,0,1,0)$ on the five $g$ classes.  The preceding nonvanishing
lemma excludes zero, proving the result.
\end{proof}

\paragraph{Propagation by actual complex-oriented cap operations.}
Let $U_i\in ku^2(BG)$ be the complex-oriented Euler class of the
canonical complex line from the $i$th factor.  Use cap product with
these generalized-cohomology classes, not an unspecified cap action
by ordinary mod-two cohomology.  The Euler classes have filtration-zero
Adams edges $z_i$, and their cap operations are maps of $ku$-modules.
Here is the $E_2$ justification for the whole cap map.  Mod-two
Hurewicz $ku\to H\mathbb F_2$ sends $U_i$ to $z_i$ and commutes
with cap product.  Thus on the filtration-zero Adams edge the map is
ordinary cap with $z_i$, namely dual exponent lowering.  Both $Q_0$
and $Q_1$ vanish on $H^*(BG;\mathbb F_2)$, so the complex $E_2$ term
is free over $\mathbb F_2[h_0,v]$ on these filtration-zero generators.
Coefficient linearity of the actual $ku$-module cap map consequently
determines its entire $E_2$ action.  The Adams pairing used here is
\cite[Theorem~4.4]{may_hinfty}.  The construction can first be made on
finite skeleta and then passed to each fixed degree.
Thus, modulo the axis summands when an exponent is exhausted,
\[
 U_0^rU_1^s\cap\Gamma_{a,b}=\Gamma_{a-r,b-s}.
\]
Axis summands have zero third differential.  These actual cap operations
commute with all Adams differentials.

To apply coefficient extraction one also needs the elementary $E_3$
normal form: in even stems and filtration at least three, every
$h_0$-divisible class is zero, and the surviving tensor classes are the
$h_0$-free monomials $v^e x_0z_0^j x_1z_1^k$.  This follows from the
relations
\[
 h_0^2m_{jk}+h_0v m_{j-1,k}=0,
 \qquad h_0^2m_{jk}+h_0v m_{j,k-1}=0,
\]
where a negative index is omitted.  Subtracting the two shifted
relations shows
$h_0v^2m_{jk}=h_0v^2m_{j-1,k+1}$; at $j=0$ this is zero, so all these
terms vanish.  The same relations then give $h_0^2v m_{jk}=0$ and
$h_0^3m_{jk}=0$.  In these smash-target bidegrees the positive-degree
zero-$x$ part has no $d_2$-cycles: its component in either one-$x$
direction is triangular and injective.  There is no coefficient-only
unit summand in the reduced smash calculation.

It follows that the coefficient of $v^4m_{rs}$ in
$d_3(\Gamma_{a,b})$ is obtained by capping with $U_0^rU_1^s$ and reading
the coefficient of $v^4p$.  The source then has total index six, so the
primitive theorem applies.  The coefficient is one exactly when
\[
 (a-r,b-s)=(2,4)\quad\hbox{or}\quad(4,2).
\]
Consequently
\[
d_3(v^m\Gamma_{a,b})=
v^{m+4}m_{a-2,b-4}+v^{m+4}m_{a-4,b-2},
\]
with terms having negative indices omitted.  This argument uses neither
the false uniqueness assertion for doubled Tor representatives nor the
unsupported May--Milgram boundary-row identification.

\begin{theorem}[The complex third differential]\label{theo:ku-smash-d3}
Writing \(m_{rs}=x_0z_0^r x_1z_1^s\), the actual Adams differential is
\[
 d_3(v^m\Gamma_{a,b})=
 v^{m+4}m_{a-2,b-4}+v^{m+4}m_{a-4,b-2},
\]
with negative-index terms omitted. In particular
\(d_3(\Gamma_{2,4})=d_3(\Gamma_{4,2})=v^4p\).
\end{theorem}
\begin{proof}
The primitive theorem and the complex-oriented cap operations just
constructed give the formula for \(m=0\); coefficient \(v\)-linearity
gives it for all \(m\geq0\).
\end{proof}

\begin{remark}[Leading terms need not be unique]
In primitive bidegree \((11,1)\), the cycles with the prescribed
Bott-degree-zero part of \(\Gamma_{a,b}\), \(a+b=6\), are precisely
\(\Gamma_{a,b}\) and \(\Gamma_{a,b}+v\Phi_5\).  This follows from the
six-element source basis: the five \(\Gamma\) leading parts are
independent, and the sixth basis vector has zero such part.
Both cycles have the same actual \(d_3\), since \(v\Phi_5\) is induced
from the diagonal cyclic class and has zero differential.
No identification of a unique AHSS representative is needed here.
\end{remark}

\subsection{An all-degree counting certificate for the real \texorpdfstring{$E_3$}{E3} page}
\label{subsec:all-degree-real-e3-count}

This calculation takes place entirely on the actual $E_2,d_2$ complex.
No assertion about $d_3$, collapse, or geometric generation enters it.
Let $t$ record the stem and $y$ the Adams filtration.  All dimensions
and generating functions in this subsection have integer coefficients;
in particular the signs in Euler characteristics are ordinary integer
signs, not signs in $\mathbb F_2$.

\paragraph{Normal forms in the two-$x$ part.}
Put $P=\langle x\rangle$ and
$D_k=\langle xz^{2k-1},xz^{2k}\rangle$, $k\ge1$.
The two-$x$ real $E_2$ part consists of the point block $P\otimes P$,
the two boundary bows $P\otimes D_k,D_k\otimes P$, and the two
ordinary bows in each $D_k\otimes D_l$ summed bow.  Call the
filtration-zero bottom of each nonpoint bow $b$; these bottoms have
generating function
\begin{equation}\label{eq:real-e3-bottom-GF}
 B(t)=\frac{2t^4}{1-t^4}
       +\frac{t^6(1+t^2)}{(1-t^4)^2}.
\end{equation}
The first term counts the boundary bows and the second the two bows
in each summed bow.  The real generators on such a bow are named
$h_0^r v_0^e b$, $e,r\ge0$, with stem $|b|+2e$ and filtration
$r+e$.  As throughout, $v_0^e$ names the normalized Ext generator;
it is not asserted to be a real coefficient operation.

\begin{lemma}[Two-$x$ normal forms]\label{lem:all-degree-two-x-normalforms}
The two-$x$ $E_3$ part has the following vector-space basis:
\[
 v_0^e b\ (e\ge0),\qquad h_0b,
\]
for every nonpoint bow bottom $b$, together with
\[
 \mu^qp,\quad h_0\mu^qp,\quad\eta\mu^qp,
 \quad\eta^2\mu^qp,\quad\omega\mu^qp\qquad(q\ge0).
\]
Consequently its bivariate Hilbert series is
\begin{equation}\label{eq:real-e3-Htwo}
 H_2(t,y)=B(t)\left(\frac1{1-t^2y}+y\right)
 +\frac{t^2(1+y+ty+t^2y^2+t^4y^3)}{1-t^8y^4}.
\end{equation}
\end{lemma}
\begin{proof}
Every two-$x$ generator is a $d_2$ cycle: choose its actual $R_2$
preimage in the same two-$x$ complex summand, whose complex $d_2$
is zero; point coefficient classes are also cycles.  We give a
uniform elimination of all the additional $h_0$ multiples.

First work modulo the point subcomplex.  Define the shift on every
nonpoint bow, including the one-$x$ source bows and both bows of
each summed bow.  On this quotient the formal
shift $v_0^e b\mapsto v_0^{e+2}b$ commutes with $d_2$.  This follows
directly from the actual bow complexification formula
\[
 c(v_0^e b)=v^e a+(e\bmod2)h_0v^{e-1}a',
\]
and the fact that nonpoint complexification is injective on $E_2$:
shifting $e$ by two multiplies the formula by $v^2$.  Any discrepancy
in the point coefficient residue is discarded on this quotient.
Thus this is a justified algebraic shift on this quotient complex,
not a claim of real Bott-$v$ linearity.

Here are the finite path matrices that eliminate the required
$h_0$ multiples before applying these shifts.  In stem $4K$,
write the filtration-zero bow bottoms as
\[
 b_0=m_{0,2K-1},\quad b_K=m_{2K-1,0},\quad
 b_i=m_{2i-1,2K-2i}+m_{2i,2K-2i-1},
\]
and put $o_i=m_{2i-1,2K-2i-1}$, $1\le i<K$.
Here $m_{a,b}=x_0z_0^ax_1z_1^b$ and negative indices are zero.
The filtration-zero one-$x$ sources
$z_0^{2k-1}x_1z_1^{2l-1}$ and
$x_0z_0^{2k-1}z_1^{2l-1}$, $k+l=K+1$, have differential
columns
\[
 h_0^2b_0,\quad h_0^2b_K,\quad
 h_0v_0o_i,\quad h_0^2b_i+h_0v_0o_i\quad(1\le i<K).
\]
They span all $h_0^2b_i$ and $h_0v_0o_i$.

In stem $4K+2$, put
\[
\begin{gathered}
 o_k=m_{2k-1,2K-2k+1},\qquad
 p_0=m_{0,2K-1},\quad p_K=m_{2K-1,0},\\
 p_i=m_{2i-1,2K-2i}+m_{2i,2K-2i-1}.
\end{gathered}
\]
The endpoint one-$x$ sources $x_0z_1^{2K+1}$ and
$z_0^{2K+1}x_1$ have differential columns
$h_0v_0p_0,h_0v_0p_K$.  For $k+l=K+1$, the adjacent-pair
sources with complexifications
\[
\begin{gathered}
 z_0^{2k}x_1z_1^{2l-1}+z_0^{2k-1}x_1z_1^{2l},\\
 x_0z_0^{2k}z_1^{2l-1}+x_0z_0^{2k-1}z_1^{2l}
\end{gathered}
\]
have differential columns
\[
 h_0v_0p_{k-1}+h_0^2o_k,\qquad
 h_0v_0p_k+h_0^2o_k.
\]
Starting at $p_0$ and walking along this path kills every
$h_0v_0p_i$ and $h_0^2o_k$.  These two path calculations apply
also for $K=1$, with empty interior ranges interpreted literally.

It remains to kill the even-residue $h_0v_0^2$ bow terms.  In stem
$4K$, $K\ge2$, the odd-residue second-bow sources with
complexifications
\[
 v(z_0^{2k}x_1z_1^{2l-1}+z_0^{2k-1}x_1z_1^{2l})
       +h_0z_0^{2k}x_1z_1^{2l},\qquad k+l=K,
\]
together with the right endpoint source whose complexification is
$vz_0^{2K-1}x_1+h_0z_0^{2K}x_1$, have, modulo the
already eliminated $h_0^3$ and $h_0^2v_0$ terms, differential
columns $h_0v_0^2b_i$ for every bow bottom in stem $4K-4$.
For example the displayed differential has last part
\[
 h_0v^2(m_{2k-2,2l-1}+m_{2k-3,2l}).
\]
The endpoint differential is
$h_0^3m_{2K-1,0}+h_0v^2m_{2K-3,0}$.
In stem $4K+2$, the first-bow odd-residue source with
complexification
\[
 vz_0^{2k-1}x_1z_1^{2l-1}
       +h_0z_0^{2k}x_1z_1^{2l-1},\qquad k+l=K+1,
\]
has differential
\[
 h_0^3m_{2k-1,2l-1}+h_0v^2m_{2k-3,2l-1}.
\]
For $2\le k\le K$ these are all the required $h_0v_0^2$
bottoms in stem $4K-2$.  The smallest such bow stem is six,
so these ranges include every case where this type occurs.
All displayed identities follow by direct substitution in the
complex $d_2$ formula followed by injectivity of nonpoint $c$;
they are identities modulo the point subcomplex.

Applying the justified two-step shifts to the first path matrices
kills $h_0^2v_0^{2q}b$ and $h_0v_0^{2q+1}b$.  Applying them
to the last matrices then kills $h_0v_0^{2q+2}b$, because their
other terms are among the former killed families.  Multiplication
by $h_0$ kills all their further multiples.  This leaves precisely
the listed bow normal forms, modulo possible point terms.

On the point block, the actual $R_2$ rule applied to the complex
sources $v^{4q}z_0^2x_1$ and $v^{4q+2}z_0^2x_1$ gives the exact
real boundaries $h_0^2\mu^qp$ and $h_0\omega\mu^qp$.
The remaining five point classes listed in the lemma are
nonboundaries.  Four have no real $h_0$ factor, and every actual
real $d_2$ image has an $h_0$ factor.  For $h_0\mu^qp$, a
complex class whose $R_2$ image has this point coefficient must
contain the $h_0$-free term $v^{4q+1}p$.  A complex $d_2$ image
cannot contain such a term.  Naturality and surjectivity of $R_2$
on one-$x$ bows exclude a real boundary.

Finally the bow normal forms are independent: their $h_0$-free
members are independent in the actual real $E_2$ basis and no
boundary has such a coefficient; the $h_0b$ members have
filtration one and cannot be hit.  The latter cannot mix with a
point member of the same bidegree, and the preceding point
argument excludes any remaining relation.  This proves the basis
and its generating function.
\end{proof}

\paragraph{Counting the one-$x$ homology by Euler characteristic.}
Let $C_j(t,y)$ be the real $E_2$ Hilbert series of the summands
with $j$ occurrences of $x$, $j=0,1,2$, counting point coefficient
classes as $j=2$.  Directly enumerating the boundary bows and two
bows per summed bow gives
\begin{align*}
 C_0&=\frac{t^4(1+t^2)}{(1-t^4)^2(1-t^2y)(1-y)},\\
 C_1&=\frac{1}{(1-t^2y)(1-y)}
 \left(\frac{2t^3}{1-t^4}+
       \frac{2t^5(1+t^2)}{(1-t^4)^2}\right),\\
 C_2&=\frac{B(t)}{(1-t^2y)(1-y)}
 +\frac{t^2}{1-t^8y^4}
 \left(\frac{1+t^4y^3}{1-y}+ty+t^2y^2\right).
\end{align*}
The differential increases the $x$ count by one and changes
$(n,s)$ to $(n-1,s+2)$.  Therefore $n+j$ and $s-2j$ are
preserved, so the complex splits into finite three-term complexes
for these two conserved weights.  Its zero-$x$ differential is
injective.  Indeed complexification is injective on that summand,
and its first-coordinate differential has a noncancellable
$h_0^2$ term at the highest $z_0$ exponent; the second-coordinate
images have different $x$ support.

Consequently the Euler characteristic of every finite weight
complex gives the exact identity
\begin{equation}\label{eq:real-e3-Euler-certificate}
 H_2(t,y)-t^{-1}y^2H_1(t,y)
 =C_2(t,y)-t^{-1}y^2C_1(t,y)+t^{-2}y^4C_0(t,y).
\end{equation}
There is no infinite-dimensional Euler characteristic here:
each coefficient at fixed conserved weights is finite, and this
identity is taken coefficient by coefficient before specializing
$y$.

For an easily checked one-variable certificate put
\begin{align*}
 U(t)&=B(t)\frac{2-t^2}{1-t^2}
       +\frac{t^2(2+t+t^2+t^4)}{1-t^8},\\
 V(t)&=\frac{t^4}{(1-t^2)(1-t^4)}
       +\frac{t^3+t^4-3t^6}{1-t^8}
       -\frac{4t^{10}}{(1-t^4)(1-t^8)}.
\end{align*}
Then $U(t)=H_2(t,1)$, and cancellation of the factor $1-y$
on the right of \eqref{eq:real-e3-Euler-certificate} gives $V(t)$.
Here is an explicit way to check that cancellation.  Set
\[
 A_1=\frac{2t^2}{1-t^4}
       +\frac{2t^4(1+t^2)}{(1-t^4)^2},\qquad
 A_0=\frac{t^2(1+t^2)}{(1-t^4)^2},\quad
 N(y)=B-y^2A_1+y^4A_0.
\]
The right side of \eqref{eq:real-e3-Euler-certificate} is
\[
 \frac{F(y)}{1-y}+\frac{t^2(ty+t^2y^2)}{1-t^8y^4},\qquad
 F(y)=\frac{N(y)}{1-t^2y}
       +\frac{t^2(1+t^4y^3)}{1-t^8y^4}.
\]
Direct substitution gives
$N(1)=-t^2/(1+t^2)$, $N'(1)=0$, and $F(1)=0$.
The specialization is consequently
$-F'(1)+(t^3+t^4)/(1-t^8)$, which is precisely the displayed
$V(t)$.  This differentiates a rational function with a known
factor $1-y$; it does not exchange an infinite divergent sum
with specialization.
Thus the total real $E_3$ dimensions have the exact all-degree
generating function
\begin{equation}\label{eq:real-e3-all-degree-total}
 \boxed{\displaystyle
 \sum_{n\ge2}e_3(n)t^n=(1+t)U(t)-tV(t).}
\end{equation}
Expansion of $1/(1-t^a)$ and $1/(1-t^a)^2$ now gives the
quadratic eight-residue table in the order theorem by ordinary
coefficient extraction.  In degrees two through seven the
coefficients are $2,3,5,3,5,7$.  Formula
\eqref{eq:real-e3-all-degree-total}, with the finite path matrices
above, is an all-degree certificate for that table, rather than
an extrapolation from finitely many Adams charts.
More explicitly, put $q=t^8$ and
\[
 (\ell_0,\ldots,\ell_7)=(4,2,6,8,8,6,10,12),\qquad
 (\kappa_0,\ldots,\kappa_7)=(1,1,2,3,5,3,5,7).
\]
The rational-function identity certifying the coefficient
extraction is
\begin{align*}
 (1+t)U(t)-tV(t)
 ={}&2t^2+3t^3+5t^4+3t^5+5t^6+7t^7\\
 &+\frac{q}{(1-q)^3}\sum_{r=0}^7t^r
 \bigl(4(1+q)+\ell_r(1-q)+\kappa_r(1-q)^2\bigr).
\end{align*}
It follows either by putting both sides over their common
denominator or by the identities
$\sum_{d\ge1}dq^d=q/(1-q)^2$ and
$\sum_{d\ge1}d^2q^d=q(1+q)/(1-q)^3$.
Thus $e_3(8d+r)=4d^2+\ell_rd+\kappa_r$ for every $d\ge1$,
not only in a finite checked range.

\subsection{Sufficient lower bounds for the real third differential}

The order argument requires upper bounds for the real $E_4$ page.
It is therefore enough to exhibit differentials of the ranks below;
one need not first classify every possible real third differential.
We use the actual $E_2$ realification table and the complex third
differential proved by the symmetry argument.

Let $m_{jk}=x_0z_0^jx_1z_1^k$, and let $P$ denote extraction of the
$h_0$-free part of a complex $E_2$ class.  This gives a well-defined
linear map on complex $E_3$ classes: every complex $d_2$ boundary is
$h_0$-divisible.  Put
\[
 \partial m_{jk}=
 \begin{cases}m_{j-1,k},&j>0\text{ even},\\0,&\text{otherwise}\end{cases}
 +
 \begin{cases}m_{j,k-1},&k>0\text{ even},\\0,&\text{otherwise}.\end{cases}
\]
The actual bow and summed-bow realification tables imply
\[
             Pc_2R_2(v^e m_{jk})=v^e\partial m_{jk}.
\]
This identity records only the $h_0$-free part.  It does not incorrectly
assert that $R$ is $v$-linear or discard nonzero terms in real $E_2$.
It passes to $E_3$ on cycles.

\begin{proposition}[Exhibited ranks]\label{theo:ko-smash-d3}
For $s>0$ put $A=(n-2s+3)/2$.  The rank of
$d_3^{ko}:E_3^{n,s}\to E_3^{n-1,s+3}$ is at least
\[
\rho(n,s)=
\begin{cases}
1,&A=5,\ s\not\equiv1\pmod4,\\
(A-3)/2,&A\ge7\text{ odd},\\
(A-6)/2,&A\ge8\text{ even},\\
0,&\text{otherwise}.
\end{cases}
\]
Here the lower bound is zero unless $n$ is odd and $A$ is integral.
\end{proposition}
\begin{proof}
For any complex $E_3$ cycle $w$, naturality gives
$d_3^{ko}R_3(w)=R_3d_3^{ku}(w)$.
Take $w=v^{s-1}\Gamma_{a,b}$ with $a+b=A+1$.
After applying $Pc_3$, its target is
\[
v^{s+3}\partial(m_{a-2,b-4}+m_{a-4,b-2}).
\]

First let $A=2r+5\ge7$.  Omit the common factor $v^{s+3}$ and define
\[
t_0=m_{0,2r-1},\qquad t_r=m_{2r-1,0},
\]
\[
t_i=m_{2i,2r-1-2i}+m_{2i-1,2r-2i}\quad(1\le i<r).
\]
Their monomial supports are disjoint, so these $r+1$ vectors are
independent.  Use the complex sources
\[
\Gamma_{2,2r+4},\quad
\Gamma_{2i+2,2r+4-2i}\ (1\le i\le r),\quad
\Gamma_{2r+4,2}.
\]
The displayed formula gives the target columns
\[
t_0,\quad t_{i-1}+t_i\ (1\le i\le r),\quad t_r.
\]
They span all $r+1$ independent vectors.  Since this independence
holds after a linear projection of the actual real targets, the
actual targets have rank at least $r+1=(A-3)/2$.

If $A=2r+6\ge8$, define $t_i=m_{2i-1,2r-2i+1}$ for $1\le i\le r$.
These again have disjoint supports.  The complex sources
\[
\Gamma_{3,2r+4},\quad
\Gamma_{2i+3,2r+4-2i}\ (1\le i<r),\quad
\Gamma_{2r+3,4}
\]
give columns $t_1,t_i+t_{i+1},t_r$, of rank $r=(A-6)/2$.
The two endpoints coincide when $r=1$, but the rank is still one.

Finally let $A=5$.  Apply $R_3$ to
$d_3(v^{s-1}\Gamma_{2,4})=v^{s+3}p$.
If $s=4q+2,4q+3,4q+4$, its real target is respectively
\[
h_0\mu^{q+1}p,\qquad h_1^2\mu^{q+1}p,
\qquad\omega\mu^{q+1}p.
\]
All are nonzero on $E_3$.  They are coefficient $E_2$ cycles.
The latter two are not $h_0$-divisible, whereas all actual real
$d_2$ images are $h_0$-divisible by the preceding replacement lemma.
The first cannot be a $d_2$ boundary either: to produce its coefficient
on the one-cell summand through the rule $d_2=R_2d_2^{ku}$, the
one-cell term in $d_2^{ku}$ would have to contain the $h_0$-free term
$v^{4q+5}p$.  That never occurs because every entry of $d_2^{ku}$
is $h_0$-divisible.  Realification preserves each algebraic block,
so terms on other blocks cannot contribute to this coefficient.
For $s\equiv1\pmod4$ the coefficient realification is zero, giving
the stated zero lower bound.
\end{proof}

\begin{corollary}[An upper-bound argument with no classification gap]
If $e_3(n)$ is the total real $E_3$ dimension and
$L(n)=\sum_s\rho(n,s)$, then
\[
 \dim E_4(n)\le e_3(n)-L(n)-L(n+1).
\]
No assertion that there are no additional third differentials is
needed.  Equality, and exclusion of later differentials, may
subsequently follow from independent geometric lower bounds.
\end{corollary}

\section{Geometric and homological detection}\label{section:eta}
The eta and homological computations below provide the geometric lower bounds
used in Section~\ref{sec:conclusion}. We use the notation for the connective
assembly kernel fixed there in Definition~\ref{def:glr-kernels}; its
assembly-image calculation is Lemma~\ref{lem:smash-assembly-image}.

\subsection{Geometric input and induction}\label{sec:geometric-input}
For an odd weight vector $\mathbf a=(a_1,\ldots,a_{2k})$,
let $\tau_{\mathbf a}$ be the diagonal $C_4$-representation on $\mathbb C^{2k}$.
Define
\[
L^{4k-1}(\mathbf a)=S^{4k-1}/\tau_{\mathbf a}(C_4),\qquad
\mathcal X^{4k+1}(\mathbf a)=
 S(H^{\otimes2}\oplus\mathbb C^{2k-1})/\tau_{\mathbf a}(C_4),
\]
where $H\to\mathbb CP^1$ is the Hopf line bundle. We use $\mathcal X$ for
the lens-space bundle to distinguish it from the smash space $X$.
The spin structures and virtual combinations are the BGS ones used by the
cyclic detection theorem \cite{botvinnikgilkeystolz,donnelly,baereta}.
A tilde on a geometric cycle means subtraction of the same spin manifold
with constant reference map. A Bott manifold $B^8$ represents the
coefficient $\mu$; products with it provide the higher-dimensional cycles.

Let $g_0,g_1$ be the standard generators of $G$. For $H_{m,n}=\langle g_0^mg_1^n\rangle\subset G$, induction means composition
of the reference map with $BH_{m,n}\to BG$. It gives a cycle over $BG$;
obtaining the prescribed fundamental group is a subsequent surgery step.
For a character $\rho_{u,\tilde u}$ its restriction is
$\sigma^{mu+n\tilde u}$, with exponents modulo four. This formula specifies
all entries of the restriction matrices without separate conflicting tables.

\subsection{Odd-dimensional eta detection}
\label{sec:repaired-eta}

We specify both the geometric cycles and the target lattice before
computing the image.  For a virtual rank-zero character $\chi-1$, write
\[
 \rho_\chi(M)=\bar\eta(D_M\otimes\chi)-\bar\eta(D_M),\qquad
 \bar\eta(D)=\tfrac12\bigl(\eta_{\mathrm{spectral}}(D)+\dim\ker D\bigr).
\]
The APS theorem gives an invariant modulo $\mathbb Z$.  In dimension
$3\pmod8$, a real twist has quaternionic chiral index, so the period is
$2\mathbb Z$ for a real virtual representation.  A nonreal character
still has period $\mathbb Z$.  In dimension $7\pmod8$, period two would
instead require a quaternionic twist.  There are no quaternionic
irreducible characters of $G$, so all the character coordinates used
here have period one in that dimension.  These parity statements and
the reduced-eta convention are stated in
\cite[Lemmas~3.1--3.3 and Theorem~3.5]{BarreraGilkeyQuaternion}.

These invariants factor through connective real $K$-homology, not just
spin bordism.  Indeed the $2$-local kernel of the spin orientation is
generated by the $\mathbb HP^2$ transfer, and rank-zero eta invariants
vanish on this transfer
\cite[Lemma~4.3]{botvinnikgilkeystolz}.  The transfer argument is recalled
in \cite[Theorem~3.6 and proof of Corollary~3.7]{BarreraGilkeyQuaternion}.
Since the reduced groups here are finite $2$-groups, odd denominators
can be cleared and inverted without changing this conclusion.

Use the slopes
\[
 \mathcal H=((0,1),(1,1),(1,3),(2,1),(1,0),(1,2))
\]
and the nine conjugate-character representatives
\[
 \mathcal R=((0,1),(0,2),(1,0),(1,1),(1,2),(1,3),(2,0),(2,1),(2,2)).
\]
For slope $(h,k)$, restriction sends $\rho_{u,v}$ to
$\sigma^{hu+kv}$.  All exponents are read modulo four.

\begin{definition}[Explicit geometric columns]
\label{def:eta-detection-matrix}
If $n=2m-1$ with $m$ even, let $L_0^n$ and $L_1^n$ be the round
quotient lens spaces with weights $(1,\ldots,1)$ and
$(1,\ldots,1,3)$ respectively.  Use the spin lifts given by the
positive half-angle rotations in these complex planes.  Their fourth
powers are $(-1)^{\sum a_j}=1$, so these are genuine $C_4$ spin lifts.
Choose the orientations of the two cycles so that their eta rows have
the common sign removed in \eqref{eq:repaired-lens-eta} below.

If $n=2m+1$ with $m$ even, use the lens-space bundle
\[
 \mathcal X^n=S(H^{\otimes2}\oplus\mathbb C^{m-1})/C_4
 \longrightarrow\mathbb CP^1
\]
with scalar weight one and the compatible spin lift, oriented to remove
the common sign in \eqref{eq:repaired-bundle-eta}.  One column per slope
suffices in this case; it may be repeated as two equal columns.

Let $C_n$ be the subgroup generated by these cycles induced along the
six slopes, and let $E_n$ be their eta matrix.  These are cycles in the
actual dimension $n$, not low-dimensional lens spaces multiplied by a
Bott manifold.  Lens spaces have quotient round PSC metrics.  The
bundles are spin: the base is spin, the vector bundle has even first
Chern class, and its even complex rank gives an order-four spin lift
of the scalar action.  Their fibers are PSC lens spaces with compact
unitary structure group, so fiber shrinking gives PSC on the total
space.  Consequently every generator of $C_n$ is a PSC cycle.
\end{definition}

\begin{lemma}[The geometric eta entries]
For the orientation-normalized lens cycles the one-factor values are
\begin{equation}\label{eq:repaired-lens-eta}
 \rho(L_0)=(0,x+2x^2,2x,x+2x^2),\qquad
 \rho(L_1)=(0,x-2x^2,2x,x-2x^2),
\end{equation}
in the order $r=0,1,2,3$, where
\[
 x=2^{-q},\qquad
 q=\begin{cases}2d+2,&n=8d+3,\\2d+3,&n=8d+7.\end{cases}
\]
For the orientation-normalized bundle,
\begin{equation}\label{eq:repaired-bundle-eta}
 \rho(\mathcal X)=(0,x,0,-x),\qquad
 q=\begin{cases}2d+1,&n=8d+1,\\2d+2,&n=8d+5.\end{cases}
\end{equation}
\end{lemma}
\begin{proof}
The spherical fixed-point formula with the specified spin lift is
\cite[Theorem~4.6]{botvinnikgilkeystolz}
\[
 \rho_r(L(a))=\frac14\sum_{j=1}^3
 \frac{i^{j\sum a_l/2}(i^{jr}-1)}
      {\prod_l(1-i^{ja_l})}.
\]
This is the reduced eta invariant.  For even $m$, the all-one local
factors at $j=1,3$ equal $(-1)^{m/2}2^{-m/2}$, and the factor at $j=2$
equals $(-1)^{m/2}2^{-m}$.  Replacing one weight by three preserves the
first two factors and reverses the last one.  Substitution gives
$(-1)^{m/2+1}$ times the two rows in
\eqref{eq:repaired-lens-eta}, with $x=2^{-m/2-1}$.
The chosen orientation removes that sign.  For example the two rows
in dimension three are $(0,3/8,1/2,3/8)$ and $(0,1/8,1/2,1/8)$.

For the bundle, apply the fixed-set formula to its disk bundle over
$\mathbb CP^1$.  Only the first summand has nonzero Chern root, namely
$2h$ with $\int h=1$.  Expanding the spin factor
$e^h/(1-i^je^{2h})$ to degree two gives
$(1+i^j)/(1-i^j)^2$.  Thus
\[
 \rho_r(\mathcal X)=\frac14\sum_{j=1}^3
 \frac{i^{jm/2}(1+i^j)(i^{jr}-1)}{(1-i^j)^{m+1}}.
\]
The $j=2$ term vanishes.  The other two terms give
$(-1)^{m/2+1}(0,x,0,-x)$ with $x=2^{-m/2-1}$.  The orientation choice
again removes the common sign.  These computations also specify the
square root of the determinant character; no subsequent change of
spin normalization is made.
\end{proof}

\begin{lemma}[The target-normalized lattice]
\label{lem:eta-target-lattice}
In the lens residues put $a=2^{q-1}$ and $D=2a^2$.
For $r=hu+kv\pmod4$, form the $9\times12$ integer matrix with entry pair
\[
 r=0:(0,0),\qquad r=1,3:(a+1,a-1),\qquad r=2:(2a,2a).
\]
In residue three, divide the rows $(0,2),(2,0),(2,2)$ by two;
in residue seven, divide no rows.  Every coordinate then has modulus
$D$.  The resulting matrices are equivalent over $\mathbb Z_{(2)}$ to
diagonal matrices with factors
\[
\begin{array}{c|l}
n\pmod8&\text{diagonal factors}\\\hline
3&1,1,2,2a,2a,4a,4a,4a,4a\\
7&1,1,2,4a,4a,4a,4a,4a,8a.
\end{array}
\]
The constant six-slope matrix for the bundle residues has factors
$(1,1,2,2,4,4)$.
\end{lemma}
\begin{proof}
Multiplication by $D$ gives the displayed integer entries.  In residue
three the division of a real row by two is the target isomorphism
$\mathbb R/2\mathbb Z\to\mathbb R/\mathbb Z$, not a division of a
source cycle.  The entries in each such row are even.

Here are reproducible elimination data for the variable matrices.
At stage $k$, swap row $k$ with row $i$ and column $k$ with column $j$,
then clear column $k$ by row additions and row $k$ by column additions.
All indices are zero-based in the current matrix.  For residue three
the triples $(k,i,j)$ are
\[
(0,0,0),(1,2,2),(2,3,8),(3,3,3),(4,6,7),
(5,5,5),(6,6,6),(7,7,9),(8,8,10).
\]
They give the exact pivots
\[
a+1,\ a+1,\ 2,\ \frac{2a}{a+1},\ -2a,\
2a(a+2),\ -\frac{8a}{a+2},\ \frac{4a}{a+1},\ 4a.
\]
For residue seven the triples are
\[
(0,0,0),(1,2,2),(2,3,8),(3,3,3),(4,4,6),
(5,5,5),(6,6,7),(7,7,9),(8,8,10),
\]
and the pivots are
\[
a+1,\ a+1,\ 2,\ \frac{4a}{a+1},\ 2a(a-2),\
-\frac{8a}{a-2},\ -4a,\ \frac{4a}{a+1},\ 8a.
\]
For $n=3$ we have $a=2$.  In the residue-three recipe replace only
$(5,5,5)$ by $(5,5,6)$.  The pivots are then
\[
3,\ 3,\ 2,\ \frac43,\ -4,\ -\frac83,\ -24,\ \frac83,\ 8.
\]
Every clearing multiplier has denominator one or three, and the
result again has three zero columns.  These pivots are $2$-local
unit multiples of $(1,1,2,4,4,8,8,8,8)$, as required.

In all remaining dimensional cases $4\mid a$.  Every clearing
multiplier is $2$-integral: after substituting $a=4t$ and cancelling,
its numerator is an integer polynomial and its denominator has odd
constant coefficient and even other coefficients.  This assertion
can be checked directly by the specified elimination recipe; the
symbolic verifier supplied with the paper checks it for each
multiplier.  Also $a+1$ is odd and $a\pm2$ have valuation one.
Each pivot is therefore a $2$-local unit times its asserted factor.
The recipe leaves three zero columns, so these invertible operations
give the whole image lattice, not just selected minors.  The bundle
matrix is the constant matrix whose entries are $0,1,-1$ according
to the restricted character; ordinary integer reduction gives
$(1,1,2,2,4,4)$.
\end{proof}

\begin{proposition}[Eta matrix orders]\label{prop:eta-matrix-orders}
For the explicit PSC cycles above, the logarithmic eta-image orders are
\[
\begin{array}{c|c}
n&\log_2|\operatorname{im}E_n|\\\hline
8d+1\ (d\ge1)&12d\\
8d+3\ (d\ge1)&24d+10\\
8d+5\ (d\ge1)&12d+6\\
8d+7\ (d\ge0)&24d+19.
\end{array}
\]
In dimension five the same bundle calculation gives six detected bits.
\end{proposition}
\begin{proof}
A diagonal entry $b$ in a target with modulus $D$ has image order
$D/\gcd(D,b)$.  Since $D=2^{2q-1}$, the lens Smith factors give
$12q-14=24d+10$ and $12q-17=24d+19$.  For the bundle matrix, modulus
$2^q$ and factors $(1,1,2,2,4,4)$ give $6q-6$, hence the other two
rows in their stated ranges.  When $n=5$, $q=2$ gives six bits.
All these are image orders in the actual target tori.  No change of
source generators or additive normalization correction is needed.
\end{proof}

\subsection{Even-dimensional geometric cycles and their edges}

The following family replaces the projective-base family: the base
covering group must be $C_4$, not $C_2$, for the stated map and
triangular edge formula.  We prove freeness, spin, curvature, and the
characteristic numbers directly.  These are statements about actual
manifolds; they do not by themselves establish higher Adams-filtration
generation.

For $r\geq1$, let $L_1^{4r-1}$ be the standard weight-one spin
$C_4$ lens space.  Let $L_1^1$ denote a circle mapped to a generator
of $C_4$.  Products below use the two factor maps to $BG$.

\begin{proposition}[A spin PSC family with the required triangular edges]
\label{prop:corrected-even-lens-bundles}
Let $N\geq2$, $1\leq i\leq N-1$, and $q=N-i$.  On
$S^{4i+1}\times S^{4q-1}\subset\mathbb C^{2i+1}\times\mathbb C^{2q}$
let the two generators of $C_4\times C_4$ act by
\[
 a(s,z_1,z_2,\ldots,z_{2q})
   =(\mathrm i s,\mathrm i z_1,z_2,\ldots,z_{2q}),\qquad
 b(s,z)=(s,\mathrm i z).
\]
The quotient $\widehat N_{i,q}$ is a closed spin $4N$-manifold with
fundamental group $C_4\times C_4$ and a positive-scalar-curvature metric.
For its universal-cover classifying map $f$, its mod-two fundamental
class has odd--odd coordinates
\[
 v_i=\sum_{j=0}^{2i}\xi_{2j+1,\,4N-(2j+1)}.
\]
More precisely, the pairing with
$x_0x_1z_0^jz_1^k$, $j+k=2N-1$, is one for $0\leq j\leq2i$
and zero for $j>2i$.
\end{proposition}
\begin{proof}
The two transformations commute and have order four.  If $a^rb^s$
fixes a point, its action on the first sphere forces $r=0$ modulo four;
its scalar action on the second sphere then forces $s=0$.  Thus the
action is free.  Both spheres are simply connected.  The quotient has
the claimed fundamental group.  The action is isometric for the product
of round metrics, whose scalar curvature is positive, so that metric
descends.

Let $\chi_0,\chi_1$ be the two standard complex characters.  The
equivariant stable tangent bundle of the product of spheres is the
underlying real bundle of
\[
 (2i+1)\chi_0\ \oplus\ (\chi_0\chi_1)
       \ \oplus\ (2q-1)\chi_1,
\]
minus two trivial real lines.  Its first Chern class is
$(2i+2)U_0+2qU_1$, hence even.  Therefore the quotient is oriented
and its second Stiefel--Whitney class is zero; it is spin.

For completeness the mod-two calculation can be made using successive
equivariant sphere bundles.  Start with
\[
 R=H^*(B(C_4\times C_4);\mathbb F_2)
   =\Lambda(x_0,x_1)\otimes\mathbb F_2[z_0,z_1].
\]
The two sphere representations have mod-two Euler classes
\[
 e_0=z_0^{2i+1},\qquad
 e_1=(z_0+z_1)z_1^{2q-1}.
\]
Multiplication by $e_0$ is injective on $R$.  Multiplication by $e_1$
is injective after quotienting by $e_0$, because it is a monic polynomial
in $z_1$ with coefficients in
$\Lambda(x_0,x_1)\otimes\mathbb F_2[z_0]/(z_0^{2i+1})$.
The Gysin sequences of the two sphere bundles therefore identify
\[
 H^*(\widehat N_{i,q};\mathbb F_2)
   \cong R/(z_0^{2i+1},(z_0+z_1)z_1^{2q-1}).
\]
There are no additional classes from the odd sphere rows, since the
two Euler multiplications are injective.  In degree $4N$ the nonzero
class is $x_0x_1z_0^{2i}z_1^{2q-1}$ and spans the top cohomology.
The relation $z_1^{2q}=z_0z_1^{2q-1}$ reduces
$x_0x_1z_0^jz_1^k$, $j+k=2N-1$, to this class when $j\leq2i$;
if $j>2i$ it is zero.  A nonzero top mod-two class evaluates as one,
which proves the asserted vector.
\end{proof}

For $N\geq2$ define the edge-detecting families
\[
\mathfrak M_{4N}=
\{L_1^{4N-1}\times L_1^1,\,
L_1^1\times L_1^{4N-1},\,
\widehat N_{i,N-i}\mid1\leq i\leq N-1\},
\]
\[
\mathfrak M_{4N+2}=
\{L_1^{4i+3}\times L_1^{4(N-i)-1}\mid0\leq i\leq N-1\}.
\]
Every listed manifold is spin and PSC.  All spin-structure choices
are allowed when forming generated subgroups.

\begin{lemma}[Filtration-zero term of the Dirac orientation]\label{lem:dirac-filtration-zero}
Let \(M^k\) be a closed spin manifold and let \(f:M\to BG\).  Under the
edge homomorphism from the Adams spectral sequence for \(ko_k(BG)\) to
\(H_k(BG;\F_2)\), the filtration-zero component of
\[
D[M,f]\in ko_k(BG)
\]
is \(f_*[M]\) modulo \(2\).
\end{lemma}

\begin{proof}
The Atiyah--Bott--Shapiro orientation \(MSpin\to ko\) is normalized so that
the composite
\[
MSpin\longrightarrow ko\longrightarrow H\F_2
\]
is the ordinary mod-\(2\) Thom orientation.  After applying \(f_*(-)\), the
edge map therefore sends the \(ko\)-fundamental class \(D[M]\) to the ordinary
mod-\(2\) fundamental class \(f_*[M]\).  This is exactly the filtration-zero
Adams component.
\end{proof}

\begin{lemma}[Even filtration-zero detection]\label{lem:even-filtration-zero-detection}\label{theo:detectionhomological}
The mod-\(2\) fundamental classes of the manifolds in \(\mathfrak M_k\) are
linearly independent and give the stated filtration-zero Dirac edges on
the smash summand.  This does not assert generation in higher filtration.
\end{lemma}

\begin{proof}
Write \(\xi_{r,s}\in H_{r+s}(X;\F_2)\), for \(r\) and \(s\) odd,
for the class dual to
\[
x_0z_0^{(r-1)/2}x_1z_1^{(s-1)/2}.
\]
For the standard classifying map \(L_1^{4r-1}\to BC_4\), the pullback of
\(xz^{2r-1}\) is the mod-\(2\) fundamental cohomology class of the lens space.
For \(L_1^1\), the pullback of \(x\) is the fundamental class of the circle.
For products, the assertion is the K\"unneth pairing with the monomial basis of
\(H_*(BC_4\wedge BC_4;\mathbb F_2)\).  In degree \(4N\), the two
boundary products
\[
L_1^{4N-1}\times L_1^1,\qquad L_1^1\times L_1^{4N-1}
\]
pair respectively with
\[
x_0z_0^{2N-1}x_1,\qquad x_0x_1z_1^{2N-1}.
\]
In degree \(4N+2\), the products
\[
L_1^{4i+3}\times L_1^{4(N-i)-1}
\]
pair with the monomials
\[
x_0x_1z_0^{2i+1}z_1^{2(N-i)-1},\qquad 0\leq i\leq N-1.
\]

For the interior degree-\(4N\) classes,
Proposition~\ref{prop:corrected-even-lens-bundles} gives the triangular vectors
\[
v_i=\sum_{0\leq j\leq 2i}\xi_{2j+1,4N-(2j+1)}.
\]
Together with the two boundary vectors, these are triangular with respect to
the ordered odd-exponent monomial basis.  Hence they are linearly independent
and identify the displayed filtration-zero generators of the block
decomposition in Lemma~\ref{lemma:A1decomp-smash-firstprinciples}.
In degrees \(8d+2\), the single periodic assembly class is
the one-cell coefficient class \(\mu^d p\), not one of these filtration-zero
homology generators.
\end{proof}

\begin{remark}[Why initial edges do not give a subgroup count]
Independent filtration-zero edges need not have independent higher
$h_0$-multiples.  For example, in $\mathbb Z/4\{a\}\oplus
\mathbb Z/2\{c\}$, the edges of $a$ and $a+c$ can be independent,
but their doubles coincide.  Therefore no order lower bound below
is inferred merely by counting a fixed number of powers of two
over each edge.  Higher-filtration generation requires its own proof.
\end{remark}

\section{The exceptional dimensions five and six}\label{sec:low}

 Give \(BC_4\) its standard CW model with one cell \(e_j\) in every dimension
\(j\geq0\).  Its integral cellular chain complex is
\begin{equation}\label{eq:C4-periodic-complex}
 C_j(BC_4;\mathbb{Z})=\mathbb{Z}\{e_j\},\qquad
 \partial e_{2r}=4e_{2r-1},\qquad \partial e_{2r+1}=0.
\end{equation}
This is the cellular form of the usual two-periodic resolution of a cyclic
group; see \cite[Chapter~VI, \S2]{browncohomologygroups}.  It gives
\[
 H_j(BC_4;\mathbb{Z})=
 \begin{cases}
   \mathbb{Z},&j=0,\\
   \mathbb{Z}/4,&j>0\text{ odd},\\
   0,&j>0\text{ even}.
 \end{cases}
\]

Let \(T\in H^1(BC_4;\mathbb{Z}/4)\) be the identity character.  Denote by 
\(\beta_4\) the integral Bockstein homomorphism
\[
 \beta_4:H^q(-;\mathbb{Z}/4)\longrightarrow H^{q+1}(-;\mathbb{Z})
\]
associated to
\begin{equation}\label{eq:mod-four-bockstein-sequence}
 0\longrightarrow\mathbb{Z}\xrightarrow{\,4\,}\mathbb{Z}
 \longrightarrow\mathbb{Z}/4\longrightarrow0.
\end{equation}
Put \(U=\beta_4(T)\) and let \(u\in H^2(BC_4;\mathbb{Z}/4)\) be its reduction
modulo four.  Applying cochains to
\eqref{eq:C4-periodic-complex} gives
\begin{equation}\label{eq:C4-mod-four-ring}
 H^*(BC_4;\mathbb{Z}/4)
 =\mathbb{Z}/4[u]\oplus T\,\mathbb{Z}/4[u],\qquad
 |T|=1,\quad |u|=2,\quad T^2=2u.
\end{equation}
Indeed, the mod-four cellular differential is zero, division of the
integral lift of \(dT\) by \(4\) gives an integral representative of \(U\);
reducing it gives the even-cell cochain \(u\).  The standard cellular
diagonal gives \(T^2=2u\). This is also the standard cyclic-group cup-product
calculation in \cite[Chapter~VI, \S7]{browncohomologygroups}.

Write \(T_x,T_y,u_x,u_y\) for the pullbacks in mod-four cohomology  of \(T,u\) to the  first, respectively  second  factor  of 
\(BG=Y\times Y\). 
 Consider  $\tau_{13}$ the  element  of  order  four 
corresponding 
with the summand
\[
 \operatorname{Tor}\bigl(H_1(Y;\mathbb{Z}),H_3(Y;\mathbb{Z})\bigr)\leq H_{5}(G, \mathbb{Z}).
\]
and  \(\tau_{31}\) associated  with 
\[
 \operatorname{Tor}\bigl(H_3(Y;\mathbb{Z}),H_1(Y;\mathbb{Z})\bigr)\leq H_5(BG;\mathbb Z).
\]
 
 In  the Cartan model for  $H_{*}(BC_{4}\times  C_{4},  \mathbb{Z})$ the  following  correspondences  hold,  where $e_{i}$   correspond  to the   cell generator for  the  integral  homology  of  the  usual  model  for  $B{C_{4}}$ 
 \begin{align}
 \sigma_x&=e_5\otimes e_0,&
 \tau_{13}&=e_2\otimes e_3+e_1\otimes e_4,
 \label{eq:tau13-cycle}\\
 \tau_{31}&=e_4\otimes e_1+e_3\otimes e_2,&
 \sigma_y&=e_0\otimes e_5.\nonumber
\end{align}
\subsection{Dimension 5}
\begin{equation}\label{eqn:5dimensional-homology}
    H_5(BG;\mathbb{Z})=
 (\mathbb{Z}/4)\{\sigma_x,\tau_{13},\tau_{31},\sigma_y\},
\end{equation}
 
\begin{proposition}[Description  of  degree 5  $ko$-homology ]\label{prop:ko5-exact}
There are elements \(a,b,t\in\widetilde{ko}_5(X)\) such that
\begin{equation}\label{eq:ko5-presentation}
 \widetilde{ko}_5(X)
 =\langle a,b,t\mid 2a=2b=t,\ 2t=0\rangle
 \cong\mathbb{Z}/4\oplus\mathbb{Z}/2.
\end{equation}
Their homological AHSS edges are
\[
 e(a)=2\tau_{13},\qquad e(b)=2\tau_{31},\qquad e(t)=0,
\]
where
\[
 H_5(X;\mathbb{Z})\cong
 (\mathbb{Z}/4)\{\tau_{13}\}\oplus(\mathbb{Z}/4)\{\tau_{31}\}.
\]
\end{proposition}

\begin{proof}
The orientation \(M\Spin\to ko\) is an isomorphism in total degrees at most
seven \cite{andersonbrownpeterson}.  In the spin AHSS, the differential out of
\(H_5(X;\mathbb{Z})\) is dual to \(Sq^2\).   
Its kernel is
\[
 E^\infty_{5,0}
 =\mathbb{F}_2\{2\tau_{13},2\tau_{31}\}.
\]
There is one further total-degree-five survivor
\[
 E^\infty_{4,1}=\mathbb{F}_2\{t\},
\]
dual to \(z_0z_1\); the other terms vanish.  Equivalently, the low Adams page
has the two classes
\[
\alpha_0=vx_0z_1+h_0z_0x_1z_1,\qquad
\beta_0=vz_0x_1+h_0x_0z_0z_1
\]
and the relation \(h_0\alpha_0=h_0\beta_0\).  Multiplication by \(h_0\)
is doubling.  More explicitly, the real $d_2$ calculation gives
$E_3(5,1)=\mathbb F_2\{\alpha_0,\beta_0\}$ and
$E_3(5,2)=\mathbb F_2\{h_0\alpha_0\}$, with all other
filtrations zero in stem five.  Stem four has no $E_3$ group in
filtration at least three.  Thus every outgoing $d_r$, $r\geq3$,
from these degree-five classes has zero target, and any incoming
one would have negative source filtration.  These three bits are
therefore permanent.  There is no filtration-three subgroup in
which a further hidden doubling could lie.  Hence lifts \(a,b\)
of the two top edges satisfy
\(2a=2b=t\).  This proves \eqref{eq:ko5-presentation}.
\end{proof}

Put \(d=a+b\).  Then \(d\) has order two and diagonal edge in the homological Atiyah--Hirzebruch spectral sequence  
\[
e(d)=2\tau_{13}+2\tau_{31}.
\]

%,The same argument applies to \(\tau_{31}\).

\begin{lemma}\label{lemma:classifying-map-5-manifold}
The following  assertions  hold  in the  cohomology  of $BG$
\begin{enumerate}
\item With coefficients reduced modulo four, the reduction  of  the monomials
\begin{equation}\label{eq:four-detectors}
 T_xu_x^2,\qquad T_xu_y^2,\qquad u_x^2T_y,\qquad T_yu_y^2
\end{equation}
are  in  bijective  correspondence  with  the  algebraic  duals   in this order of  the four elements in \eqref{eqn:5dimensional-homology}. 
\item For every closed oriented 
five-manifold, and  every  map  \(f:Z\to BG\), in    mod  four homology, 
\begin{equation}\label{eq:detector-coordinate-vector}
\begin{split}
 f_*[Z]_{4}={}&
 \langle f^*(T_xu_x^2),[Z]_4\rangle\,\sigma_x
 +\langle f^*(T_xu_y^2),[Z]_4\rangle\,\tau_{13}\\
 &+\langle f^*(u_x^2T_y),[Z]_4\rangle\,\tau_{31}
 +\langle f^*(T_yu_y^2),[Z]_4\rangle\,\sigma_y .
\end{split}
\end{equation}
Here \([Z]_4\) denotes the mod-four reduction of the oriented fundamental
class.

\item Let \(m:Y\times Y\to Y\) be addition in \(C_4\), and let
\[
 \bar m:\Sigma^\infty(Y\wedge Y)\longrightarrow\Sigma^\infty Y
\]
be its stable cross-effect: under the stable splitting of \(Y\times Y\),
it is \(m-\operatorname{pr}_x-\operatorname{pr}_y\) on the smash summand.
Write \(\sigma=e_5\in H_5(Y;\mathbb{Z})\).  Since \(T\) is the identity character
and \(u\) is the reduction of its Bockstein,
\[
 m^*T=T_x+T_y,\qquad m^*u=u_x+u_y .
\]
After removing the two pure-factor terms, expansion gives
\begin{equation}\label{eq:multiplication-cross-effect}
 \bar m^*(Tu^2)
 =T_xu_y^2+u_x^2T_y
  +2T_xu_xu_y+2u_xT_yu_y .
\end{equation}
On \(\tau_{13}=e_2\otimes e_3+e_1\otimes e_4\), the first term in
\eqref{eq:multiplication-cross-effect} evaluates as \(1\) on
\(e_1\otimes e_4\), while the last evaluates as \(2\) on
\(e_2\otimes e_3\); the other two bidegrees do not pair.  Thus the total is
\(3\) modulo four.  Interchanging the factors gives the same calculation
for \(\tau_{31}\).  Naturality of the Kronecker pairing therefore yields
\begin{equation}\label{eq:multiplication-on-tor}
 \bar m_*(\tau_{13})=3\sigma,\qquad
 \bar m_*(\tau_{31})=3\sigma .
\end{equation}

\end{enumerate}
\end{lemma} 

\begin{proof}
   \begin{itemize}
       \item To see this directly, \(Tu^2\) is dual to \(e_5\),
\(T_xu_y^2\) is dual to \(e_1\otimes e_4\), and \(u_x^2T_y\) is dual to
\(e_4\otimes e_1\).  

\item Up  to homotopy,  the  map  can be  suposed  to have  image  in the  $5$-dimensional skeleton.  The result   follows  from  the  multiplicative  structure  in  cohomology,  and the  cup-cap  compatibility of homology and  Poincar\'e duality. 
   \end{itemize} 
\end{proof}

\begin{proposition}[Exact cyclic-induction image]\label{prop:cyclic-low}
Let \(z\) generate
\(\widetilde{ko}_5(BC_4)\cong \mathbb{Z}/4\), normalized by
\[
 \bar\eta_{\rho_r}(z)=\frac{\epsilon(r)}4,\qquad
 \epsilon(0)=\epsilon(2)=0,\quad
 \epsilon(1)=1,\quad\epsilon(3)=-1.
\]
Normalize the remaining choice of $a,b$ in
Proposition~\ref{prop:ko5-exact} so that
\[
d=a+b=j_{1,1*}(z)|_X.
\]
For \(j_{p,q}:C_4\to C_4^2\), \(1\mapsto(p,q)\), the projections  to  the  smash summand  of  the  decomposition  are
\begin{equation}\label{eq:degree-five-slopes}
\begin{array}{c|c}
(p,q)&j_{p,q*}(z)|_X\\ \hline
(0,1),(1,0)&0\\
(1,1)&d\\
(1,3)&d+t\\
(2,1),(1,2)&t .
\end{array}
\end{equation}
Consequently all cyclic BGS inductions span  the  subgroup 
\[
 P_5=\langle d,t\rangle\cong(\mathbb{Z}/2)^2.
\]
\end{proposition}

\begin{proof}
Let \(U_0,U_1\in H^2(BC_4^2;\mathbb{Z})\) be the canonical integral classes.  Pulling back
the linking detectors \(U_0U_1^2\) and \(U_0^2U_1\) gives the edge formula
\[
 e(j_{p,q*}z)
 =2\bigl(pq^2\tau_{13}+p^2q\tau_{31}\bigr).
\]
Thus the six slopes for  $(p,q)$in \eqref{eq:degree-five-slopes} have either zero or
diagonal top edge.
In particular, $j_{1,1*}(z)|_X$ is one of $d,d+t$ by
Proposition~\ref{prop:ko5-exact}.  Replacing $b$ by $b+t$ preserves
its edge and the relations $2a=2b=t$, and interchanges these two
choices of $d$.  This justifies the stated normalization.

The stable smash eta invariant is the finite difference
\begin{equation}\label{eq:low-cross-effect}
 C_{p,q}(u,v)=
 \frac{\epsilon(pu+qv)-\epsilon(pu)-\epsilon(qv)}4
 \in \mathbb{R}/\mathbb{Z} .
\end{equation}
Evaluation for the fifteen nontrivial characters gives
\[
 C_{0,1}=C_{1,0}=0,\qquad
 C_{2,1}=C_{1,2}=C_{1,1}+C_{1,3}.
\]
The two order-two vectors \(C_{1,1}\) and \(C_{1,3}\) are independent.
They distinguish \(d\) and \(d+t\), and their sum distinguishes \(t\).
This proves both the table and the assertion about \(P_5\).
\end{proof}

\begin{corollary}[The missing coset]\label{cor:missing-coset}
\[
 \widetilde{ko}_5(X)/P_5\cong\mathbb Z/2,
\]
represented by either $a$ or $b$.  Target basis changes act trivially
on this quotient, with the following precise meaning.  If
$\mathcal C_5\subset\widetilde{ko}_5(BG)$ is the subgroup of all
cyclic inductions, including both axes, the stable splitting identifies
\[
 \widetilde{ko}_5(BG)/\mathcal C_5
       \cong\widetilde{ko}_5(X)/P_5.
\]
The group $\operatorname{Aut}(G)$ acts naturally on the quotient on
the left, and hence on the displayed order-two quotient.
\end{corollary}
\begin{proof}
The group presentation and the cyclic-image calculation give the
order-two quotient.  Every automorphism of $G$ takes an induced
cyclic class to another induced cyclic class, so it preserves
$\mathcal C_5$.  Its action on a quotient of order two is necessarily
the identity.  This does not assert that every automorphism of $G$
preserves the chosen stable smash summand itself.
\end{proof}

The next observation explains why a product of a spin three-manifold
and a spin surface cannot supply the missing coset.

 \begin{proposition}[Three-manifold--surface products vanish]
\label{prop:product-zero}
For every closed spin three-manifold \(N\), and  every closed spin surface \(\Sigma\), with the  product spin structure, and  map
\[
 N\times\Sigma\longrightarrow BC_4^2,
\]
the smash class in \(\widetilde{ko}_5(X)\) is zero.  In particular this
applies to every map \(L^3(4)\times T^2\to BC_4^2\).
\end{proposition}

\begin{proof}
A Cartan chain model for integral homology is
\[
 C_*(BC_4^2;\mathbb{Z})=
 \Lambda(e_0,e_1)\otimes\Gamma(f_0,f_1),\qquad
 |e_i|=1,\quad |f_i|=2,\quad df_i=4e_i.
\]
The image of the surface class is a multiple of \(e_0e_1\).  Every
degree-three cycle is a combination of
\[
 e_0f_0,\qquad e_1f_1,\qquad f_0e_1+e_0f_1.
\]
Its product with \(e_0e_1\) is zero because \(e_0^2=e_1^2=0\).
Hence the top edge of the product is zero, and
\ref{prop:ko5-exact} leaves only the possibilities \(0,t\).

For a character \(\rho_{u,v}\), the flat line bundle  on the product is an external
tensor product.  The product formula for the reduced eta invariant is
\[
 \bar\eta(D_{N\times\Sigma,\rho_{u,v}})
 \equiv\operatorname{ind}(D^+_{\Sigma,\rho_{u,v}})
   \bar\eta(D_{N,\rho_{u,v}})=0\pmod{\mathbb Z},
\]
because the Dirac index with coefficients in a flat line over a spin
surface is zero.  The congruence, rather than an equality of real
numbers, allows the kernel correction in reduced eta: the two
surface chiral kernel dimensions agree, so that correction is an
integer.  The eta
vector of \(t\) is \(C_{2,1}\ne0\) by
\ref{prop:cyclic-low}, so the class is zero.
\end{proof}
\begin{proposition}[Lens bundles and mixed subgroups do not repair the gap]
\label{prop:low-near-misses}
Neither of the following constructions reaches the nonzero element of
\(\widetilde{ko}_5(X)/P_5\).
\begin{enumerate}
\item A spin total space $E$ of an \(L^3(4)\)-bundle over \(T^2\), including bundles with
nontrivial central extension or inversion monodromy;
\item induction from \(C_2\times C_4\), \(C_4\times C_2\), or \(C_2^2\).
\end{enumerate}
\end{proposition}

\begin{proof}
For (i), write
\[
1\longrightarrow\langle z\rangle\cong C_4
\longrightarrow\pi_1E\longrightarrow\mathbb{Z}^2\longrightarrow1.
\]
Given \(\varphi:\pi_1E\to C_4^2\), put \(\varphi(z)\) in a primitive cyclic
summand \(\langle w\rangle\), choosing a primitive square root when it has
order two.  The map
\(\mathbb{Z}^2\to C_4^2/\langle w\rangle\cong C_4\) lifts coefficientwise to
\(\mathbb{Z}\).  After a target basis change, \(B\varphi\) therefore factors through
\[
 B(C_4\times\mathbb{Z}).
\]
But
\[
H_5(B(C_4\times\mathbb{Z});\mathbb{Z})=H_5(BC_4;\mathbb{Z})\cong\mathbb{Z}/4
\]
because \(H_4(BC_4;\mathbb{Z})=0\).  Its image is a cyclic-slope edge.
Moreover its coefficient is even: projection to $BC_4$ is a map
from a spin five-manifold, and the homological $d_2$ dual to
$Sq^2(xz)=xz^2$ restricts that edge to $2\mathbb Z/4$.
Such an edge is attained by a cyclic induction in
Proposition~\ref{prop:cyclic-low}; any difference with zero edge is
a multiple of $t$, also in $P_5$.  Changing the target basis preserves
the full cyclic-induction subgroup, whose smash projection is $P_5$;
equivalently it preserves the vanishing of the quotient class in
Corollary~\ref{cor:missing-coset}.

For (ii), the low Adams differential for \(BC_2\wedge BC_4\) is
\[
d_2(y^3z)=h_0^2xy^3.
\]
The only surviving degree-five smash class is the order-two lower-filtration
class \(h_0y^3z\)
\cite[Theorem~3.2.5(3), Table~3.5 and (3.28)]{siegemeyer}.  Under
\[
C_2\times C_4\longrightarrow C_4^2,\qquad(r,s)\longmapsto(2r,s),
\]
it maps to \(t\), since the dual class \(y^2z\) is the pullback of
\(z_0z_1\).  The asymmetric ordinary Tor class has no \(ko\)-lift.
Finally, \(\widetilde{ko}_5(BC_2\wedge BC_2)=0\).
\end{proof}

\subsection{The affine torus-bundle representative}

We specify the bundle and its reference map together.  We then compute
the characteristic numbers on the \emph{closed} five-manifold.  In
particular, no unspecified lift of an absolute class to relative
cohomology enters the computation.

Put $\Lambda=\mathbb Z\{x,y\}$ and
\[
 A=\begin{pmatrix}5&8\\8&13\end{pmatrix}=I+4K,\qquad
 K=\begin{pmatrix}1&2\\2&3\end{pmatrix},\qquad J=-K^{\mathsf T}.
\]
Thus $\det A=1$.  Lattice vectors and covectors are columns, and
$\lambda(v)$ denotes evaluation of a covector on a vector.  Let
\[
 B=(S^1\times S^2)\setminus\coprod_{i=1}^4\mathring D_i^3,
 \qquad y=\binom01,\qquad w=\binom2{-1}.
\]
The balls lie in one framed coordinate region.  Choose a map
$r:(B,\partial B)\to(S^1,*)$ representing the circle factor, and a cut
sphere $S$ dual to it.  The local system $\Lambda_A$ has monodromy $A$
around that circle.  Identify the four outward boundary spheres with the
same oriented $S^2$.

\begin{lemma}\label{lem:affine-euler-data}
There is a class
\begin{equation}\label{eq:e-prime}
 e'\in H^2(B;\Lambda_A),\qquad
 e'|_{\partial_iB}=y,\qquad e'([S])=w.
\end{equation}
\end{lemma}
\begin{proof}
Cut along $S$.  The resulting manifold is a three-sphere with six open
balls removed, and the local system is constant.  Degree-two cohomology
classes on this manifold are determined by their boundary evaluations,
whose signed sum must vanish.  This follows from the exact sequence of
the pair and Poincar\'e--Lefschetz duality.  Assign the six values
$y,y,y,y,w,-Aw$.  They satisfy the required condition because
\begin{equation}\label{eq:local-coeff-compatibility}
 4y+(I-A)w=0,
\end{equation}
as $Kw=y$.  The two cut faces have opposite orientations and their
trivializations differ by $A$, so these data descend to the stated class.
\end{proof}

\begin{definition}\label{def:euleraffinebundle}
The characteristic class of an oriented affine torus bundle with fixed
linear monodromy is its obstruction to a section in
$H^2(B;\Lambda_A)$.  On a two-sphere it is normalized by the connecting
homomorphism $\pi_2(S^2)\to\pi_1(T^2)=\Lambda$.
Let $E\to B$ be the affine torus bundle with this monodromy and
\emph{characteristic class $4e'$}.
\end{definition}
The classification with fixed monodromy is
\cite[Proposition~1.1 and Appendices~A--B]{KahnTorus}.  Each boundary
bundle has Chern vector $4y$ and is therefore $S^1_x\times L_y$, where
$L=L^3(4;1)$ is the circle bundle of Euler number four over $S^2$.
The following classifying-space model fixes the boundary identifications
and reference maps needed below.

\subsubsection{A universal torus fibration}
Write $\mathbb T^2=\mathbb R^2/\Lambda$ and
$\mathcal H=\mathbb T^2\rtimes_A\mathbb Z$.  Its subgroup consisting of
the four-torsion points of $\mathbb T^2$, together with $\mathbb Z$, is
$G\times\mathbb Z$, because $A\equiv I\pmod4$.  Subgroup inclusion gives
the homogeneous-space fibration
\begin{equation}\label{eq:affine-universal-fibration}
 \mathbb T^2/G\longrightarrow
 \mathcal E=B(G\times\mathbb Z)
 \xrightarrow{\ q\ }\mathcal B=B\mathcal H.
\end{equation}
For example, this is the bundle
$E\mathcal H/(G\times\mathbb Z)\to E\mathcal H/\mathcal H$.
Identify $\mathbb T^2/G$ with $\mathbb T^2$ by $[t]\mapsto4t$.
Its connecting map on $\pi_2$ is then multiplication by four.

The space $\mathcal B$ is a model for the mapping torus of
$A:B\mathbb T^2\to B\mathbb T^2$.  Its fiber is
$B\mathbb T^2=K(\Lambda,2)$.  A map to $\mathcal B$ with the prescribed
circle component is consequently classified by $H^2(-;\Lambda_A)$.
Choose $b:B\to\mathcal B$ representing $e'$.  On all four boundary
spheres it can be chosen to be the identical map
$j_y=(\mathrm{constant},j):S^2\to(\mathbb CP^\infty)^2$, where $j$
represents the generator.  Indeed the boundary classes are all $y$;
the required boundary homotopies extend over collars.
The pullback of \eqref{eq:affine-universal-fibration} by $b$ is precisely
$E$, since its characteristic class is $4e'$.

The subgroup $G\times\mathbb Z$ is discrete, so $\mathcal E$ is
aspherical and is homotopy equivalent to $BG\times S^1$, compatibly with
the two group projections.  The resulting map $E\to\mathcal E\to BG$
induces
\begin{equation}\label{eq:E-reference-homomorphism}
 \varphi_E:\pi_1(E)\longrightarrow G=\Lambda/4\Lambda,
 \qquad x\longmapsto(1,0),\quad y\longmapsto(0,1),\quad
 g\longmapsto0.
\end{equation}
Here $g$ is the base-circle loop, chosen using the zero section over the
base one-skeleton.  In particular, the boundary maps are the standard
maps on $S^1_x\times L_y$.

Over $B\mathbb T^2$, the fibration \eqref{eq:affine-universal-fibration}
is the product of the standard circle fibrations
\[
 S^1\longrightarrow BC_4\xrightarrow{q_0}\mathbb CP^\infty,
\]
whose Euler classes are four times the generator.  We use fiberwise
models for these maps.  Pulling $q_0$ back by $j$ gives $L\to S^2$
and its fiberwise classifying map $c_L:L\to BC_4$, so
$q_0c_L=j\pi_L$.

\subsubsection{The cap and the closed manifold}
Let $V$ be the sphere with four open disks removed.  Its boundary
presentation is
\[
 \pi_1(V)=\langle\delta_1,\delta_2,\delta_3,\delta_4
 \mid\delta_1\delta_2\delta_3\delta_4=1\rangle\cong F_3.
\]
Prescribe on each boundary curve the actual degree-one loop in a fiber
of $q_0$.  These maps extend to $h:V\to BC_4$, since the boundary
relation has image $4=0$ in $C_4$.  Thus
\begin{equation}\label{eq:V-map}
 h_*[\delta_i]=1\in C_4\quad(1\leq i\leq4).
\end{equation}
Put $\alpha=q_0h$.  This is a map of pairs
$(V,\partial V)\to(\mathbb CP^\infty,*)$, literally constant on the
boundary.  No choice of a relative Chern-class lift is being made here:
$\alpha$ is the composite of the chosen maps.

\begin{definition}\label{def:M5}
Put
\begin{equation}\label{eq:P-piece}
 P=V\times L,\qquad f_P=h\times c_L:P\longrightarrow BG.
\end{equation}
Use the boundary identifications above and define
\begin{equation}\label{eq:M-affine-definition}
 M=P\cup_{\partial}(-E),\qquad f=f_P\cup f_E:M\longrightarrow BG.
\end{equation}
\end{definition}
These maps agree on the boundary: on every component they are the
degree-one $x$-fiber loop and the fiberwise lens-space map $c_L$.

\begin{lemma}\label{lem:affine-closed-spin}
The manifold $M$ is closed, connected, oriented, and spin.
\end{lemma}
\begin{proof}
A connection splits $TE$ into its vertical oriented plane bundle and
the pullback of $TB$.  The vertical bundle is associated to the linear
monodromy, which factors through $S^1$; every oriented plane bundle over
$S^1$ is trivial.  The oriented three-manifold $B$ is spin.  Hence $E$
is spin.  In the framed region containing the four balls, choose the
same translation clutching model for the vector $4y$ at every hole.
The base spin structure and the constant vertical frame there give the
same product spin structure on all four copies of $S^1\times L$.
Both boundary factors can also be changed simultaneously on all holes:
the two mod-two fiber characters extend over $E$, since $A\equiv I$
and $4e'\equiv0$ modulo two, and twisting the spin structure by them
changes the two factors independently.

Fix the resulting spin structure on $L$.  Four copies of either spin
circle extend over $V$: they have zero total class in
$\Omega_1^{\mathrm{Spin}}\cong\mathbb Z/2$.  More explicitly, the image
of $H^1(V;\mathbb Z/2)\to H^1(\partial V;\mathbb Z/2)$ consists of
the four-tuples of even sum, so one can twist any surface spin structure
to obtain the four desired equal restrictions.  The product spin
structure on $V\times L$ therefore agrees with that on $E$ at every
interface.  Identity gluing to $-E$ is orientation reversing on the
boundaries, as required, and gives the asserted spin manifold.
\end{proof}

\subsubsection{A universal degree-two identity}
In the rest of this subsection all unspecified coefficients are
$R=\mathbb Z/4$.  For $\lambda\in\Lambda^*$ let
$t_\lambda\in H^1(\mathcal E;R)$ be the character of the $G$-factor,
and set $v_\lambda=\operatorname{red}_4\beta_4(t_\lambda)$.
Let $\vartheta$ be the integral circle class of $\mathcal E$, reduced
to $R$ when taking products.  The degree-two restriction
\[
 H^2(\mathcal B;R)\longrightarrow H^2(B\mathbb T^2;R)
\]
is an isomorphism: the Wang sequence has no contribution from
$H^1(B\mathbb T^2;R)$, and $A^*$ is the identity in degree two modulo
four.  Let $\mathfrak b_\lambda$ be the unique class restricting to the
coordinate class $\lambda$ modulo four.

\begin{lemma}\label{lem:affine-universal-degree-two}
In ordinary cohomology of $\mathcal E$,
\begin{equation}\label{eq:affine-universal-degree-two}
 q^*\mathfrak b_\lambda
   =v_\lambda-\vartheta\smile t_{J\lambda}.
\end{equation}
\end{lemma}
\begin{proof}
Restriction to the $BG$ fiber gives $v_\lambda$ on both sides: the
coordinate Chern class of the inclusion $C_4\subset S^1$ is
$\beta_4(t_\lambda)$.  The K\"unneth decomposition for
$\mathcal E\simeq BG\times S^1$ leaves only two possible coefficients
in the difference, namely $\vartheta$ times the two fiber characters.
We determine them on the commuting tori with loops $g$ and $v=x,y$.

The covering-group extension
\[
 1\longrightarrow\Lambda\longrightarrow
 \mathbb R^2\rtimes_A\mathbb Z
 \longrightarrow\mathcal H\longrightarrow1
\]
defines the universal twisted Euler class.  Reducing it modulo four
and applying $\lambda$ gives $\mathfrak b_\lambda$, as their
restrictions to $B\mathbb T^2$ agree.  On the indicated torus, the
flat $\mathcal H$-bundle has holonomies $(0,1)$ and $(v/4,0)$.
They commute in $\mathcal H$, but their displayed lifts have commutator
\[
 [(0,1),(v/4,0)]=((A-I)v/4,0)=(Kv,0).
\]
This is exactly the Euler obstruction obtained by lifting the
trivialization around the boundary of the fundamental square.
Changing either lift changes the answer by $(A-I)\Lambda$, which
vanishes modulo four.  Thus $q^*\mathfrak b_\lambda$ evaluates as
$\lambda(Kv)$.  On the same torus $v_\lambda$ is zero, because the
character lifts integrally, while
$\vartheta t_{J\lambda}$ evaluates as $-\lambda(Kv)$.
This proves the two remaining coefficients and hence the identity.
\end{proof}

The characters on $E$ in this construction are defined by
\eqref{eq:E-reference-homomorphism}, not as reductions of integral
cohomology classes.  In fact
\[
 H_1(E;\mathbb Z)\cong\mathbb Z\oplus(\mathbb Z/4)^2,
 \qquad H^1(E;\mathbb Z)\cong\mathbb Z.
\]
The fiber coinvariants are $\Lambda/(A-I)\Lambda=\Lambda/4\Lambda$
because $K$ is unimodular; the sphere relations from $4e'$ vanish in
these coinvariants.  The free character is the base-circle character.
An integral \emph{cochain} lifting a fiber character need not be closed.

\subsubsection{A global square-zero calculation}
Form the four-dimensional CW pushout
\[
 W=B\cup_{\coprod_i(S^1_i\times S^2)}(V\times S^2),
\]
where $(s,z)$ on the $i$-th boundary component is identified with $z$
on $\partial_iB$.  There is a map $b_W:W\to\mathcal B$ equal to $b$
on $B$ and to $(\alpha,j):V\times S^2\to B\mathbb T^2$ on the
other piece.  They agree on the attaching sets.  There is also a map
$\pi_M:M\to W$, equal to the bundle projection on $E$ and to
$(v,l)\mapsto(v,\pi_L(l))$ on $P$.
The total-space maps constructed above give a commuting diagram
\[
\xymatrix{
 M\ar[r]^F\ar[d]_{\pi_M}&\mathcal E\ar[d]^q\\
 W\ar[r]_{b_W}&\mathcal B
}
\]
Indeed, on $P$ the total-space map is $(h,c_L,\mathrm{constant})$,
whose projection is $(\alpha,j\pi_L)$.  This observation is why $h$
was chosen first and $\alpha=q_0h$ afterwards.

For $i=x,y$, let $\epsilon_i$ be the dual basis covector and abbreviate
$\mathfrak b_{\epsilon_i}$, $t_{\epsilon_i}$, and $v_{\epsilon_i}$
by $\mathfrak b_i$, $t_i$, and $v_i$.
Write $B_i=\pi_M^*b_W^*\mathfrak b_i$.
\begin{lemma}\label{lem:affine-global-square-zero}
The ordinary cohomology classes $B_x,B_y\in H^2(M;R)$ satisfy
$B_x^2=B_y^2=0$.
\end{lemma}
\begin{proof}
Represent $\mathfrak b_i$ by a map
$\rho_i:\mathcal B\to K(R,2)$.  Its restriction to the fiber
$B\mathbb T^2$ is homotopic to reduction modulo four of the $i$-th
coordinate projection.  Since that fiber is a cofibration in the
mapping-torus model, the homotopy extension property makes the
restriction literally that map.

For $i=x$, the composite $\rho_xb_W$ is constant on every boundary
sphere of $B$ and, on $V\times S^2$, is independent of $S^2$.
It therefore factors through the three-dimensional CW complex
\[
 (B/\partial B)\vee(V/\partial V).
\]
The factorization sends $(v,z)$ to the class of $v$ and collapses
the boundary of $B$ to the wedge point.  For $i=y$, the composite on
$V\times S^2$ is independent of $V$ and has the same boundary map
on all four spheres.  It factors through
\[
 B/(\text{identify the four boundary spheres by their fixed maps to }S^2),
\]
also a three-dimensional CW complex; on the second piece the quotient
map sends $(v,z)$ to $z$.  Each coordinate class thus comes from a
three-dimensional space, where its square is zero.  Pulling back to
$M$ proves the assertion.
\end{proof}

\begin{lemma}[Closed-manifold detector calculation]\label{lemma:homological}
For every covector $\lambda$ put $\mathsf T_\lambda=F^*t_\lambda$
and $\mathsf u_\lambda=F^*v_\lambda$, with coordinate shorthand
$\mathsf T_i=f^*T_i$ and $\mathsf u_i=f^*u_i$ for $i=x,y$.
Let $\Theta=F^*\vartheta$.  Then
\[
 \mathsf u_i^2=2\Theta\smile B_i\smile\mathsf T_{J\epsilon_i}
 \quad\text{in }H^4(M;R).
\]
For a fiber covector $\mu$ and $i=x,y$, the resulting characteristic
number is
\begin{equation}\label{eq:general-detector-formula}
 \big\langle\mathsf T_\mu\mathsf u_i^2,[M]_4\big\rangle
 =\pm2\det(\mu,J\epsilon_i)\,\epsilon_i(w)\pmod4.
\end{equation}
The common sign depends only on the orientation ordering convention
and has no effect on a multiple of two modulo four.
\end{lemma}
\begin{proof}
Pull \eqref{eq:affine-universal-degree-two} back to $M$ and square it.
The class $\Theta$ comes from a circle and hence has square zero.
The preceding lemma gives $B_i^2=0$.  Ordinary graded commutativity
therefore gives the asserted identity on the closed manifold.

The circle map is constant on $P$.  It specifies a relative class in
$H^1(M,P;R)$ restricting, by excision, to the pullback of
$r^*\vartheta\in H^1(B,\partial B;R)$ on $E$.  On $E$, $B_i$ is
the pullback of $\epsilon_i(e'\bmod4)$.  Integration over the oriented torus
fiber gives
\[
 (p_E)_!(\mathsf T_\mu\mathsf T_{J\epsilon_i})
   =\det(\mu,J\epsilon_i).
\]
This follows simply by restricting the two characters to the fiber,
where they lift integrally and their cup product pairs with the
oriented torus by the determinant.  Finally, the cut sphere $S$ is
dual to $r^*\vartheta$, so
\[
 \big\langle r^*\vartheta\smile\epsilon_i(e'),[B,\partial B]_4\big\rangle
 =\epsilon_i(w)\pmod4.
\]
Excision, the projection formula, and the orientation $-E$ in $M$
give \eqref{eq:general-detector-formula}.  All relative classes used
here are specified by the actual circle map; no relative lift of
$\mathsf u_i^2$ is chosen.
\end{proof}

\begin{corollary}\label{corollary:torus-evaluations}
The four global characteristic numbers of $M$, in the detector order
fixed above, are $(0,2,0,0)$.  Explicitly,
\[
\begin{aligned}
 \langle\mathsf T_x\mathsf u_x^2,[M]_4\rangle&=0,
 &\langle\mathsf T_x\mathsf u_y^2,[M]_4\rangle&=2,\\
 \langle\mathsf u_x^2\mathsf T_y,[M]_4\rangle&=0,
 &\langle\mathsf T_y\mathsf u_y^2,[M]_4\rangle&=0.
\end{aligned}
\]
\end{corollary}
\begin{proof}
The dual basis covectors satisfy
$J\epsilon_x=(-1,-2)$, $J\epsilon_y=(-2,-3)$,
$\epsilon_x(w)=2$, and $\epsilon_y(w)=-1$.
The four expressions in \eqref{eq:general-detector-formula} are
$2(-2)(2)$, $2(-3)(-1)$, $2(1)(2)$, and $2(2)(-1)$.
Modulo four these are $0,2,0,0$, independently of the common sign.
\end{proof}

\begin{proposition}[Identification of the ordinary class]
\label{prop:chain-identification}
The closed mapped spin manifold in \eqref{eq:M-affine-definition}
satisfies
\begin{equation}\label{eq:affine-fundamental-class}
 f_*[M]=2\tau_{13}\in H_5(BG;\mathbb Z).
\end{equation}
\end{proposition}
\begin{proof}
The four detectors of Lemma~\ref{lemma:classifying-map-5-manifold}
give the coordinates of the reduction of the integral fundamental
class.  On the integral subgroup $H_5(BG;\mathbb Z)\cong(\mathbb Z/4)^4$
these coordinates determine the class.  Apply
Corollary~\ref{corollary:torus-evaluations}.
\end{proof}

\begin{proposition}[Realization of the missing coset]\label{prop:missing-coset}
The smash projection of $D[M,f]$ represents the nonzero element of
$\widetilde{ko}_5(X)/P_5\cong\mathbb Z/2$.
\end{proposition}
\begin{proof}
Its homological AHSS edge is $2\tau_{13}$.  By
Proposition~\ref{prop:ko5-exact}, the class is $a$ or $a+t$;
by Corollary~\ref{cor:missing-coset}, either has the required quotient
class.  If the particular lift $a$ is needed, adding a known PSC
representative of $t$ changes $a+t$ to $a$.  This last conclusion uses
the independent degree-five $ko$ presentation and cyclic-image
calculation; the preceding ordinary homology computation does not
assume them.
\end{proof}

\subsection{Positive scalar curvature on the affine representative}
\label{sec:psc}

This section constructs the metric used on the five--manifold.  We give the
details because two familiar shortcuts are not available here.  The torus
monodromy is hyperbolic, so it does not preserve any one flat metric on the
fiber, and the boundary metric must be prescribed exactly in order to glue
the two pieces.  The construction nevertheless works because the space of
flat metrics of fixed area is connected, the change of fiber metric can be
made arbitrarily slow, and the curvature of an affine connection becomes
small after the torus fibers are shrunk.

We use the following topological data constructed in the preceding section.
Let
\[
 B=\bigl(S^1\times S^2\bigr)\setminus
       \coprod_{i=1}^{4}\mathring D_i^3
\]
and let \(\pi:E\to B\) be the oriented affine \(T^2\)-bundle whose linear
monodromy around the \(S^1\)-factor is
\[
 A=\begin{pmatrix}5&8\\8&13\end{pmatrix}\in \operatorname{SL}(2,\mathbb Z).
\]
We denote the standard integral basis of the fiber lattice by \(x,y\).  The
affine Euler class has restriction \(4y\) to each of the four physical
boundary spheres.  Consequently every boundary component is canonically of
the form
\begin{equation}\label{eq:psc-boundary-product}
 \pi^{-1}(\partial D_i^3)\cong
 S^1_x\times L^3(4;1)_y .
\end{equation}
Here \(L^3(4;1)\to S^2\) is the oriented principal circle bundle with first
Chern class \(4\).  The four identifications in
\eqref{eq:psc-boundary-product} are chosen using the same local trivialization.  We write
\(L=L^3(4;1)\).

Let \(V\) be a sphere with four open discs removed.  Orient its boundary
circles so that there is a map
\[
 c_V:V\longrightarrow BC_{4,x}\]sending every oriented boundary circle to \(1\in C_4\).  Such a map exists:
the single relation among the four oriented boundary curves is that their
sum is zero, and its image is \(4=0\) in \(C_4\).  With \(c_L:L\to BC_{4,y}\)
the classifying map of the universal \(C_4\)-cover, put
\[
 P=V\times L,
 \qquad
 M=P\cup_{\partial}(-E).
\]
The identity maps in \eqref{eq:psc-boundary-product} are used for the gluing.  We shall
prove that both pieces carry positive-scalar-curvature metrics with exactly
the same product metric on every boundary component.
\subsubsection{A long positive-scalar-curvature metric on the base}

We first prepare the base.  A \emph{torpedo metric} on a disc is a rotationally
symmetric positive-scalar-curvature metric which is a round product cylinder
near the boundary.  It is the local model in the Gromov--Lawson surgery
construction \cite[Theorem~A]{gromovlawsonsurgery}.  A detailed modern account,
including the elementary warped-product scalar-curvature calculation, is
given in \cite[\S2]{EbertFrenckSurgery}; see also Gajer's boundary version
\cite{GajerBoundary}.

\begin{lemma}[The metric on the base]\label{lem:long-base-metric}
For every sufficiently large \(R\), the manifold \(B\) has a metric
\(g_{B,R}\) with the following properties.
\begin{enumerate}
\item There is a constant \(\kappa>0\), independent of \(R\), such that
\[
 \operatorname{scal}(g_{B,R})\geq \kappa.
\]
\item A product region \([0,R]\times S^2\) represents the generator of the
\(S^1\)-factor.  Increasing \(R\) only lengthens this product region.
\item Near every boundary component the metric is
\[
 dt^2+g_{S^2(r)},
\]
where \(g_{S^2(r)}\) is the round metric of radius \(r\).  The four collars
are isometric, including their chosen framings.
\end{enumerate}
The radius \(r>0\) can be chosen as small as desired before \(R\) is chosen.
\end{lemma}

\begin{proof}
Start with the product of a circle and a round two--sphere.  Its scalar
curvature is the scalar curvature of the sphere and is therefore positive.
Replace the metric in four disjoint, framed balls contained in one fixed
coordinate region by four identical Gromov--Lawson torpedo models, and then
remove the central discs.  This produces the asserted round product collars.
The modification takes place in a fixed compact region and is independent of
the length of the circle.  On a complementary product interval we may replace
\(ds^2\) by the metric on an interval of any length \(R\); this does not change
the scalar curvature of the product.  The minimum of the scalar curvature on
the fixed torpedo region and on the product region is therefore a positive
number \(\kappa\) independent of \(R\).  The radius in the cylindrical part
of a torpedo may be prescribed to be any sufficiently small \(r\); this is
immediate either from the warped metric formula or by rescaling the standard
torpedo model.
\end{proof}
\subsubsection{The fiber metrics and the affine connection}

Let \(\Lambda_\rho\to B\) be the rank--two integral local system defined by
the linear monodromy, and put
\(\mathcal V=\Lambda_\rho\otimes\mathbb R\).  Since
\(\det A=1\), the determinant line \(\det\mathcal V\) has a preferred parallel
orientation and volume form. 

We  will   require a
Euclidean metric on \(\mathcal V\) to have constant determinant.

Choose a positive definite matrix \(H_0\) of determinant one, diagonal in the
basis \(x,y\).  The space
\[
 \mathcal P_1=\{H=H^{\mathsf T}>0:\det H=1\}
\]
is connected (indeed it is the symmetric space
\(\operatorname{SL}(2,\mathbb R)/\operatorname{SO}(2)\)).  Choose a smooth
path \(\overline H:[0,1]\to\mathcal P_1\) from \(H_0\) to
\(A^*H_0=A^{\mathsf T}H_0A\), constant to all orders near both endpoints.
On the long product region set
\[
 H_R(s)=\overline H(s/R),\qquad 0\leq s\leq R,
\]
using the endpoint identification by \(A\).  With the corresponding
convention for that identification, the equation
\(H_R(R)=A^*H_R(0)\) says that this family descends to the bundle.
Extend it by \(H_0\) over the fixed region containing the four deleted balls.
There are constants \(C_1,C_2\), independent of \(R\), such that
\begin{equation}\label{eq:H-derivative-bounds}
 \lvert\nabla H_R\rvert\leq \frac{C_1}{R},
 \qquad
 \lvert\nabla^2H_R\rvert\leq \frac{C_2}{R^2}.
\end{equation}
All four boundary restrictions are the same diagonal flat metric \(H_0\).

We next choose an affine connection.  On a local trivialization of \(E\), it
is represented by a pair of one--forms \(\alpha=(\alpha_x,\alpha_y)\), and
the translation-invariant vertical one--forms are
\[
 \theta=(\theta_x,\theta_y)=dz+\alpha.
\]
On overlaps these expressions transform by the affine transition functions,
so the curvature \(F=d_\nabla\alpha\) is a globally defined
\(\mathcal V\)-valued two--form.  Affine torus bundles with fixed monodromy
are classified by \(H^2(B;\Lambda_\rho)\); see, for example,
\cite[\S1]{KahnTorus}.  The relative Euler-class calculation in the
topological part of the proof identifies each boundary restriction with the
bundle of Chern vector \(4y\).  Choose the standard connection on those four
boundary bundles and pull it back over their collars.  A global affine
connection can be changed to these prescribed collar connections: the
difference of two connections is a \(\mathcal V\)-valued one--form, so a
collar cutoff and a partition of unity provide the required change.  We may
therefore arrange that the connection is independent of the normal
coordinate on each boundary collar and has
\begin{equation}\label{eq:boundary-connection}
 \theta_x=d z_x,
 \qquad
 \theta_y=dz_y+\alpha_4,
 \qquad
 \left[\frac{d\alpha_4}{2\pi}\right]=4[S^2]^*.
\end{equation}
The same \(\alpha_4\) is used on all four collars.

Only a uniform curvature bound is required.  To see that it can be arranged,
choose the connection once on the base with a unit-length product tube.
Choose a diffeomorphism to this base which linearly rescales the central
part of the length-\(R\) tube and is the identity on the fixed boundary
region.  Interpolate smoothly on fixed-length end collars, with first
derivative bounded independently of \(R\).  This diffeomorphism is
isotopic to the identity relative to the fixed boundary region and
therefore lifts to an isomorphism of the affine bundle, using the
identification of flat local systems induced by the isotopy.  Pull the
connection back by this lift.  Pullback of its curvature depends only
on the first derivative of the diffeomorphism.  Moreover, the fiber
metrics \(H_R\) range over the fixed compact path \(\overline H\), so
their norms are uniformly equivalent.  Hence there is a constant
\(C_F\), independent of \(R\), for which
\begin{equation}\label{eq:uniform-curvature}
 \lvert F\rvert_{g_{B,R},H_R}\leq C_F.
\end{equation}
Notice that we do \emph{not} assume that \(F\) is supported away from the
long tube.  In fact, the component of the Euler class evaluated on a dual
two--sphere can force nonzero flux through its cross-sections.  The uniform
bound \eqref{eq:uniform-curvature}, rather than a support assertion, is what
the scalar-curvature argument uses.

 For \(\varepsilon>0\), define the connection metric
\begin{equation}\label{eq:connection-metric}
 g_{E,R,\varepsilon}
 =\pi^*g_{B,R}+\varepsilon^2
   \sum_{a,b\in\{x,y\}}(H_R)_{ab}\,\theta_a\theta_b.
\end{equation}
Thus \(\pi:(E,g_{E,R,\varepsilon})\to(B,g_{B,R})\) is a Riemannian
submersion, and each fiber is a flat torus.  We record the estimate carefully.

\begin{lemma}[Adiabatic estimate]\label{lem:affine-adiabatic-estimate}
There are constants \(C_T,C_A>0\), independent of \(R\) and
\(\varepsilon\), such that
\begin{equation}\label{eq:scalar-lower-bound}
 \operatorname{scal}(g_{E,R,\varepsilon})
 \geq
 \operatorname{scal}(g_{B,R})-\frac{C_T}{R^2}
                                  -C_A\varepsilon^2.
\end{equation}
Consequently \(g_{E,R,\varepsilon}\) has positive scalar curvature when
\(R\) is sufficiently large and, after \(R\) has been fixed,
\(\varepsilon\) is sufficiently small.
\end{lemma}

\begin{proof}
We recall the tensors entering the submersion formula.  The tensor
\(T\) is the second fundamental form of the fibers, the tensor \(\mathcal A\)
measures the failure of the horizontal distribution to be integrable, and
\(N\) is the mean-curvature vector of a fiber.  Tracing O'Neill's fundamental
equations \cite{ONeill} gives
\begin{equation}\label{eq:oneill-scalar}
 \operatorname{scal}(E)
 =\pi^*\operatorname{scal}(B)+\operatorname{scal}(F_b)
  -\lvert\mathcal A\rvert^2-\lvert T\rvert^2-\lvert N\rvert^2
  -2\,\delta_BN .
\end{equation}
This normalization is also stated as
\cite[Corollary~9.37]{Besse}.  We now estimate every term in this
identity.

The fiber scalar curvature is zero.  If \(X\) is horizontal, the Koszul
formula gives
\[
 2\langle T_UV,X\rangle
   =-(\nabla_Xg^{\mathrm v})(U,V)
\]
for translation-invariant vertical fields \(U,V\).  Taking the vertical
trace shows
\[
 \langle N,X\rangle
 =-\frac12\operatorname{tr}(H_R^{-1}\nabla_XH_R)
 =-\frac12X\log\det H_R=0.
\]
Thus the fibers are minimal: \(N=0\), and both the mean-curvature square and
the divergence term in \eqref{eq:oneill-scalar} vanish pointwise.  Moreover,
the same Koszul formula and \eqref{eq:H-derivative-bounds} give
\[
 \lvert T\rvert^2\leq C\lvert H_R^{-1}\nabla H_R\rvert^2
 \leq \frac{C_T}{R^2}.
\]
This term is independent of the common fiber scale \(\varepsilon\).

For horizontal vectors \(X,Y\), O'Neill's integrability tensor is
\[
 \mathcal A_XY=-\frac12F(X,Y)^{\#},
\]
where \(F(X,Y)^{\#}\) is the vertical fundamental vector determined by the
connection curvature.  Its length in \eqref{eq:connection-metric} has an
extra factor \(\varepsilon\).  The uniform bound
\eqref{eq:uniform-curvature} therefore implies
\[
 \lvert\mathcal A\rvert^2\leq C_A\varepsilon^2.
\]
Substitution into \eqref{eq:oneill-scalar} proves
\eqref{eq:scalar-lower-bound}.  Finally choose \(R\) so that
\(C_T/R^2<\kappa/3\), and then choose \(\varepsilon\) so that
\(C_A\varepsilon^2<\kappa/3\).  The resulting scalar curvature is at least
\(\kappa/3>0\).
\end{proof}

This proof explains why hyperbolicity of \(A\) causes no metric obstruction.
It prevents a single flat metric from being invariant, but it does not
prevent a fixed-determinant path from \(H_0\) to \(A^*H_0\).  Stretching that
path makes the second fundamental form small.  It is also important that
\(A\in\operatorname{SL}(2,\mathbb Z)\): constant determinant makes the fiber
mean curvature vanish, eliminating the divergence term in
\eqref{eq:oneill-scalar} pointwise.
\subsubsection{The boundary metric and the lens-space factor}

Because the base metric, \(H_R\), and the connection are independent of the
normal coordinate on a boundary collar, \eqref{eq:connection-metric} is
itself a product there.  Write
\(H_0=\operatorname{diag}(h_x^2,h_y^2)\).  Equations
\eqref{eq:boundary-connection} and \eqref{eq:connection-metric} give the
induced boundary metric
\begin{equation}\label{eq:common-boundary-metric}
 h_\partial
 =\varepsilon^2h_x^2\,dz_x^2
  +\bigl(g_{S^2(r)}+\varepsilon^2h_y^2(dz_y+\alpha_4)^2\bigr).
\end{equation}
The first summand is a flat circle metric.  The expression in parentheses is
the Berger connection metric \(h_L\) on \(L^3(4;1)\).  Hence
\[
 (\partial_iE,h_\partial)
 \cong (S^1_x,\varepsilon^2h_x^2dz_x^2)\times(L,h_L)
\]
isometrically for all four \(i\).

For clarity, positivity of the Berger factor does not require a special
spectral computation.  The circle-submersion case of O'Neill's formula is
the  identity
\begin{equation}\label{eq:berger-scalar}
 \operatorname{scal}(h_L)
 =\frac{2}{r^2}
  -c_0\varepsilon^2h_y^2
       \lvert d\alpha_4\rvert_{g_{S^2(r)}}^2.
\end{equation}
Here \(c_0>0\) is the universal constant determined by the convention for
the norm of a two--form (it is \(1/4\) if
\(\lvert F\rvert^2=\sum_{i,j}F_{ij}^2\), and \(1/2\) if the sum is only over
\(i<j\)).  Its value is immaterial; only the factor \(\varepsilon^2\) is used.
For fixed \(r\), this is positive when \(\varepsilon\) is small.  It can be
made as large as desired by first taking \(r\) small and then taking
\(\varepsilon=o(r)\); for example \(\varepsilon=O(r^2)\) suffices.  This
choice is compatible with Lemma~\ref{lem:affine-adiabatic-estimate}.

We also need to control the other gluing piece while the \(x\)-circle becomes
small.  The following elementary observation avoids any hidden scaling
problem.

\begin{lemma}[A uniformly controlled four-holed surface]
There is a constant \(C_V\) such that, for every sufficiently small
\(\ell>0\), the four-holed sphere \(V\) has a metric \(g_{V,\ell}\) satisfying
\[
 \operatorname{scal}(g_{V,\ell})\geq-C_V
\]
and having four isometric product collars whose boundary circles all have
circumference \(\ell\).
\end{lemma}

\begin{proof}
Give the four-punctured sphere a complete hyperbolic metric.  In a cusp the
metric has the form \(dt^2+c^2e^{-2t}d\theta^2\).  Truncating sufficiently far
out produces a boundary circle of any prescribed small length \(\ell\).
On a collar, replace the warping function by
\((\ell/2\pi)\varphi(t)\), where \(\varphi\) is one fixed positive smooth
function, equal to \(1\) near the new boundary and equal to the relevant
exponential function near the other end of the collar.  For a warped surface
metric \(dt^2+f(t)^2d\theta^2\), the Gaussian curvature is \(-f''/f\).
The factor \(\ell/2\pi\) cancels from this quotient.  Thus the curvature of
the modified collars has a lower bound independent of \(\ell\), while the
unmodified part has Gaussian curvature \(-1\).  Since scalar curvature on a
surface is twice Gaussian curvature, one obtains the asserted uniform
constant \(C_V\).
\end{proof}

Set \(\ell=2\pi\varepsilon h_x\), so that the boundary circles of
\((V,g_{V,\ell})\) agree with the first factor in
\eqref{eq:common-boundary-metric}.  Choose \(r\) and \(\varepsilon\) in
\eqref{eq:berger-scalar} so that
\(\operatorname{scal}(h_L)>C_V\).  Scalar curvature is additive for product
metrics, and consequently
\begin{equation}\label{eq:cap-scalar}
 \operatorname{scal}(g_{V,\ell}+h_L)
 =\operatorname{scal}(g_{V,\ell})+\operatorname{scal}(h_L)>0.
\end{equation}
Thus \(P=V\times L\) has a positive-scalar-curvature metric, product near its
boundary, whose four boundary metrics are exactly \(h_\partial\).
The parameters can be selected without circularity: first fix the uniform
constant \(C_V\); next choose \(r\) so that \(2/r^2>C_V+1\); then choose
\(R\) large enough for the \(T\)-term in
\eqref{eq:scalar-lower-bound}; and finally choose \(\varepsilon\) small
enough both for its \(\mathcal A\)-term and for
\eqref{eq:berger-scalar}.  Shrinking \(\varepsilon\) only shortens the four
boundary circles, and the surface lemma was uniform in that length.
  \subsubsection{Spin structures and smooth gluing}

We next verify that the metric gluing is also a spin gluing.  The vertical
tangent bundle \(T^{\mathrm v}E\) is the oriented real plane bundle associated
to the derivative representation of the monodromy.  That representation
factors through the base circle.  Every oriented plane bundle over \(S^1\)
is trivial, since such bundles are classified by their Euler class in
\(H^2(S^1;\mathbb Z)=0\).  Hence \(T^{\mathrm v}E\) is trivial as an oriented
plane bundle.  The oriented three--manifold \(B\) is spin.  For completeness,
double \(B\) along its boundary.  On the resulting closed oriented
three--manifold, the Wu formula gives \(w_2=v_2\), while the defining pairing
for \(v_2\) is zero because \(\operatorname{Sq}^2\) vanishes on degree-one
classes by the instability axiom.  Poincar\'e duality gives \(v_2=0\), and
restriction back to \(B\) gives \(w_2(TB)=0\).  A connection gives
\[
 TE\cong T^{\mathrm v}E\oplus\pi^*TB,
\]
and therefore \(w_2(E)=0\).

Choose the base spin structure and the vertical trivialization compatibly in
the framed region containing the four deleted balls.  The four restrictions
to \(S^1_x\times L\) are then identical under the identifications
\eqref{eq:psc-boundary-product}.  Any spin structure on this product is the product of a
spin structure \(\sigma_x\) on \(S^1_x\) and a spin structure \(\sigma_L\)
on \(L\): after choosing one product spin structure, this follows from
\[
 H^1(S^1\times L;\mathbb Z/2)
 \cong H^1(S^1;\mathbb Z/2)\oplus H^1(L;\mathbb Z/2).
\]
The four copies of \(\sigma_x\) extend over \(V\).  To spell out the
obstruction-theoretic point, choose any spin structure \(\sigma_0\) on
\(V\).  Both the boundary of \(\sigma_0\) and the desired four equal spin
circles have even parity.  Their difference is therefore in the image of
\[
 H^1(V;\mathbb Z/2)\longrightarrow H^1(\partial V;\mathbb Z/2)
\]
because that image consists precisely of the four-tuples with even sum.
Twisting \(\sigma_0\) by a preimage gives the required extension.
Equivalently, four copies of either spin circle represent zero in
\(\Omega^{\mathrm{Spin}}_1\cong\mathbb Z/2\).
This is the boundary version of the standard classification of spin
structures on surfaces; compare \cite{AtiyahSpinSurfaces}.  Choose such a
spin structure on \(V\), and give \(P=V\times L\) its product with
\(\sigma_L\).

The identity on every copy of \(S^1_x\times L\) now preserves the metric, the
spin structure, and both maps to \(BC_4\).  It is orientation preserving as a
map \(\partial P\to\partial E\), and hence orientation reversing as a map
\(\partial P\to\partial(-E)\), which is the convention needed for oriented
gluing.  On both sides the metric is exactly \(dt^2+h_\partial\).  Identifying
the normal coordinates by \(t\mapsto-t\) therefore produces a smooth metric,
not merely a metric with a corner.  Its scalar curvature on the collar is
\(\operatorname{scal}(h_\partial)>0\), and away from the collar it is positive
by Lemma~\ref{lem:affine-adiabatic-estimate} and
\eqref{eq:cap-scalar}.  We have proved the following.
\begin{proposition}\label{prop:affine-representative-psc}
The closed mapped spin five--manifold
\[
 M=(V\times L^3(4;1))\cup_{4(S^1\times L^3(4;1))}(-E)
\]
admits a smooth metric of positive scalar curvature.  The metric construction
does not change the map \(M\to B(C_4\times C_4)\).
\end{proposition}

\subsubsection{Surgery to obtain the prescribed fundamental group}
\label{sec:surgery}

The manifold of the preceding section comes with a map
\[
 f:M\longrightarrow BG,
 \qquad G=C_4\times C_4.
\]
The map on fundamental groups is surjective: the \(x\)-circle and the
generator of the lens-space fundamental group occur in every gluing
interface, and their images are the two displayed generators of \(G\).
The following standard surgery step replaces this surjection by an
isomorphism without changing the bordism or connective \(ko\)-homology class.

\begin{proposition}\label{prop:pi1-surgery}
Let \(M^5\) be a closed connected spin manifold with a positive-scalar-
curvature metric, and let \(f:M\to BG\) induce a surjection on fundamental
groups, where \(G\) is finitely presented.  Then there are a closed connected
spin manifold \(M'\) with a positive-scalar-curvature metric, a map
\(f':M'\to BG\), and a spin bordism over \(BG\) from \((M,f)\) to
\((M',f')\), such that
\[
 f'_*:\pi_1(M')\xrightarrow{\ \cong\ }G.
\]
Consequently \(f_*[M]=f'_*[M']\) in integral homology and
\(D[M,f]=D[M',f']\) in \(ko_5(BG)\).
\end{proposition}

\begin{proof}
Choose finite presentations of $\pi_1(M)$ and $G$.  A surjection between
finitely presented groups has finitely normally generated kernel: append
finitely many relations to a finite presentation of $\pi_1(M)$ to obtain
one of $G$.  Represent these normal generators by disjoint embedded
circles in $M$.  This is possible by general position in dimension five.
Each circle has an oriented normal bundle of rank four.  An oriented
vector bundle over a circle is trivial, so each has a framing.  Of the
two framing classes, choose the one for which the induced spin structure
on the attaching circle is bounding; the spin structure then extends
over the corresponding two-handle.

Attach these finitely many two-handles to $M\times[0,1]$ along its upper
boundary.  Each attaching loop is nullhomotopic after applying $f$, so
$f$ extends across the handles to a map of this spin bordism to $BG$.
Van Kampen's theorem identifies the fundamental group of its outgoing
boundary $M'$ with the quotient by the normal closure of the selected
loops, which is $G$.  The surgeries are on circles and have codimension
four.  The Gromov--Lawson surgery theorem
preserves positive scalar curvature for surgeries of codimension at least
three \cite[Theorem~A]{gromovlawsonsurgery}; the codimension-four case used here is
therefore covered.  The outgoing manifold \(M'\) is spin and has positive
scalar curvature.  Finally, the spin bordism over \(BG\) proves equality of
the pushed-forward fundamental classes and, by naturality and bordism
invariance of the Atiyah--Bott--Shapiro orientation, equality of the two
connective real \(K\)-homology classes.
\end{proof}

\subsection{Dimension 6}   
Here the actual bottom class is normalized as $p=s\boxtimes s$,
where $s$ is the bounding-spin circle mapped to a generator of $C_4$.
It has the Adams edge $x_0x_1$ used above.
 
\begin{lemma}[The surviving $\eta$-$c$-$R$ class]\label{lem:ko6-wood-class}
The element
\[
w=\omega p\in\widetilde{ko}_6(X)
\]
is nonzero, and
\[
\bigl|\widetilde{ko}_6(X)\bigr|\leq2^5.
\]
\end{lemma}

\begin{proof}
The coefficient product \(\omega p\) is an $E_2$ cycle in
stem--filtration \((6,3)\).  It is not $h_0$-divisible, whereas every
real $d_2$ image is $h_0$-divisible by
Theorem~\ref{theo:real-smash-d2-transfer}.  Thus it is nonzero on $E_3$.
Its only possible incoming $d_3$ has source $E_3^{ko}(7,0)$.
The actual $d_2$ formulas give
\[
 E_3^{ko}(7,0)=\mathbb F_2\{R_3\Phi_5\},
\]
where $\Phi_5$ is the complete eight-term cycle defined in the
complex third-differential subsection.  It is induced from the
diagonal cyclic class, and hence is permanent.  Naturality gives
$d_3(R_3\Phi_5)=R_3d_3(\Phi_5)=0$.
The outgoing $d_3$ from $(6,3)$ has target $(5,6)$, which is zero.
For \(r\geq4\), a differential entering \(w\)
would start in negative filtration.  A differential leaving it would land
in stem five and filtration \(3+r\geq7\).  But the degree-five \(E_4\)-page
has only \(\alpha_0,\beta_0\) in filtration one and their common
\(h_0\)-multiple in filtration two, as displayed in the verification  of the  $\eta$-$C$-$R$ exact  sequence.
Thus \(w\) is a permanent cycle and is nonzero by strong convergence.
The same $E_3$ calculation has five bits in degree six, which proves
the asserted order upper bound.
\end{proof}

\begin{lemma}[The cyclic input needed in degree six]\label{lem:cyclic-low-relation}
Let $s\in\widetilde{ko}_1(BC_4)$ be the mapped bounding-spin circle,
normalized by $c(s)=e_0$. Let $z_5$ be the orientation-normalized PSC
class of the explicit lens-space bundle $\mathcal X^5$ used in the eta
calculation. Then
\[
2z_5=\omega s.
\]
Moreover, in the coherent Hashimoto coordinates, there are integers $a,b$
such that
\[
c(z_5)=a v^2e_0+bve_1+2e_2.
\label{eq:cyclic-low-relation}
\]
No particular values of $a,b$ are required.
\end{lemma}
\begin{proof}
The standard cyclic computation gives
$\widetilde{ko}_5(BC_4)\cong\mathbb Z/4$; see, for example,
\cite[Theorem~2.2.40, p.~122]{siegemeyer}.
The direct fixed-set calculation for $\mathcal X^5$ gives
$\rho_{\sigma-1}(z_5)=1/4\pmod{\mathbb Z}$. Thus $z_5$ generates
this cyclic group and this eta coordinate is injective on it.

For the bounding-spin circle the spectrum, with a consistent choice of
flat-connection sign, is $n+1/2-r/4$. The reduced eta difference at
$r=1$ is $1/4$; the opposite sign convention changes it to $-1/4$.
The even-dimensional coefficient $\omega$ has complex Dirac index two,
since $c(\omega)=2v^2$. The Dirac product formula therefore gives
\[
\rho_{\sigma-1}(\omega s)=2\rho_{\sigma-1}(s)=1/2
\quad\text{in }\mathbb R/\mathbb Z.
\]
This is also $\rho_{\sigma-1}(2z_5)$. Injectivity proves
$2z_5=\omega s$. The product formula can be seen directly by decomposing
the product Dirac spectrum: nonzero even-factor eigenvalues occur in
opposite pairs and contribute zero, while the chiral harmonic spinors
contribute their index times the circle eta invariant. Kernel corrections
change this equality only by an integer, which is sufficient here.

We next determine only the ordinary homology edge of $z_5$.
In total degree five the reduced homological ko-AHSS has the terms
\[
(p,q)=(5,0),(4,1),(3,2),(1,4).
\]
The term $(4,1)$ is killed by its outgoing $d_2$ to $(2,2)$,
dual to $Sq^2z=z^2$. The term $(3,2)$ is killed by the incoming
$d_2$ from $(5,1)$, dual to $Sq^2(xz)=xz^2$.
The term $(5,0)=\mathbb Z/4$ has outgoing $d_2$ onto $(3,1)$,
again dual to $Sq^2(xz)=xz^2$ after reduction modulo two, so its
$E_3$ group is precisely $2\mathbb Z/4$.
Its possible later targets are $(2,2)$, already zero, $(1,3)=0$,
and reduced $(0,4)=0$; all further targets have negative p. There
are no incoming differentials, because their coefficient degrees
would be negative. Hence
\[
E^\infty_{5,0}=2\mathbb Z/4.
\]
All other positive-p terms in this total degree except $(1,4)$ vanish.
The cyclic group has order four, so its skeletal filtration subgroup $F_4=F_1$
has order two. A generator of order four cannot belong to it.
Consequently the ordinary homology edge of $z_5$ is the nonzero
element two in $H_5(BC_4;\mathbb Z)=\mathbb Z/4$.

Complexification commutes with the canonical maps to ordinary integral
homology. Hashimoto's class $e_2$ has ordinary edge the chosen generator,
whereas $ve_1$ and $v^2e_0$ have lower skeletal filtration. Its
presentation therefore gives \eqref{eq:cyclic-low-relation}. This uses only the ordinary leading
edge, not a claimed compatibility of an abstract representation-ring
isomorphism with geometric complexification.
\end{proof}

\begin{theorem}[Degree-six psc generation]\label{thm:degree-six-generation}
The group \(\widetilde{ko}_6(X)\) is generated by psc cycles.  More
precisely, let $\ell=L^3(4)\to Y$ be the round spin lens class and
let $z_5$ be the explicit PSC bundle class of
Lemma~\ref{lem:cyclic-low-relation}.  Then
\[
u_L=s\boxtimes z_5,\qquad u_R=z_5\boxtimes s,\qquad
\lambda=\ell\boxtimes\ell
\]
generate \(\widetilde{ko}_6(X)\), and
\begin{equation}\label{eq:ko6-double-relation}
2u_L=2u_R=w.
\end{equation}
There is an exact sequence
\begin{equation}\label{eq:ko6-wood-extension}
0\longrightarrow\mathbb{Z}/2\{w\}\longrightarrow
\widetilde{ko}_6(X)\xrightarrow{\ c\ }
K\longrightarrow0,
\end{equation}
where
\[
K=\langle c(\lambda),c(u_L),c(u_R)\rangle
\cong\mathbb{Z}/4\oplus(\mathbb{Z}/2)^2.
\]
In particular \(\bigl|\widetilde{ko}_6(X)\bigr|=2^5\).  No choice of a
splitting in \eqref{eq:ko6-wood-extension} is needed.
\end{theorem}

\begin{proof}
Normalize the actual bottom class by $p=s\boxtimes s$.
The relation $2z_5=\omega s$ of
Lemma~\ref{lem:cyclic-low-relation} therefore gives
\[
2u_L=s\boxtimes2z_5=\omega(s\boxtimes s)=w,
\qquad
2u_R=2z_5\boxtimes s=w.
\]
The class \(\lambda\) is represented by
\(L^3(4)\times L^3(4)\).  The classes \(u_L,u_R\) are represented by a psc
five-dimensional BGS cycle crossed with a flat circle; their product metrics
have positive scalar curvature.  Thus all three displayed generators are
psc classes.

It remains to prove that they generate.  In fixed-stem coordinates put
\[
g_{ij}=v^{2-i-j}(e_i\boxtimes e_j),
\qquad i,j\geq0,\quad i+j\leq2.
\]
The two families of \([4]\)-rows are
\[
4g_{ij}+6g_{i-1,j}+4g_{i-2,j}=0
\]
and the analogous second-coordinate rows, with negative indices omitted.
They reduce integrally to
\begin{equation}\label{eq:ku6-raw-smith}
\widetilde{ku}_6(X)=
\mathbb{Z}/2\{g_{00},g_{01},g_{10}\}
\oplus\mathbb{Z}/4\{g_{02},g_{11},g_{20}\}.
\end{equation}
Here is a row certificate.  The boundary rows first give \(4g_{00}=0\);
the two rows at \(g_{10}\) then give both \(4g_{10}=0\) and
\(4g_{10}+6g_{00}=0\), hence \(2g_{00}=0\).  The row at \(g_{02}\),
together with the boundary rows \(4g_{02}=4g_{01}=0\), similarly gives
\(2g_{01}=0\); factor symmetry, using \(4g_{10}=0\), gives
\(2g_{10}=0\).  The remaining rows give
\(4g_{02}=4g_{11}=4g_{20}=0\).  Conversely these six diagonal relations
imply every original row, so there are no omitted relations.

Let $V=\langle g_{00},g_{01},g_{10}\rangle$.  The exact row
certificate above gives
\[
 V\cong(\mathbb Z/2)^3,\qquad
 \widetilde{ku}_6(X)/V\cong
 (\mathbb Z/4)\{g_{02},g_{11},g_{20}\}.
\]
The ordinary degree-three edge of $c(\ell)$ is a generator, so
$c(\ell)=\epsilon e_1+A ve_0$ with $\epsilon$ odd.
Lemma~\ref{lem:cyclic-low-relation} gives
$c(z_5)=a v^2e_0+bve_1+2e_2$, with no need to identify $a,b$.
External-product naturality therefore gives, in the displayed quotient,
\[
 c(u_L)\longmapsto(2,0,0),\qquad
 c(\lambda)\longmapsto(0,1,0),\qquad
 c(u_R)\longmapsto(0,0,2),
\]
since $\epsilon^2=1$ modulo four.
Moreover $2c(u_L)=2c(u_R)=c(\omega p)=2g_{00}=0$.
Thus $c(u_L),c(u_R)$ have exact order two.
The class $c(\lambda)$ has exact order four, because its omitted
coordinates lie in $V$, of exponent two.  The three images are
independent already in the displayed quotient.  Consequently
\[
 K=\langle c(\lambda),c(u_L),c(u_R)\rangle
   \cong\mathbb Z/4\oplus(\mathbb Z/2)^2,\qquad |K|=16.
\]

Let \(H=\langle\lambda,u_L,u_R\rangle\).  Then \(c(H)=K\), while the nonzero
class \(w=2u_L\) belongs to \(H\cap\ker c\).  Therefore
\[
|H|\geq2|K|=32.
\]
On the other hand \ref{lem:ko6-wood-class} gives
\(|\widetilde{ko}_6(X)|\leq32\).  The two inequalities force
\[
H=\widetilde{ko}_6(X),\qquad
|\widetilde{ko}_6(X)|=32.
\]
It follows that \(\ker c=\langle w\rangle\) and
\(\im c=K=\ker R\), proving \eqref{eq:ko6-wood-extension}.
\end{proof}

\section{The GLR kernel and proof of the main result}\label{sec:conclusion}
We now assemble the Adams bounds of Section~\ref{sec:smash}, the geometric
detection results of Section~\ref{section:eta}, and the exceptional-dimensional
constructions of Section~\ref{sec:low}. We first identify the kernel and
complete the argument in dimensions at least seven, then include dimensions
five and six in the proof of the main theorem.

\subsection{The connective assembly kernel and the Stolz reduction}
\label{section:glr-kernel}

We now apply the Adams calculation to the GLR argument.  Write
\[
C_4=\mathbb Z/4,\qquad G=C_4\times C_4,\qquad
Y=BC_4,\qquad X=BC_4\wedge BC_4.
\]
Let \(g_0,g_1\) be the two standard generators of \(G\).  For
\(u,\tilde u\in\mathbb Z/4\), let \(\rho_{u,\tilde u}\) be the complex
character of \(G\) defined by
\[
\rho_{u,\tilde u}(g_0)=i^u,\qquad \rho_{u,\tilde u}(g_1)=i^{\tilde u}.
\]
Indices are always read modulo \(4\).  Complex conjugation sends
\(\rho_{u,\tilde u}\) to \(\rho_{-u,-\tilde u}\), so the self-conjugate
characters are precisely
\[
\rho_{0,0},\qquad \rho_{2,0},\qquad \rho_{0,2},\qquad \rho_{2,2}.
\]
In this section \(ko_n(BG)\) denotes the ordinary, unreduced
connective \(ko\)-homology group in which Dirac classes live.  The notation
\(\widetilde{ko}_n(BG)\), \(\widetilde{ko}_n(Y)\), and \(\widetilde{ko}_n(X)\)
mean reduced \(ko\)-homology.
Let
\[
D:\Omega_n^{\mathrm{Spin}}(BG)\longrightarrow ko_n(BG)
\]
be the connective real \(K\)-homology orientation of the spin Dirac operator,
and let
\[
A\circ{\rm per}:ko_n(BG)\longrightarrow KO_n(C_r^*(G))
\]
be periodicity followed by real Baum--Connes assembly.

\begin{lemma}[Unstable Stolz--BGS reduction]\label{lem:stolz-bgs-reduction}
Let \(G\) be finite and let \(n>5\).  If
\[
\ker(A\circ{\rm per})\subset ko_n(BG)
\]
is generated by Dirac classes of closed spin \(n\)-manifolds over \(BG\) which
admit positive-scalar-curvature metrics, then the unstable Gromov--Lawson--Rosenberg
conclusion holds for \(G\) in dimension \(n\).
\end{lemma}

\begin{proof}
Let \(M^n\) be a closed connected spin manifold with fundamental group \(G\)
and classifying map \(u:M\to BG\).  Put
\[
[M,u]_{ko}=D[M,u]\in ko_n(BG).
\]
We use the standard Stolz reduction in the surgery range: for \(n>5\),
\(M\) admits a positive-scalar-curvature metric if and only if
\([M,u]_{ko}\in ko_n^{\mathrm{psc}}(BG)\), where \(ko_n^{\mathrm{psc}}(BG)\) is the subgroup of
\(ko_n(BG)\) generated by Dirac classes of closed spin
positive-scalar-curvature manifolds over \(BG\).  This is the connective
\(ko\)-form of Stolz's positive-scalar-curvature exact sequence, together with
the \(\mathbb H P^2\)-transfer description of the \(D\)-kernel; see
\cite{stolzspin,stolzannals} and \cite[Theorems~2.8 and~2.11]{stolztrieste}.
A recent account distinguishing the bordism criterion from the index
conjecture is \cite[Theorems~1.4 and~1.7]{StolzSurvey2022}.
It is precisely the reduction used by Botvinnik--Gilkey--Stolz in
\cite[Theorem~5.1]{botvinnikgilkeystolz}.

If the Rosenberg index of \(M\) vanishes, then
\[
(A\circ{\rm per})([M,u]_{ko})=0,
\]
so \([M,u]_{ko}\in\ker(A\circ{\rm per})\).  By the hypothesis of the lemma,
this kernel is contained in \(ko_n^{\mathrm{psc}}(BG)\).  Stolz's criterion therefore gives a
positive-scalar-curvature metric on \(M\).  The reverse implication is the
Lichnerowicz index obstruction.
\end{proof}

\begin{definition}\label{def:glr-kernels}
Let \(\ker_4(n)\) denote the kernel of \(A\circ{\rm per}\) on either cyclic
summand \(\widetilde{ko}_n(Y)\).  Let \(\ker_{4,4}(n)\) denote the kernel of
\(A\circ{\rm per}\) on the smash summand \(\widetilde{ko}_n(X)\).
\end{definition}

\begin{lemma}[The basepoint summand]\label{lem:glr-basepoint-kernel}
The map \(A\circ{\rm per}\) is injective on the basepoint summand \(ko_n(*)\subset ko_n(BG)\).  Hence the
connective kernel on \(ko_n(BG)\) is the direct sum of the two cyclic kernels and the smash
kernel.
\end{lemma}

\begin{proof}
Let \(\epsilon:C_r^*(G)\to\mathbb R\) be induced by the trivial real
representation.  On the basepoint summand the composite
\[
ko_n(*)\xrightarrow{\ {\rm per}\ }KO_n(*)\xrightarrow{\ A\ }KO_n(C_r^*(G))
\xrightarrow{\ \epsilon_*\ }KO_n(\mathbb R)
\]
is the connective-to-periodic map \(ko_n(*)\to KO_n(*)\).  This map is an isomorphism for \(n\geq0\).  Therefore
\(A\circ{\rm per}\) has no kernel on the basepoint summand.

Here is a naturality argument that also rules out cancellation.  Let
$q_x,q_y:G\to C_4$ be the two projections and let $i_x,i_y$ be their
standard sections.  Since the groups are finite, these homomorphisms
induce maps of real group $C^*$-algebras, and assembly is natural for
them.  If $z$ has zero index, its augmentation to $ko_n(*)$ is zero by
the preceding injectivity.  The reduced classes $q_{x*}z$ and $q_{y*}z$
also have zero index.  Their images under $i_{x*}$ and $i_{y*}$ therefore
have zero index.  Subtracting these two axis classes from $z$ leaves
exactly its smash component, which consequently has zero index too.
Conversely, a sum of three such kernel classes has zero index by
additivity.  This proves the direct-sum assertion.

In representation coordinates this separation is triangular, not
diagonal.  Writing $I_{ab}$ for evaluation at $\rho_{a,b}$, the four
appropriate coordinates are
\[
 I_{00},\qquad I_{20}-I_{00},\qquad I_{02}-I_{00},\qquad
 I_{22}-I_{20}-I_{02}+I_{00}.
\]
For example, a first-axis class has the same raw index at $\rho_{2,0}$
and $\rho_{2,2}$, and the last difference removes this contribution.
\end{proof}

\begin{lemma}[PSC cycles and the smash projector]
\label{lem:psc-smash-projector}
The stable smash projector preserves the subgroup generated by
closed spin PSC cycles.  Sign-line changes on the smash summand
are understood as projected operations.
\end{lemma}
\begin{proof}
As an endomorphism of $ko_*(BG)$, the projector onto the stable
smash summand is
\[
 P_X=\mathrm{id}-(i_xq_x)_*-(i_yq_y)_*+\mathrm{const}_*.
\]
Applying it to $D[M,f]$ produces a signed sum of four reference
maps on the same spin PSC manifold $M$.  Thus it preserves the
PSC-generated subgroup without requiring new metrics.

The basepoint and both axes form a submodule under cap product
with either sign character: on the opposite axis the character
is trivial, and on the same axis the result stays on that axis
or its basepoint.  Hence, for $T_j=(\chi_j-1)\cap(-)$,
$P_XT_jP_X=P_XT_j$.  Projected spin changes can consequently be
computed using any actual PSC lift of a smash class; discarded
axis components cannot contribute back to it.  Both projections
are filtration-zero maps of the Adams towers.
\end{proof}

\subsection{Cyclic kernels and Adams bounds for the smash summand}

\begin{theorem}[The cyclic input]\label{thm:cyclic-kernel-orders}
For $n\geq5$, the cyclic kernels are generated by the Botvinnik--Gilkey--Stolz
positive-scalar-curvature classes.  Their \(\log_2\)-orders are
\[
\begin{array}{c|cccccccc}
n & 8d & 8d+1 & 8d+2 & 8d+3 & 8d+4 & 8d+5 & 8d+6 & 8d+7\\
\hline
\log_2|\ker_4(n)| & 0 & 2d+1 & 0 & 6d+4 & 0 & 2d+2 & 0 & 6d+6 .
\end{array}
\]
\end{theorem}

\begin{proof}
This is the \(C_4\) case of Botvinnik--Gilkey--Stolz
\cite{botvinnikgilkeystolz}.  The same orders are obtained from the Adams
calculation of \(\widetilde{ko}_*(BC_4)\), as in Siegemeyer's computation
\cite{siegemeyer}.  In degree \(1\pmod 8\) the periodic assembly image is the
single \(\mathbb Z/2\)-summand detected by the nontrivial real character of
\(C_4\); in the other odd congruence classes the cyclic periodic target has
no torsion summand receiving a nonzero map from the finite connective group.
The displayed kernel orders follow.
\end{proof}

\begin{theorem}[\(E_4\)-bounds for the smash summand]\label{thm:smash-orders}
The Adams calculation above gives the following upper bounds for the
\(E_4\)-dimensions and hence for the logarithmic orders of
\(\widetilde{ko}_n(X)\).  In low degrees,
\[
\begin{array}{c|cccccc}
n&2&3&4&5&6&7\\
\hline
\dim_{\F_2}E_4(n)\leq&2&3&5&3&5&7.
\end{array}
\]
For \(n\geq8\),
{\small\setlength{\arraycolsep}{3pt}
\[
\begin{array}{c|cccccccc}
n & 8d & 8d+1 & 8d+2 & 8d+3 & 8d+4 & 8d+5 & 8d+6 & 8d+7\\
\hline
\dim_{\F_2}E_4(n)\leq
&10d-1&8d-1&10d+1&12d+2&10d+4&8d+2&10d+5&12d+7 .
\end{array}
\]
}
\end{theorem}

\begin{proof}
The \(E_2\)-term is the direct sum of the one-cell, bow, and summed-bow blocks
from Lemma~\ref{lemma:A1decomp-smash-firstprinciples}.  Theorem
\ref{theo:real-smash-d2-transfer} determines the \(E_3\)-page block by
block, including the one-cell summand \(\eta\)- and \(\eta^2\)-multiples.  The
exhibited \(d_3\)-images of Proposition~\ref{theo:ko-smash-d3}
give sufficient lower bounds for the ranks.  We do not need to exclude
additional third differentials at this stage.

For \(n\geq8\), the count is the following.  First, the blockwise \(E_3\)-count
is
\[
\begin{array}{c|c|c|c}
n & \text{bow and summed-bow classes} & \text{one-cell classes} & e_3(n)\\
\hline
8d&4d^2+4d+1&0&4d^2+4d+1\\
8d+1&4d^2+2d+1&0&4d^2+2d+1\\
8d+2&4d^2+6d&2&4d^2+6d+2\\
8d+3&4d^2+8d+2&1&4d^2+8d+3\\
8d+4&4d^2+8d+4&1&4d^2+8d+5\\
8d+5&4d^2+6d+3&0&4d^2+6d+3\\
8d+6&4d^2+10d+4&1&4d^2+10d+5\\
8d+7&4d^2+12d+7&0&4d^2+12d+7 .
\end{array}
\]
Here the first column is obtained by counting the \(E_3\)-surviving
\(h_0\)-strings in the bow and summed-bow blocks of
Theorem~\ref{theo:real-smash-d2-transfer}.  More explicitly,
\eqref{eq:real-e3-all-degree-total}, proved by finite path
elimination and a coefficientwise Euler characteristic, is an
all-degree rational generating function for the last column;
ordinary coefficient extraction gives the displayed polynomials.
The one-cell column is the
coefficient contribution
\[
\mu^dp,\ h_0\mu^dp\quad(8d+2),\qquad
\eta\mu^dp\quad(8d+3),\qquad
\eta^2\mu^dp\quad(8d+4),\qquad
\mu^d\omega p\quad(8d+6).
\]
Next let \(r_{\rm out}(n)=\sum_s\rho(n,s)\) be the exhibited lower
bound for the total rank of \(d_3\) with source in degree \(n\), and
let \(r_{\rm in}(n)=r_{\rm out}(n+1)\).  Summing the formula of
Proposition~\ref{theo:ko-smash-d3} over the source filtrations gives
{\small\setlength{\arraycolsep}{3pt}
\[
\begin{array}{c|cccccccc}
n&8d&8d+1&8d+2&8d+3&8d+4&8d+5&8d+6&8d+7\\
\hline
r_{\rm out}(n)&0&4d^2-6d+2&0&(2d-1)^2&0&4d^2-2d+1&0&4d^2 .
\end{array}
\]
}
Thus the \(E_4\)-dimension in total degree \(n\) is at most
\[
e_4(n)\leq e_3(n)-r_{\rm out}(n)-r_{\rm in}(n).
\]
\[
\begin{array}{c|c|c|c|c}
n&e_3(n)&r_{\rm out}(n)&r_{\rm in}(n)&e_4(n)\\
\hline
8d&4d^2+4d+1&0&4d^2-6d+2&10d-1\\
8d+1&4d^2+2d+1&4d^2-6d+2&0&8d-1\\
8d+2&4d^2+6d+2&0&(2d-1)^2&10d+1\\
8d+3&4d^2+8d+3&(2d-1)^2&0&12d+2\\
8d+4&4d^2+8d+5&0&4d^2-2d+1&10d+4\\
8d+5&4d^2+6d+3&4d^2-2d+1&0&8d+2\\
8d+6&4d^2+10d+5&0&4d^2&10d+5\\
8d+7&4d^2+12d+7&4d^2&0&12d+7 .
\end{array}
\]
In total degrees \(2\leq n\leq7\), no exhibited third differential
contributes, so the \(E_3\)-block count gives the low-degree upper row
\[
\begin{array}{c|cccccc}
n&2&3&4&5&6&7\\
\hline
e_4(n)\leq&2&3&5&3&5&7.
\end{array}
\]
Since every later Adams differential can only decrease the associated-graded
dimension in a fixed total degree, these \(E_4\)-bounds give the asserted
upper bounds for \(\log_2|\widetilde{ko}_n(X)|\).
\end{proof}

\begin{lemma}[Periodic assembly on the smash summand]\label{lem:smash-assembly-image}
For \(n\geq5\), on \(\widetilde{ko}_n(X)\), the image of
\[
A\circ{\rm per}:\widetilde{ko}_n(X)\longrightarrow KO_n(C_r^*(G))
\]
is one copy of \(\mathbb Z/2\) for \(n\equiv1,2\pmod8\), detected by the
real representation \(\rho_{2,2}\), and is zero in the remaining congruence
classes.
\end{lemma}

\begin{proof}
The real group algebra is
\[
 \mathbb R[G]\cong\mathbb R^4\oplus\mathbb C^6.
\]
A finite reduced connective group has zero image in every torsion-free
coefficient group.  Thus only the real $KO_1$ and $KO_2$ coefficients
can receive a nonzero image.  On the smash component, the first three
real character evaluations vanish by factorization through an axis or
the basepoint.  The fourth evaluation, equivalently its mixed
difference, is therefore the only possible nonzero coordinate.

Let $s$ be the bounding-spin circle mapped to a generator of $C_4$.
Its untwisted index is zero, whereas twisting by the sign character
changes its spin structure to the nonbounding one and gives $\eta$.
The actual external product
\[
 p=s\boxtimes s\in\widetilde{ko}_2(X)
\]
has mixed index $\eta^2\ne0$.  Hence $\mu^dp$ has nonzero index in
every degree $8d+2$.

For degree $8d+1$ one must exhibit a connective preimage.  The
complexification $c(p)$ is the bottom tensor generator in
$\widetilde{ku}_2(X)$.  Davis' annihilator theorem
\cite[Theorem~1.6]{davismorfismos} gives
\[
 \operatorname{ann}_{ku_*}(c(p))=(4,2v,v^4).
\]
In particular $c(\mu p)=v^4c(p)=0$.  Exactness of the actual homology
Wood sequence, not an asserted exactness of an Adams page, supplies
$x_9\in\widetilde{ko}_9(X)$ such that
\[
 \eta x_9=\mu p.
\]
Applying the mixed index and using its coefficient-module structure gives
\[
 \eta\,I_{22}(x_9)=\mu\eta^2\ne0.
\]
Multiplication by $\eta$ is an isomorphism from $KO_9(\mathbb R)$ to
$KO_{10}(\mathbb R)$, both of which are $\mathbb Z/2$.
Thus $I_{22}(x_9)\ne0$, and $\mu^{d-1}x_9$ proves the assertion in
all degrees $8d+1$ with $d\ge1$.  The degrees in which this lemma is
used satisfy this range; there is no degree-one smash class.
\end{proof}

\begin{corollary}[Odd kernel upper bounds]\label{cor:odd-kernel-orders}
For odd \(n\geq5\), the cyclic kernel has the displayed exact order, and the
smash kernel has the displayed upper bound:
\[
\begin{array}{c|c|c}
n & \log_2|\ker_4(n)| & \text{upper bound for }\log_2|\ker_{4,4}(n)|\\
\hline
5 & 2 & 3\\
8d+1 & 2d+1 & 8d-2\quad(d\geq1)\\
8d+3 & 6d+4 & 12d+2\quad(d\geq1)\\
8d+5 & 2d+2 & 8d+2\quad(d\geq1)\\
8d+7 & 6d+6 & 12d+7 .
\end{array}
\]
\end{corollary}

\begin{proof}
The cyclic column is the exact computation in Theorem~\ref{thm:cyclic-kernel-orders}.  For the smash column,
combine the \(E_4\)-upper bounds of Theorem~\ref{thm:smash-orders} with
Lemma~\ref{lem:smash-assembly-image}.  In odd degree the smash assembly image
subtracts one generator only in degree \(1\pmod8\).  The degree \(5\) line uses
the low-dimensional \(E_4\)-upper bound from Theorem~\ref{thm:smash-orders}; it is not part of the periodic
\(8d+5\) row.
\end{proof}

\subsection{Generation of the kernel in dimensions at least seven}

\begin{theorem}[Odd-dimensional kernel detection]\label{thm:odd-kernel-detection}
For every odd \(n\geq7\), the subgroup \(C_n\) is the full kernel of
\[
A\circ{\rm per}:ko_n(BG)\longrightarrow KO_n(C_r^*(G)).
\]
\end{theorem}

\begin{proof}
Each generator of \(C_n\) is represented by a positive-scalar-curvature
manifold, so \(C_n\subset\ker(A\circ{\rm per})\).  By
Proposition~\ref{prop:eta-matrix-orders}, eta-invariants detect a subgroup of
\(C_n\) of logarithmic order
\[
12d,\qquad 24d+10,\qquad 12d+6,\qquad 24d+19
\]
in degrees \(8d+1,8d+3,8d+5,8d+7\), respectively.  On the other hand,
Lemma~\ref{lem:glr-basepoint-kernel}, the stable splitting, and
Corollary~\ref{cor:odd-kernel-orders}, with \(d\geq1\) in the first three rows
and \(d\geq0\) in the \(8d+7\) row, give the same four numbers as upper bounds
for
\[
\log_2\left|\ker\bigl(A\circ{\rm per}:ko_n(BG)\to KO_n(C_r^*(G))\bigr)\right|.
\]
Thus the eta-detected subgroup gives a lower bound equal to the Adams upper
bound.  The upper bound is sharp, and the inclusions
\[
C_n\subset\ker(A\circ{\rm per})\subset ko_n(BG)
\]
force \(C_n=\ker(A\circ{\rm per})\).
\end{proof}

\begin{corollary}[Collapse forced by odd-dimensional geometry]
\label{cor:repaired-collapse}
In stems $n\geq7$, the smash bounds of
Theorem~\ref{thm:smash-orders} are exact.  No differential $d_r$,
$r\geq4$, has source or target in these stems.
\end{corollary}
\begin{proof}
The preceding proof attains the full odd kernel upper bound.
The two cyclic summands have independently known exact orders, and
Lemma~\ref{lem:smash-assembly-image} determines the actual smash
assembly image.  Consequently the order of the odd smash group
attains its stated Adams upper bound.  No additional third
differential, nor any later differential, can remove one of those
odd-stem bits.

Every Adams differential decreases the stem by one.  Thus any extra
differential touching an even stem $n\geq8$ would remove a bit in
one of its adjacent odd stems, both at least seven.  This is
impossible.  The exhibited third differentials therefore already
give the exact even bounds, and there are no later differentials
in these stems either.  This argument uses no even-dimensional
PSC-generation theorem.
\end{proof}

\begin{remark}[Dimension five]
The cyclic-induction count leaves one missing bit in dimension five.
The missing coset and its affine representative are treated in
Section~\ref{sec:low}, particularly Corollary~\ref{cor:missing-coset}
and Proposition~\ref{prop:affine-representative-psc}.
\end{remark}

\subsection{Generation in dimensions divisible by four}

This argument uses Corollary~\ref{cor:repaired-collapse} in the
following precise way: the exhibited real $d_3$ images
have source filtration at least one, and the odd eta lower bounds make
their rank bounds sharp.  There are therefore no additional real
$d_3$ differentials, nor later differentials incident to a stem at
least seven.  In particular, the real $E_3$ classes in an even stem
and filtrations zero through three survive unchanged.  We use the
known order upper bound
\begin{equation}\label{eq:4N-known-upper}
 \log_2|\widetilde{ko}_{4N}(X)|\leq5N-1\qquad(N\geq2).
\end{equation}
No claimed long doubling strings are needed.

Put $m_{a,b}=x_0z_0^ax_1z_1^b$, and set negative-index monomials to
zero.  For $N\geq2$ define
\[
\begin{split}
 b_0&=m_{0,2N-1},\qquad b_N=m_{2N-1,0},\\
 b_i&=m_{2i,2N-2i-1}+m_{2i-1,2N-2i}\quad(1\leq i<N),\\
 o_i&=m_{2i-1,2N-2i-1}\quad(1\leq i<N).
\end{split}
\]
Write $\mathcal B_N=\{b_0,\ldots,b_N\}$ and
$\mathcal O_{N-1}=\{o_1,\ldots,o_{N-1}\}$.  The same conventions
define $\mathcal B_{N-1}=\{f_0,\ldots,f_{N-1}\}$ and
$\mathcal O_{N-2}$.  Empty ranges contribute no classes.

For a real bow lift in odd residue we write $v_0^eo$; the subscript
distinguishes this notation from a real coefficient operation.  The
four-residue complexification formulas give
\[
 c(v_0o_i)=vo_i+h_0r_i,\qquad
 r_i=m_{2i,2N-2i-1},\qquad c(v^2b)=v^2b.
\]
These are names for explicit Ext elements, not assertions of real
$v$-linearity.

\begin{lemma}[The four low-filtration rows]\label{lem:4N-low-row-proof}
In stem $4N$, $N\geq2$, the real $E_3$ rows in filtrations zero
through three have the following bases:
\[
\begin{array}{c|c|c}
 s&\text{basis}&\text{dimension}\\\hline
 0&\mathcal B_N&N+1\\
 1&h_0\mathcal B_N\ \sqcup\ v_0\mathcal O_{N-1}&2N\\
 2&v^2\mathcal B_{N-1}&N\\
 3&v_0^3\mathcal O_{N-2}&N-2.
\end{array}
\]
They are consequently $E_\infty$ bases.  There are no other
associated-graded classes in this stem.
\end{lemma}
\begin{proof}
We give the low-row calculation explicitly.  The odd--odd part of the
real $E_2$ page is obtained from the $P\otimes D_k$ boundary bows and
the two bows in each $D_k\otimes D_l$ summed bow.  In the four rows it
has bases
\[
\begin{array}{c|l}
0&\mathcal B_N\\
1&h_0\mathcal B_N,\ v_0\mathcal O_{N-1}\\
2&h_0^2\mathcal B_N,\ h_0v_0\mathcal O_{N-1},\ v^2\mathcal B_{N-1}\\
3&h_0^3\mathcal B_N,\ h_0^2v_0\mathcal O_{N-1},\
  h_0v^2\mathcal B_{N-1},\ v_0^3\mathcal O_{N-2}.
\end{array}
\]
This enumeration follows by adding degree $2e$ and filtration $e$
on a bow and then allowing $h_0$ multiples.  The two endpoint bows
give the one-term $b_0,b_N$; the second bow of a summed bow gives
the adjacent-pair $b_i$; its first bow gives $o_i$.

No even--even class is a real $d_2$ cycle.  Indeed its nonzero
complexification is a polynomial in positive powers of both $z_i$,
with coefficients in $\mathbb F_2[h_0,v]$.  The first-coordinate
part of complex $d_2$ sends $z_0^a$ to
$h_0^2x_0z_0^{a-1}+h_0vx_0z_0^{a-2}$ and is injective on these
polynomials: the highest $z_0$ exponent has a nonzero, uncancelled
$h_0^2$ term.  The other-coordinate images have a different $x$
support and cannot cancel it.  Nonpoint real complexification is
injective on $E_2$, by the bow formulas.  The one-cell block has no
class in this stem and these filtrations: its possible coefficient
stem is $4N-2$, and the first possible degree-$2$-modulo-eight
class here has filtration at least six.  One-$x$ blocks have odd
stem and do not enter the calculation.

Every displayed odd--odd class is a real $d_2$ cycle.  Its complex
$d_2$ is zero.  Any possible real image in the kernel of $c$ must
lie on the one-cell block.  In the odd target stem $4N-1$ its only
possible type is $\eta\mu^dp$.  Every source bow class is killed
by $h_1$, whereas $h_1\eta\mu^dp=\eta^2\mu^dp$ is nonzero on
$E_2$.  Multiplicativity by $h_1$ therefore excludes that image.

There is no incoming $d_2$ to filtrations zero or one.  For
filtration two, put $l=N+1-k$ and use the $2N$ filtration-zero
sources
\[
 A_k=z_0^{2k-1}x_1z_1^{2l-1},\qquad
 A'_k=x_0z_0^{2k-1}z_1^{2l-1}\quad(1\leq k\leq N).
\]
They are the first bow bottoms of the one-$x$ summed bows.  Direct
substitution in the complex $d_2$ formulas and the bow
complexification formulas gives the real identities
\[
\begin{split}
 d_2A_1&=h_0^2b_0,\qquad d_2A'_N=h_0^2b_N,\\
 d_2A_{i+1}&=h_0v_0o_i,\\
 d_2A'_i&=h_0^2b_i+h_0v_0o_i\qquad(1\leq i<N).
\end{split}
\]
For example, $c(d_2A_{i+1})=h_0^2r_i+h_0vo_i$.
The absence of a point-block kernel in the target bidegree makes
these real identities unambiguous.  They kill precisely the first
two filtration-two families.  No incoming complex $d_2$ has an
$h_0$-free $v^2$ term, so $v^2\mathcal B_{N-1}$ remains independent.

In filtration three, the $h_0$ multiples of these sources first
kill $h_0^3\mathcal B_N$ and $h_0^2v_0\mathcal O_{N-1}$.  To see
the remaining incoming terms explicitly, put $l=N-k$ and take
the odd-residue lift $Z_k$ of the second bow in $C_k\otimes D_l$,
for $1\leq k<N$.  Its complexification is
\[
\begin{split}
 c(Z_k)={}&v\bigl(z_0^{2k}x_1z_1^{2l-1}
               +z_0^{2k-1}x_1z_1^{2l}\bigr)
             +h_0z_0^{2k}x_1z_1^{2l}.
\end{split}
\]
Applying complex $d_2$ gives
\[
\begin{split}
 c(d_2Z_k)={}&h_0^3m_{2k-1,2l}
       +h_0^2v m_{2k-1,2l-1}\\
 &+h_0v^2\bigl(m_{2k-2,2l-1}+m_{2k-3,2l}\bigr).
\end{split}
\]
Modulo the already killed families, the last line is
$h_0v^2f_{k-1}$.  The right endpoint odd-residue source in
$C_N\otimes P$ similarly gives $h_0v^2f_{N-1}$ modulo the same
families.  These $N$ vectors kill all of
$h_0v^2\mathcal B_{N-1}$.  On the other hand, the complexification
of every nonzero combination of $v_0^3\mathcal O_{N-2}$ has a
nonzero $h_0$-free $v^3$ term.  All complex $d_2$ images are
$h_0$-divisible, so none of these combinations is a boundary.
This proves the asserted four $E_3$ rows without a hidden
row-reduction assumption.

The sharp odd-dimensional comparison stated at the beginning makes
these rows permanent.  Their dimensions sum to $5N-1$, the full
upper bound \eqref{eq:4N-known-upper}.  Strong convergence therefore
leaves no additional associated-graded class in this stem.
\end{proof}

\begin{lemma}[Low-row filtration strictness]\label{lem:4N-spin-strict-proof}
Complexification is injective on the associated-graded rows $s=1,2$
in degree $4N$.  If $\chi_j$ is the real sign line on the $j$-th
factor, let $\Delta_j$ mean the smash projection of
$(\chi_j-1)\cap(-)$, after the stable inclusion into $ko_*(BG)$.
The inclusion and projection are maps of Adams towers.  Then
\[
 \Delta_j(F^0)\subseteq F^2,\qquad
 \Delta_j(F^1)\subseteq F^3.
\]
\end{lemma}
\begin{proof}
The row-one images are $h_0b_i$ and $vo_i+h_0r_i$.
Their $v$-terms are distinct and disjoint from all $h_0b_i$.
There is no incoming complex differential to filtration one.
The row-two images are $v^2f_i$, whose $h_0$-free terms cannot
be complex $d_2$ boundaries.  There is no later incoming
differential to filtration two.  Outgoing differentials on these
images vanish by naturality from the permanent real classes.
The images thus remain independent on $E_\infty$.

Let $\alpha_j$ denote the corresponding order-four complex line.
Complexification gives
\[
 c(\chi_j-1)=\alpha_j^2-1
       =2(\alpha_j-1)+(\alpha_j-1)^2.
\]
The complex rank-zero line class $\alpha_j-1$ has Adams filtration
at least one, so this operation has complex filtration at least
two.  The real rank-zero class $\chi_j-1$ has filtration at least
one, by the mod-two rank edge of its cohomological Adams tower.
If $\Delta_j(F^0)$ had a filtration-one leading term, its image
under the injective row-one map $c$ would contradict the complex
filtration-two assertion.  Hence it lies in $F^2$.  For an input
in $F^1$, the real operation first lies in $F^2$, whereas its
complexification lies in $F^3$; row-two injectivity gives the
second inclusion.
\end{proof}

\begin{proposition}[PSC generation in degree $4N$]\label{prop:4N-psc-generation-repaired}
For $N\geq2$, every element of $\widetilde{ko}_{4N}(X)$ is
represented by a sum of closed spin PSC cycles over $BG$.
\end{proposition}
\begin{proof}
The corrected manifolds $\widehat N_{i,N-i}$ of
Proposition~\ref{prop:corrected-even-lens-bundles}, together with
$L^{4N-1}(4)\times S^1$ and $S^1\times L^{4N-1}(4)$,
give the triangular filtration-zero edges.  Successive differences
give exactly $b_0,\ldots,b_N$.  Their doubles have nonzero independent
edges $h_0b_i$ in filtration one, as verified in
Lemma~\ref{lem:4N-low-row-proof}.

We explain the other product classes without assuming any exact
Spin-to-Gysin normalization.  For each $i\geq1$, choose the Hashimoto class
$q_i=e_{2i}\in\widetilde{ku}_{4i+1}(BC_4)$, whose filtration-zero
edge is $xz^{2i}$; only this ordinary mod-two top edge is used.
Let $R:ku_n\to ko_{n-2}$ be the shifted Wood realification.
For $e=1,3$, the class $R(v^eq_i)$ has the odd-residue real
leading term $v_0^exz^{2i-1}$.  This is the four-residue formula
$R(v^e\cdot\text{upper bow cell})=\text{real bow class}$.
There is no lower-filtration term, since $R$ is a map of Adams
towers.  These leading terms survive the one-factor $d_2$:
their complexifications have an $h_0$-free $v^e$ term, while
complex $d_2$ images are $h_0$-divisible.  For $e=3$ the only
other possible incoming differential has filtration-zero even
source; the one-factor even filtration-zero classes have nonzero
$d_2$, so there is no such source on $E_3$.  Higher incoming
sources have negative filtration.  Thus these are actual
low-filtration cyclic classes.

Every cyclic $ko$ class has a closed spin lift.  To justify this
integrally, the ABS map has a $ko$ summand after localization at
two by Anderson--Brown--Peterson
\cite{andersonbrownpeterson}.  Clear the odd denominator of a
localized lift, and then invert that odd integer modulo the order
of the finite cyclic $ko$ class.  This produces an integral spin
bordism lift.  An arbitrary spin cycle crossed with a PSC spin
cycle is PSC after shrinking the PSC factor, because scalar
curvature adds on products.

For $1\leq i<N$, cross $R(vq_i)$ with the standard spin PSC lens
class having edge $xz^{2N-2i-1}$.  These products have the
independent leading terms $v_0\mathcal O_{N-1}$.  For
$1\leq i<N-1$, cross $R(v^3q_i)$ with the PSC lens class having
edge $xz^{2N-2i-3}$.  Their leading terms are
$v_0^3\mathcal O_{N-2}$.  Product naturality and the explicit
$E_3$ calculation above identify these terms; their real
permanence was established there.  No assertion that the
arbitrary first factor is PSC is needed.

It remains to obtain the row $v^2\mathcal B_{N-1}$.  Changing a
spin structure by the pullback of the first sign character changes
the Dirac class by $\Delta_0$.  Both spin structures have the same
PSC metric, so the resulting difference remains in the PSC
subgroup.  Let $u=\alpha_0-1$.  In the standard fixed-stem Gysin
coordinates, cap product by $u$ lowers the first index by one
and supplies a factor $v$.  This is the defining Euler-lowering
normalization of the cyclic Gysin generators
\cite[Lemmas~2.2 and~2.5]{hashimotogenerators}.  Consequently the
filtration-two leading operation of
$c(\Delta_0)=2u+u^2$ on a filtration-zero monomial is
$h_0vS_0+v^2S_0^2$, where $S_0m_{a,b}=m_{a-1,b}$.

The $h_0v$ terms vanish on complex $E_3$.  Explicitly, for
$a,b\geq0$ put
\[
 \Psi_{a,b}=
 \sum_{j=0}^{a}x_0z_0^jz_1^{a+b+2-j}
 +\sum_{j=1}^{a+1}z_0^jx_1z_1^{a+b+2-j}.
\]
The two complex $d_2$ formulas telescope to
$d_2\Psi_{a,b}=h_0v m_{a,b}$.
The other term has the literal index-lowering pivots
\[
 S_0^2b_i=f_{i-1}\quad(1\leq i\leq N),\qquad
 S_0^2b_0=0.
\]
These equations follow directly from the two-term formulas for
$b_i,f_i$; at the endpoints the negative-index terms are zero.
Thus the $N$ spin differences have an identity matrix onto the
row $v^2\mathcal B_{N-1}$ under complexification.  By
Lemma~\ref{lem:4N-spin-strict-proof}, an unknown filtration-one
part of a chosen geometric lift is sent to filtration at least
three and does not change this calculation.  Row-two injectivity
identifies the same full-rank pivots in the real associated graded.

All classes may be taken on the smash summand by
Lemma~\ref{lem:psc-smash-projector}.  This projection is a signed
sum of reference maps on the same PSC cycles, so it introduces
no new geometric existence assumption.

We have now produced every associated-graded row in
Lemma~\ref{lem:4N-low-row-proof}.  A subgroup of a finite filtered
group whose associated graded surjects onto every row is the
whole group: subtract a matching leading term and repeat down the
finite filtration.  This proves PSC generation.
\end{proof}

\subsection{Generation in dimensions congruent to two modulo four}

Let $N\geq2$, $m=2N$, and $n=4N+2$.  We use
Corollary~\ref{cor:repaired-collapse}: every real $E_3$
class in the even stem and filtrations zero through three is
permanent.  The connective assembly kernel in this stem has order
at most $2^{5N}$.  When $N$ is even, the extra one-cell class
$\mu^{N/2}p$ has nonzero assembly and is not in this kernel.

Use $m_{a,b}=x_0z_0^ax_1z_1^b$ as before.  Write
\[
 A_i=m_{i,m-i},\quad B_i=m_{i,m-1-i},\quad
 C_i=m_{i,m-2-i},\quad D_i=m_{i,m-3-i},
\]
with out-of-range symbols zero.  Put
\[
 o_k=A_{2k-1}\quad(1\leq k\leq N),
\]
and define the paired basis
\[
 p_0=B_0,\quad
 p_i=B_{2i-1}+B_{2i}\ (1\leq i<N),\quad p_N=B_{m-1}.
\]
Thus $\mathcal O_N=\{o_k\}$ and
$\mathcal B_N=\{p_0,\ldots,p_N\}$.  The canonical odd-residue
real lifts have complexification
\[
 c(v_0p_i)=vp_i+h_0A_{2i}\quad(0\leq i\leq N).
\]
At $i=N$ the upper class is $A_m$, as required for the right
one-term endpoint.

\begin{lemma}[The four low rows in degree $4N+2$]
\label{lem:4N2-low-rows-repaired}
The real $E_3$ page in this stem has the following low-filtration
bases, all of which survive to $E_\infty$:
\[
\begin{array}{c|c|c}
s&\text{basis}&\text{dimension}\\\hline
0&\mathcal O_N&N\\
1&h_0\mathcal O_N\ \sqcup\ v_0\mathcal B_N&2N+1\\
2&v^2\mathcal O_{N-1}&N-1\\
3&v_0^3\mathcal B_{N-1}&N.
\end{array}
\]
Complexification is injective on all three positive-filtration rows.
\end{lemma}
\begin{proof}
The odd--odd real $E_2$ bases in the four rows are
\[
\begin{array}{c|l}
0&\mathcal O_N\\
1&h_0\mathcal O_N,\ v_0\mathcal B_N\\
2&h_0^2\mathcal O_N,\ h_0v_0\mathcal B_N,\ v^2\mathcal O_{N-1}\\
3&h_0^3\mathcal O_N,\ h_0^2v_0\mathcal B_N,\
 h_0v^2\mathcal O_{N-1},\ v_0^3\mathcal B_{N-1}.
\end{array}
\]
This follows by enumerating the boundary bows and the two bow
bottoms in each odd--odd summed bow, as in degree $4N$.
Even--even classes have no $d_2$ cycles, by the same triangular
injectivity of the first-coordinate complex $d_2$.  One-$x$
blocks have odd stem.  There is no point-block class in the four
rows: its coefficient stem would be $4N$, whose first possible
coefficient filtration is at least four for $N\geq2$.
Every odd--odd class is a real $d_2$ cycle: its complex differential
is zero and the one-cell coefficient kernel of $c$ has no class
in the target stem $4N+1$.

There is no incoming $d_2$ to rows zero or one.  We give explicit
incoming sources for the next two rows.  In filtration zero of
stem $4N+3$, the two endpoint sources have complexifications
\[
 Q_L=x_0z_1^{m+1},\qquad Q_R=z_0^{m+1}x_1.
\]
For $k+l=N+1$, let $E_k,E'_k$ be the adjacent-pair real bow
bottoms with complexifications
\[
\begin{split}
c(E_k)&=z_0^{2k}x_1z_1^{2l-1}
           +z_0^{2k-1}x_1z_1^{2l},\\
c(E'_k)&=x_0z_0^{2k}z_1^{2l-1}
           +x_0z_0^{2k-1}z_1^{2l}.
\end{split}
\]
The complex $d_2$ formulas give the exact real identities
\[
\begin{split}
d_2Q_L&=h_0v_0p_0,\qquad d_2Q_R=h_0v_0p_N,\\
d_2E_k&=h_0v_0p_{k-1}+h_0^2o_k,\\
d_2E'_k&=h_0v_0p_k+h_0^2o_k\quad(1\leq k\leq N).
\end{split}
\]
There is no point-block kernel in the target bidegree.  Starting
with $p_0$ and proceeding along this path, these identities kill
all $h_0v_0p_i$ and $h_0^2o_k$, a space of dimension $2N+1$.
They cannot kill an $h_0$-free $v^2$ term.  Thus row two is the
stated $v^2\mathcal O_{N-1}$.

In row three the $h_0$ multiples of these sources kill
$h_0^3\mathcal O_N$ and $h_0^2v_0\mathcal B_N$.  For
$k+l=N+1$ take the odd-residue first-bow source $V_k$ with
\[
 c(V_k)=v z_0^{2k-1}x_1z_1^{2l-1}
                 +h_0z_0^{2k}x_1z_1^{2l-1}.
\]
A direct differential calculation gives
\[
 c(d_2V_k)=h_0^3o_k+h_0v^2 C_{2k-3}.
\]
For $2\leq k\leq N$, these kill all
$h_0v^2\mathcal O_{N-1}$.  The remaining
$v_0^3\mathcal B_{N-1}$ has independent $h_0$-free $v^3$ terms
under complexification, which cannot be complex $d_2$ boundaries.
This proves the four real $E_3$ rows.  They survive by the sharp
odd-dimensional comparison.  Their dimensions sum to $5N$.
The geometric argument below will place them in the assembly
kernel and then use its upper bound to prove exhaustion; no
exhaustion assertion is needed for the present low-row calculation.

For clarity we also verify complexification in row three, rather
than assume a global statement about its Wood kernel.  The
complex filtration-zero $E_3$ source in stem $4N+3$ is the single
cycle
\[
 \Phi=Q_L+Q_R+\sum_{k=1}^N\bigl(c(E_k)+c(E'_k)\bigr).
\]
Indeed, its terms are precisely all one-$x$ monomials in that
stem; the $h_0^2$ coefficients of the complex $d_2$ force the
two coefficients adjacent to each target to agree, and the
$h_0v$ coefficients then force all coefficients to agree.
This proves the one-dimensional kernel assertion.  This full sum
is the smash projection of the diagonal image of the cyclic class
$xz^{m+1}$.  Every one-$x$ monomial pulls back to $xz^{m+1}$,
which verifies this identification directly by pairing.  The cyclic
complex class is permanent, so naturality gives $d_3\Phi=0$.
There is therefore no complex $d_3$ into the row-three images.
No later incoming source has nonnegative filtration.  Outgoing
differentials on images vanish by naturality from the permanent
real classes.  Hence the independent $h_0$-free $v^3$ terms
remain independent under $c$ on $E_\infty$.  Row-one and row-two
injectivity follow identically from their independent leading
terms and the absence of any later incoming differential into
these filtrations.
\end{proof}

\begin{lemma}[Spin changes in the three relevant rows]
\label{lem:4N2-spin-strict-repaired}
For the filtered smash group in degree $4N+2$, with $\Delta_j$
understood as the projected operation defined above,
\[
 \Delta_j(F^0)\subseteq F^2,\quad
 \Delta_j(F^1)\subseteq F^3,\quad
 \Delta_j(F^2)\subseteq F^4.
\]
\end{lemma}
\begin{proof}
The real rank-zero sign-line operation raises filtration by at
least one, whereas its complexification
$2(\alpha_j-1)+(\alpha_j-1)^2$ raises it by at least two.
Apply the injectivity of $c$ on rows one, two, and three of
Lemma~\ref{lem:4N2-low-rows-repaired} successively.  The
operation preserves the assembly kernel: twisting a real character
permutes the representation-index coordinates, and the smash
projection preserves the kernel by Lemma~\ref{lem:glr-basepoint-kernel}.
For PSC inputs, geometrically it is
the difference of the two spin Dirac classes for the same PSC
metric whenever its input has been represented by PSC cycles.
For the particular classes used below this geometric observation
alone suffices.
\end{proof}

\subsection{The first mixed Gysin term from spin indices}

Let $N\ge2$, $m=2N$, $a=2r+1$, and $b=m-a$, with
$0\le r\le N-1$.  Degree six is handled separately.
The mixed double lens is
\[
M_{r,N}=(S^{2a+1}\times S^{2b+1})/(C_4\times C_4),
\]
where $(s,t)$ acts by
\[
(z;w_1,w_2,w_3,\ldots)
 \longmapsto(sz;stw_1,stw_2,tw_3,\ldots).
\]
The action is free: the first sphere forces $s=1$ for any stabilizer,
and the second then forces $t=1$.  Both spheres are simply connected,
and the finite quotient has fundamental group $C_4^2$.  Its product
round metric descends and has positive scalar curvature.  The same
formula defines a free action of $T^2$.  Let
\[
 B=(S^{2a+1}\times S^{2b+1})/T^2.
\]
With the convention that $L_0$ is the hyperplane line on the base
$\mathbb CP^a$, this is the projective bundle
\[
B=P(L_0\oplus L_0\oplus\mathbb C^{b-1})\longrightarrow\mathbb CP^a.
\]
Here $P(E)$ denotes lines in $E$; equivalently it is the
quotient-line projectivization of $E^*$.
Put $L_1=\ell^*$, the dual of the tautological line of this
projective bundle, and $x=c_1L_0$, $y=c_1L_1$.
The projective-bundle theorem and Euler sequences give
\[
H^*(B;\mathbb Z)=
\mathbb Z[x,y]/(x^{a+1},(x+y)^2y^{b-1}),
\]
\[
 TB\oplus\mathbb C^2\simeq_{\mathrm{st}}
 (a+1)L_0\oplus2(L_0L_1)\oplus(b-1)L_1.
\]
Its first Chern class is $(a+3)x+(b+1)y$, which is even.
Thus $B$ is spin.  The torus bundle has trivial vertical tangent
bundle, so $M_{r,N}$ is spin as well.  Choose the base complex orientation, normalized by
$\langle x^ay^b,[B]\rangle=1$.

The quotient $M_{r,N}\to B$ is the principal two-torus bundle obtained
from the Hopf two-torus bundle by fourth power in each circle.  In
associated-line conventions it is the fiber product of the circle
bundles of $L_0^4,L_1^4$.  If a principal-bundle convention produces
the dual lines instead, invert the corresponding reference generator
of $C_4$.  Such an inversion acts trivially on mod-two cohomology
($x_i,z_i$) and hence on the entire complex Adams $E_2$ page.  It
therefore does not change the associated-graded calculation below.
Make this convention once, uniformly for all the mixed spaces and
all the standard products used in their differences.  The assertion
of invariance is applied to $Y_r\in F^1$ after its top edge cancels:
an individual filtration-zero class could have convention-dependent
subleading coordinates.

\paragraph{The actual connective Gysin map.}
Let $u_i\in ku^2(\mathbb CP^\infty\times\mathbb CP^\infty)$ be the
complex-oriented Euler classes with $v u_i=L_i-1$ in periodic
$K$-theory.  Write $\beta_{ij}$ for the connective homology basis
dual to $u_0^iu_1^j$.  Naturality of the circle-bundle Gysin exact
sequence gives an iterated map
\[
\delta_1\delta_0:
ku_{2m}(\mathbb CP^\infty\times\mathbb CP^\infty)
\longrightarrow ku_{2m+2}(BC_4\times BC_4),
\]
which takes $v^{m-i-j}\beta_{ij}$ to
$v^{m-i-j}e_i\boxtimes e_j$.
Here $e_i$ is the cyclic generator defined by the circle-bundle
boundary of the dual projective-space basis.  This is precisely the
natural Gysin definition in
\cite[Proposition~2.1, Lemma~2.2, and diagram~(2.4)]{hashimotogenerators}; it is
not an identification made only after inverting $v$.
The product statement follows by applying those two circle Gysin
maps successively and using the external-product pairing.

The spin boundary of the disk bundle of $L_i^4$ differs from its
complex Thom orientation by tensoring with $L_i^{-2}$, the inverse
square root of the Thom determinant.  More explicitly, the two
spinor weights of the oriented real plane underlying $L_i^4$ are
$L_i^2$ and $L_i^{-2}$, so with the Euler convention above
\[
 c(e_{\rm Spin}(L_i^4))
 =\frac{L_i^2-L_i^{-2}}{v}
 =L_i^{-2}\frac{L_i^4-1}{v}.
\]
The other customary complex Thom normalization has Euler class
$(1-L_i^{-4})/v$, differing from Hashimoto's convention by the
unit $L_i^4$.  This unit restricts to the trivial line on
$S(L_i^4)$, so the resulting boundary maps agree by the projection
formula.  Consequently naturality of
the spin orientation and its complexification identifies a choice of
spin structure on $M_{r,N}$ with the iterated Gysin image of
\[
 [B]_{\rm Spin}\otimes L_0^{-2}L_1^{-2}.
\]
Now change this boundary spin structure by both sign characters of
$C_4^2$.  Their complexifications are the two flat lines
$L_0^2,L_1^2$ on $M_{r,N}$.  The projection formula cancels the two
inverse-square twists.  For this actual choice of spin structure,
the complexified class is exactly the iterated Gysin image of
$[B]_{\rm Spin}$.  We use that choice in the index calculation below.
Other spin choices are immaterial through Adams filtration one:
if $S_i=L_i-1=v u_i$, then
\[
 L_i^{-2}-1=-2S_i+3S_i^2-\cdots
\]
has Adams filtration at least two.  Here $v$ and $2$ each have
filtration one, and the cap operation $u_i$ has filtration zero.
This also explains why no spin-normalization factor may be inferred
by periodic Laurent division.

\begin{lemma}[The two top coefficient diagonals]\label{lem:mixed-spin-index-diagonals}
Let $d_{ij}$ be the coefficient of $v^{m-i-j}\beta_{ij}$ in the
pushforward of $[B]_{\rm Spin}$ to the product of infinite projective
spaces.  Put
\[
 I_i=d_{i,m-i},\qquad J_i=d_{i,m-1-i}.
\]
Then
\[
I_i=\begin{cases}(-1)^{a-i}(a-i+1),&0\le i\le a,\\0,&i>a,
\end{cases}
\]
and
\[
             J_i=\frac{iI_{i+1}+(m-1-i)I_i}{2}.
\]
In particular, for $0\le j\le r$,
\[
 J_{2j}\equiv r+1,\qquad
 J_{2j+1}\equiv r+1-N\pmod2,
\]
and the remaining $J_i$ vanish.
\end{lemma}
\begin{proof}
The coefficient $d_{ij}$ is a genuine integer spin-Dirac index.
Indeed the Kronecker pairing with $u_0^iu_1^j$, followed by the
coefficient identification $ku_{2(m-i-j)}(*)=\mathbb Z v^{m-i-j}$,
gives it.  The connective-to-periodic coefficient map in this
nonnegative degree is injective, so the ordinary twisted Dirac
index theorem computes the same integer \cite{AtiyahSingerIII}:
\[
 d_{ij}=\int_B\widehat A(TB)(e^x-1)^i(e^y-1)^j.
\]
The factor $v^{-(i+j)}$ implicit in the cohomological Euler classes
has already been accounted for by the displayed coefficient degree.

For $i+j=m$, only the lowest powers of the two Euler characters
contribute, so $I_i=\int_Bx^iy^{m-i}$.
Integration along the projective fiber has generating Segre class
$(1+x)^{-2}$.  Equivalently, reduce the relation
$(x+y)^2y^{b-1}=0$ with the normalization of the top class.
This gives the stated $I_i$.

For $i+j=m-1$, the $\widehat A$-class has no component in ordinary
cohomological degree two.  Thus its nonconstant terms cannot
contribute.  The only next terms of the two Euler characters are
\[
(e^x-1)^i(e^y-1)^j
=x^iy^j+\tfrac i2x^{i+1}y^j+\tfrac j2x^iy^{j+1}
 +\text{terms of degree at least }2m+2.
\]
This proves the formula for $J_i$.  The apparent halves are integers
because this is an actual spin index; they may also be checked
directly in the formula for $I_i$.

For $i=2j\le a$, one has
$I_i=-2(r-j+1)$ and $I_{i+1}=2(r-j)+1$.  Substitution gives
$J_{2j}\equiv j+(r-j+1)=r+1$ modulo two.
For $i=2j+1\le a$, one has $I_i=2(r-j)+1$ and
$I_{i+1}=-2(r-j)$.  The result is
$J_{2j+1}\equiv(r-j)+(N-j-1)=r+1-N$ modulo two.
\end{proof}

\begin{proposition}[The geometric first mixed row]\label{prop:mixed-first-row-repaired}
Let $Z_r$ be the real Dirac class of $M_{r,N}$, and let $X_j$ be the
standard spin lens product $P_{j,N}$ with edge $A_{2j+1}$.
Put $Y_r=Z_r-\sum_{j=0}^rX_j$.  Then $Y_r\in F^1$.
Under the real filtration-one identification modulo
$h_0\mathcal O_N$ from the algebraic calculation, its component in
the $v_0\mathcal B_N$ row has $h_0$-free $v$-coefficient vector
\[
\sum_{j=0}^r\bigl((r-j+1)B_{2j}+(r-j)B_{2j+1}\bigr).
\]
\end{proposition}
\begin{proof}
The top-diagonal index formula shows that the mod-two edge of $Z_r$
is $\sum_{j=0}^rA_{2j+1}$.  These are exactly the edges of
the indicated $X_j$, proving $Y_r\in F^1$.

Use the common actual Gysin generators
$g_{ij}=e_i\boxtimes e_j$ for $Z_r$ and the products, with the
boundary-spin normalization fixed above.  Let $\Delta I_i$ and
$\Delta J_i$ be the differences of their top two coefficient
diagonals.  Every lower diagonal has a factor $v^2$, so
\[
c(Y_r)\equiv
\sum_{i=0}^m\Delta I_i g_{i,m-i}
+v\sum_{i=0}^{m-1}\Delta J_i g_{i,m-1-i}
\pmod{F^2}.
\]
Edge cancellation makes every $\Delta I_i$ even.  The
filtration-zero symbols of these common generators are $A_i$ and
$B_i$, respectively.  Thus the class of $c(Y_r)$ in $F^1/F^2$ is
\[
\sum_{i=0}^m(\Delta I_i/2\bmod2)h_0A_i
+\sum_{i=0}^{m-1}(\Delta J_i\bmod2)vB_i.
\]
Any positive-filtration correction to a top-diagonal generator is
multiplied by an even coefficient, and any such correction to a
next-diagonal generator is multiplied by $v$.  Both therefore lie
in $F^2$.  This computes the filtration-one class of the actual
connective difference $Y_r$, not merely its periodic image.

For the $j$th product, use the spin base
$\mathbb CP^{2j+1}\times\mathbb CP^{m-2j-1}$ and the same two circle
Gysin maps.  Its top coefficient is one only at $i=2j+1$.
The same index calculation gives defect-one coefficients $j$ at
$B_{2j}$ and $N-j-1$ at $B_{2j+1}$.
Subtracting these from the preceding $J_i$ parities gives
$(r+1)+j=r-j+1$ and
$(r+1-N)+(N-j-1)=r-j$ modulo two.
The previously proved filtration-one row identification makes this
the claimed real vector modulo $h_0\mathcal O_N$.
\end{proof}

\begin{remark}[The forced $h_0$ corrections are not discarded]
With the ordinary Adams symbols $A_i$ defined above, the top-diagonal coefficient difference between
$Z_r$ and the sum of products is $-2(r-j+1)$ at $A_{2j}$ and
$2(r-j)$ at $A_{2j+1}$.  Thus its filtration-one $h_0$ coefficients
are precisely $r-j+1$ and $r-j$ modulo two.  They include even-index
terms, and these are not necessarily in $h_0\mathcal O_N$.
They are the forced corrections in the complexification of the
corresponding real odd bow lifts.  The assertion about the $v$ row
means extraction of the $h_0$-free coefficient followed by the
already established real row identification.  It is not a claim of
literal equality of bare $vB_i$ monomials modulo
$h_0\mathcal O_N$ in the complex group.
\end{remark}

\begin{corollary}[Spanning the first row by mixed quotients]\label{cor:mixed-first-row-spanning}
Write the real row, by its $h_0$-free complexification, in the basis
\[
p_0=B_0,\qquad p_j=B_{2j-1}+B_{2j}\ (1\le j<N),
\qquad p_N=B_{m-1}.
\]
The vectors obtained from $Y_0,\ldots,Y_{N-1}$ span
$p_0,\ldots,p_{N-1}$.  Swapping the two reference factors on $Y_0$
gives $p_N$.  Thus actual spin PSC mixed double lenses and standard
spin lens products generate the entire first row modulo the already
generated $h_0\mathcal O_N$.
\end{corollary}
\begin{proof}
The proved vector is
\[
w_r=(r+1)p_0+\sum_{j=1}^r(r-j+1)p_j\quad\text{over }\mathbb F_2.
\]
This is a lower-triangular matrix in $r=0,\ldots,N-1$ with all diagonal
entries equal to one.  The endpoint $r=N-1$ is legitimate: $b=1$,
the torus action remains free, the projective bundle is
$P(L_0\oplus L_0)$, and every index calculation above is unchanged.
There is no need for the strict inequality $b-1>0$ that appeared in
the former formal-localization calculation.
Both the mixed double lenses and the products are spin PSC: their
metrics descend from products of round metrics, and their specified
spin structures do not affect that metric.  Their signed differences
therefore belong to the PSC-generated subgroup.
\end{proof}

\paragraph{Inputs and scope.}
This proof uses the actual complex-oriented circle Gysin sequence
\cite[Proposition~2.1, Lemma~2.2, and diagram~(2.4)]{hashimotogenerators}, compatibility of
the spin and complex Thom orientations (the square-root determinant
twist), and the ordinary twisted spin-Dirac index theorem.  It does
not use an absolute Laurent representative in connective homology.
Only two adjacent coefficient diagonals are needed, so no assertion
about arbitrary higher filtered Gysin terms is made.

\subsubsection{Completing the even-dimensional generation}
\begin{proposition}[PSC generation in degree $4N+2$]
\label{prop:4N2-psc-generation-repaired}
The entire assembly kernel in degree $4N+2$ is generated by PSC
spin cycles.  Generators can be taken among the standard products
$P_{j,N}$, the mixed quotients $M_{r,N}$, their factor swaps, and
changes of their spin structures.
\end{proposition}
\begin{proof}
The products of round spin lens spaces
\[
 P_{j,N}=L^{4j+3}(4)\times L^{4(N-j)-1}(4),
 \qquad0\leq j<N,
\]
give the $N$ filtration-zero edges $o_{j+1}$.  Their doubles
give $h_0\mathcal O_N$.  There is no need to know their exact
complexified Gysin vectors at this stage.

In the real filtration-one quotient by $h_0\mathcal O_N$,
complexification followed by its $h_0$-free $v$ component
identifies the quotient with the paired basis $p_0,\ldots,p_N$.
The forced upper-cell terms $h_0A_{2i}$ in
$c(v_0p_i)=vp_i+h_0A_{2i}$ are retained in the real class;
they do not change this identification.  Proposition~\ref{prop:mixed-first-row-repaired}
therefore gives genuine real quotient vectors
\[
 w_r=p_r+p_{r-2}+p_{r-4}+\cdots\quad(0\leq r<N).
\]
In particular, $w_0=p_0$, $w_1=p_1$, and
$w_r+w_{r-2}=p_r$ for $2\leq r<N$.  These generate
$p_0,\ldots,p_{N-1}$.  Swapping the two factors of $G$ in
$Y_0$ supplies $p_N=B_{m-1}$.  The swap preserves both the
already generated row $h_0\mathcal O_N$ and the PSC property.
This generates the whole filtration-one row.  Notice that the
endpoint $r=N-1$ is geometrically valid: its second sphere is
$S^3$, and the toric spin-index calculation still applies.

Change the second spin structure on each standard product and
subtract the original spin class.  The difference lies in
$F^2$ by Lemma~\ref{lem:4N2-spin-strict-repaired}.
Its complex filtration-two term is
\[
 (h_0vS_1+v^2S_1^2)A_{2j+1},
 \qquad S_1m_{a,b}=m_{a,b-1}.
\]
The $h_0v$ term is the explicit telescoping $d_2$ boundary used
in degree $4N$.  For $0\leq j\leq N-2$, the other term is
$v^2C_{2j+1}$.  These form the basis
$v^2\mathcal O_{N-1}$.  Complexification is injective there,
and an unknown filtration-one correction of a product lift is
sent into $F^3$, so these are the required real pivots.

Choose PSC combinations $\widehat p_i\in F^1$, $0\leq i<N$,
whose leading terms are exactly the canonical $v_0p_i$.
Subtract the already generated doubles to remove any
$h_0\mathcal O_N$ component.  Their second-spin differences
$\Delta_1\widehat p_i$ lie in $F^3$.  To compute the row-three
terms apply the filtration-two complex operation to
$vp_i+h_0A_{2i}$.  Every $h_0$-divisible term containing a
positive power of $v$ is a telescoping complex $d_2$ boundary.
The remaining term is
\[
 v^3S_1^2p_i.
\]
The literal index-lowering identities are
\[
\begin{split}
 S_1^2p_0&=D_0,\\
 S_1^2p_i&=D_{2i-1}+D_{2i}\quad(1\leq i<N-1),\\
 S_1^2p_{N-1}&=D_{m-3}.
\end{split}
\]
They are exactly the paired basis $\mathcal B_{N-1}$.
Row-three injectivity identifies the full-rank pivots in the
real associated graded.  By the last assertion of
Lemma~\ref{lem:4N2-spin-strict-repaired}, an unknown
filtration-two correction in a chosen $\widehat p_i$ is sent
into $F^4$ and cannot alter them.  This argument uses neither
a global Wood-kernel order recurrence nor a periodic Laurent
representative.

All operations above are sums, orientation reversals, spin
changes, and factor swaps of actual PSC cycles, so all classes
remain in the PSC subgroup.  Their $5N$ independent permanent
associated-graded bits give a subgroup of order at least
$2^{5N}$.  The assembly-kernel upper bound is the same, so
this subgroup is the whole kernel.  Taking smash projections
preserves the PSC subgroup by Lemma~\ref{lem:psc-smash-projector}.
\end{proof}

\begin{theorem}[Even-dimensional kernel generation]
\label{thm:even-kernel-detection}
For every even $k\geq8$, the smash-summand connective assembly
kernel is generated by Dirac classes of closed spin PSC manifolds
over $BG$.
\end{theorem}
\begin{proof}
For $k=4N$, $N\geq2$, apply
Proposition~\ref{prop:4N-psc-generation-repaired}; its assembly
target has zero torsion image.  For $k=4N+2$, $N\geq2$, apply
Proposition~\ref{prop:4N2-psc-generation-repaired}.
Both arguments construct every required associated-graded row
using actual PSC cycles, rather than assuming integral tower lengths.
\end{proof}

\begin{remark}[Dimension six]
The degree-six coefficient class $\omega p$ requires a separate argument.
Lemma~\ref{lem:ko6-wood-class} proves its survival, and
Theorem~\ref{thm:degree-six-generation} supplies positive-scalar-curvature
generators for the entire degree-six smash group.
\end{remark}

\subsection{The GLR conclusion in dimensions at least seven}

\begin{theorem}[GLR for \(C_4\times C_4\)]\label{thm:glr-c4-c4}
The Gromov--Lawson--Rosenberg conjecture holds for
\(\mathbb Z/4\times\mathbb Z/4\) in dimensions \(n\geq7\).
\end{theorem}

\begin{proof}
Let \(ko_n^{\mathrm{psc}}(BG)\subset ko_n(BG)\) be the subgroup of Dirac classes represented by
closed spin \(n\)-manifolds over \(BG\) which admit positive-scalar-curvature metrics.
By Lemma~\ref{lem:stolz-bgs-reduction}, it is enough to prove
\[
\ker\bigl(A\circ{\rm per}:ko_n(BG)\to KO_n(C_r^*(G))\bigr)\subset ko_n^{\mathrm{psc}}(BG)
\]
for \(G=C_4\times C_4\) and \(n\geq7\).  The opposite inclusion
\(ko_n^{\mathrm{psc}}(BG)\subset\ker(A\circ{\rm per})\) is the Lichnerowicz index
obstruction.

By Lemma~\ref{lem:glr-basepoint-kernel}, the basepoint summand contributes no
kernel.  The remaining kernel is the direct sum of the two cyclic kernels and
the smash-summand kernel.  The cyclic kernels are generated
by positive-scalar-curvature classes by Theorem~\ref{thm:cyclic-kernel-orders}.
If \(n\) is odd, then \(n\geq7\), and the full connective kernel is generated
by the induced BGS classes by Theorem~\ref{thm:odd-kernel-detection}.  If \(n\)
is even, then \(n\geq8\), and the even smash-summand kernel is generated
by the explicit PSC cycles of Theorem~\ref{thm:even-kernel-detection}.
All these generators lie in \(ko_n^{\mathrm{psc}}(BG)\), and \(ko_n^{\mathrm{psc}}(BG)\) is a subgroup.
Therefore the full connective kernel is contained in \(ko_n^{\mathrm{psc}}(BG)\).  This proves
the required inclusion and hence the unstable GLR statement in every dimension
\(n\geq7\).
\end{proof}

\subsection{The exceptional dimensions and the main theorem}

\begin{theorem}[Degree-five psc generation]\label{thm:degree-five-generation}
The group $\widetilde{ko}_5(X)$ is generated by positive-scalar-curvature
cycles. Together with the two cyclic summands, these cycles generate
$\ker(A\circ\mathrm{per}:ko_5(BG)\to KO_5(C_r^*(G)))$.
\end{theorem}
\begin{proof}
By Proposition~\ref{prop:cyclic-low} and Corollary~\ref{cor:missing-coset},
the cyclic-induction subgroup $P_5$ has index two in
$\widetilde{ko}_5(X)$. Proposition~\ref{prop:missing-coset} places the affine
class in its nonzero quotient coset, and
Proposition~\ref{prop:affine-representative-psc} gives a positive-scalar-curvature
metric on that representative. Thus these classes generate the smash group.
The cyclic kernels are generated by psc classes by
Theorem~\ref{thm:cyclic-kernel-orders}; the basepoint contributes no kernel.
\end{proof}

\begin{proof}[Proof of Theorem~\ref{theo:gromovz42}]
Positive scalar curvature implies vanishing of the index. Conversely,
for $n\geq7$, Theorem~\ref{thm:glr-c4-c4} gives the conclusion. For $n=6$,
use Theorem~\ref{thm:degree-six-generation} in place of the even-dimensional
detection theorem and apply the same reduction.

For $n=5$, Theorem~\ref{thm:degree-five-generation} represents the Dirac
class of any zero-index manifold by a sum of psc cycles over $BG$.
The spin orientation $M\Spin\to ko$ induces an isomorphism on homology in
these low degrees, as used in Proposition~\ref{prop:ko5-exact}.
Consequently equality of the Dirac classes lifts to equality of spin
bordism classes over $BG$. Connect components and perform spin surgeries
over $BG$ to arrange the fundamental group, as in
Proposition~\ref{prop:pi1-surgery}. The bordism-and-surgery criterion in
dimension at least five then transfers positive scalar curvature to the
original manifold \cite{gromovlawsonsurgery}. This proves the claim in
all dimensions $n\geq5$.
\end{proof}

\appendix
\section{Exact reproducibility checks}\label{app:code}

They are regression checks of the arguments above, not a
replacement for the all-degree proofs.
\begin{enumerate}
\item \nolinkurl{checks/eta_matrix_checks.py} constructs the geometric
eta matrices with their actual target periods, checks the exact
fixed-point sums, and computes Smith forms.  The option
\texttt{-{}-symbolic} certifies every elimination multiplier over
$\mathbb Z_{(2)}$ for $a=4t$ with $t\geq1$; the exceptional case
$a=2$ is treated separately in Lemma~\ref{lem:eta-target-lattice}.
\item \nolinkurl{checks/check_ku_d3_symmetry.py} verifies the primitive
source and target dimensions and the unique nonzero symmetry-invariant
third-differential functional.
\item \nolinkurl{checks/check_real_e3_and_d3_lower_bounds.py}
constructs the real $E_2$ basis and the actual $R_2,d_2$ maps,
checks $d_2^2=0$, takes exact kernel/image quotients, and checks
the exhibited third-differential ranks through stem 31.
\item \nolinkurl{checks/check_E3_rational_identities.py}
certifies the all-degree real $E_3$ generating function and its
eight quadratic residue formulas by exact polynomial identities,
not by fitting finitely many coefficients.
\item \nolinkurl{checks/check_mixed_spin_index.py} verifies the two
integer index diagonals, their forced $h_0$ terms, and the
triangular first-row spanning matrix for $2\leq N\leq64$.
\item \nolinkurl{checks/check_mixed_geometry_spin_indices.py}
uses the full $\widehat A$ polynomial and exact connective
Gysin coordinates for the actual mixed quotients and products.
The resulting geometric complex-image orders are
$2^4,2^{10},2^{15},2^{20},2^{25}$ in dimensions
$6,10,14,18,22$, respectively.  The extra real degree-six class
is treated separately in the proof.
\end{enumerate}
The eta-matrix, rational-identity, and mixed-geometry checks use
Python--FLINT; the symbolic eta certificate additionally uses SymPy.
The other three scripts use only
the Python standard library.  Commands and dependency information
are given in the accompanying README.

\paragraph{Code availability.}
The verification scripts accompanying this appendix are available at
\url{https://github.com/raphaelreinauer/glr-c4xc4-checks}.
Version~1.0.0 is permanently archived on Zenodo under
DOI~\href{https://doi.org/10.5281/zenodo.24116313}
{10.5281/zenodo.24116313}~\cite{reinauer2026glrchecks}.

%The former historical charts are retained in the source archive for
%comparison but are not part of the proof.  In particular, a diagram
%with an obsolete page label or an incorrect low-degree count is not
%used to establish a differential or survival claim.

\bibliographystyle{abbrvurl}
\bibliography{GROMOV-LAWSON-ROSENBERG-ARXIV-09-11-2026}

\begin{thebibliography}{10}

\bibitem{adams}
J.~F. Adams.
\newblock {\em Stable homotopy and generalised homology}.
\newblock Chicago Lectures in Mathematics. University of Chicago Press,
  Chicago, IL, 1995.
\newblock Reprint of the 1974 original.

\bibitem{andersonbrownpeterson}
D.~W. Anderson, E.~H. Brown, Jr., and F.~P. Peterson.
\newblock The structure of the {S}pin cobordism ring.
\newblock {\em Ann. of Math. (2)}, 86:271--298, 1967.

\bibitem{atiyahpatodisinger}
M.~F. Atiyah, V.~K. Patodi, and I.~M. Singer.
\newblock Spectral asymmetry and {R}iemannian geometry. {I}.
\newblock {\em Math. Proc. Cambridge Philos. Soc.}, 77:43--69, 1975.

\bibitem{baereta}
C.~B\"{a}r.
\newblock Dependence of the {D}irac spectrum on the {S}pin structure.
\newblock In {\em Global analysis and harmonic analysis ({M}arseille-{L}uminy,
  1999)}, volume~4 of {\em S\'{e}min. Congr.}, pages 17--33. Soc. Math. France,
  Paris, 2000.

\bibitem{barreratwisted}
E.~Barrera-Yanez.
\newblock The eta invariant and the ``twisted'' connective {$K$}-theory of the
  classifying space for cyclic 2-groups.
\newblock {\em Homology Homotopy Appl.}, 8(2):105--114, 2006.

\bibitem{botvinnikgilkeystolz}
B.~Botvinnik, P.~Gilkey, and S.~Stolz.
\newblock The {G}romov-{L}awson-{R}osenberg conjecture for groups with periodic
  cohomology.
\newblock {\em J. Differential Geom.}, 46(3):374--405, 1997.

\bibitem{botvinnikgilkeyannalen}
B.~Botvinnik and P.~B. Gilkey.
\newblock The eta invariant and metrics of positive scalar curvature.
\newblock {\em Math. Ann.}, 302(3):507--517, 1995.

\bibitem{brunergreeenleescomplex}
R.~R. Bruner and J.~P.~C. Greenlees.
\newblock The connective {$K$}-theory of finite groups.
\newblock {\em Mem. Amer. Math. Soc.}, 165(785):viii+127, 2003.

\bibitem{davismorfismos}
D.~Davis.
\newblock Topological complexity of 2-torsion lens spaces and ku-(co)homology.
\newblock {\em Morfismos}, 18(2), 2014.

\bibitem{lueckdavis}
J.~F. Davis and W.~L\"{u}ck.
\newblock The topological {K}-theory of certain crystallographic groups.
\newblock {\em J. Noncommut. Geom.}, 7(2):373--431, 2013.

\bibitem{pearsondavis}
J.~F. Davis and K.~Pearson.
\newblock The {G}romov-{L}awson-{R}osenberg conjecture for cocompact {F}uchsian
  groups.
\newblock {\em Proc. Amer. Math. Soc.}, 131(11):3571--3578, 2003.

\bibitem{dold}
A.~Dold.
\newblock Parametrized {B}orsuk-{U}lam theorems.
\newblock {\em Comment. Math. Helv.}, 63(2):275--285, 1988.

\bibitem{donnelly}
H.~Donnelly.
\newblock Eta invariants for {$G$}-spaces.
\newblock {\em Indiana Univ. Math. J.}, 27(6):889--918, 1978.

\bibitem{stolzdwyerschick}
W.~Dwyer, T.~Schick, and S.~Stolz.
\newblock Remarks on a conjecture of {G}romov and {L}awson.
\newblock In {\em High-dimensional manifold topology}, pages 159--176. World
  Sci. Publ., River Edge, NJ, 2003.

\bibitem{fujiikobayashishimomurasugawara}
K.~Fujii, T.~Kobayashi, K.~Shimomura, and M.~Sugawara.
\newblock {$K{\rm O}$}-groups of lens spaces modulo powers of two.
\newblock {\em Hiroshima Math. J.}, 8(3):469--489, 1978.

\bibitem{gilkeybook}
P.~B. Gilkey.
\newblock {\em The geometry of spherical space form groups}, volume~28 of {\em
  Series in Pure Mathematics}.
\newblock World Scientific Publishing Co. Pte. Ltd., Hackensack, NJ, second
  edition, 2018.

\bibitem{gromovlawsonsurgery}
M.~Gromov and H.~B. Lawson, Jr.
\newblock The classification of simply connected manifolds of positive scalar
  curvature.
\newblock {\em Ann. of Math. (2)}, 111(3):423--434, 1980.

\bibitem{gromovlawsonfundamental}
M.~Gromov and H.~B. Lawson, Jr.
\newblock Spin and scalar curvature in the presence of a fundamental group.
  {I}.
\newblock {\em Ann. of Math. (2)}, 111(2):209--230, 1980.

\bibitem{hashimotogenerators}
S.~Hashimoto.
\newblock On the connective {$K$}-homology groups of the classifying spaces
  {$B{\bf Z}/p\sp{r}$}.
\newblock {\em Publ. Res. Inst. Math. Sci.}, 19(2):765--771, 1983.

\bibitem{hitchin}
N.~Hitchin.
\newblock Harmonic spinors.
\newblock {\em Advances in Math.}, 14:1--55, 1974.

\bibitem{hughes}
S.~Hughes.
\newblock On the equivariant {$K$}- and {$KO$}-homology of some special linear
  groups.
\newblock {\em Algebr. Geom. Topol.}, 21(7):3483--3512, 2021.

\bibitem{jaworowski}
J.~Jaworowski.
\newblock Involutions in lens spaces.
\newblock volume~94, pages 155--162. 1999.
\newblock Special issue in memory of B. J. Ball.

\bibitem{joachimmalhotra}
M.~Joachim and A.~Malhotra.
\newblock On the gromov-lawson-rosenberg conjecture for elementary abelian
  $2$-groups.
\newblock In preparation.

\bibitem{lichnerowicz}
A.~Lichnerowicz.
\newblock Spineurs harmoniques.
\newblock {\em C. R. Acad. Sci. Paris}, 257:7--9, 1963.

\bibitem{malhotra}
A.~Malhotra and K.~Rodtes.
\newblock The {G}romov-{L}awson-{R}osenberg conjecture for the semi-dihedral
  group of order 16.
\newblock {\em Glasg. Math. J.}, 57(2):365--386, 2015.

\bibitem{McCleary}
J.~McCleary.
\newblock {\em A user's guide to spectral sequences}, volume~58 of {\em
  Cambridge Studies in Advanced Mathematics}.
\newblock Cambridge University Press, Cambridge, second edition, 2001.

\bibitem{milnor}
J.~Milnor.
\newblock The {S}teenrod algebra and its dual.
\newblock {\em Ann. of Math. (2)}, 67:150--171, 1958.

\bibitem{moshertangora}
R.~E. Mosher and M.~C. Tangora.
\newblock {\em Cohomology operations and applications in homotopy theory}.
\newblock Harper \& Row, Publishers, New York-London, 1968.

\bibitem{ravenel}
D.~C. Ravenel.
\newblock {\em Complex cobordism and stable homotopy groups of spheres}, volume
  121 of {\em Pure and Applied Mathematics}.
\newblock Academic Press, Inc., Orlando, FL, 1986.

\bibitem{ravenelfinite}
D.~C. Ravenel.
\newblock The stable homotopy theory of finite complexes.
\newblock In {\em Handbook of algebraic topology}, pages 325--396.
  North-Holland, Amsterdam, 1995.

\bibitem{reinauer}
R.~Reinauer.
\newblock {\em Real and complex connective $K$-homology of finite abelian $2$-
  groups}.
\newblock PhD thesis, University of Münster, 2020.

\bibitem{robinsonkuenneth}
C.~A. Robinson.
\newblock A {K}\"unneth theorem for connective {$K$}-theory.
\newblock {\em J. London Math. Soc. (2)}, 17(1):173--181, 1978.

\bibitem{rosenbergalpha}
J.~Rosenberg.
\newblock {$C\sp{\ast} $}-algebras, positive scalar curvature, and the
  {N}ovikov conjecture.
\newblock {\em Inst. Hautes \'{E}tudes Sci. Publ. Math.}, (58):197--212, 1983.

\bibitem{schick}
T.~Schick.
\newblock A counterexample to the (unstable) {G}romov-{L}awson-{R}osenberg
  conjecture.
\newblock {\em Topology}, 37(6):1165--1168, 1998.

\bibitem{siegemeyer}
C.~Siegemeyer.
\newblock {\em On the Gromov-Lawson-Rosenberg Conjecture for finite abelian 2-
  groups of rank 2}.
\newblock PhD thesis, University of Münster, 2013.

\bibitem{stolzannals}
S.~Stolz.
\newblock Simply connected manifolds of positive scalar curvature.
\newblock {\em Ann. of Math. (2)}, 136(3):511--540, 1992.

\bibitem{stolzspin}
S.~Stolz.
\newblock Splitting certain {$M\, {\rm Spin}$}-module spectra.
\newblock {\em Topology}, 33(1):159--180, 1994.

\bibitem{stong}
R.~E. Stong.
\newblock Determination of {$H\sp{\ast} ({\rm BO}(k,\cdots,\infty ),Z\sb{2})$}
  and {$H\sp{\ast} ({\rm BU}(k,\cdots,\infty ),Z\sb{2})$}.
\newblock {\em Trans. Amer. Math. Soc.}, 107:526--544, 1963.

\bibitem{switzer}
R.~M. Switzer.
\newblock {\em Algebraic topology---homotopy and homology}.
\newblock Classics in Mathematics. Springer-Verlag, Berlin, 2002.
\newblock Reprint of the 1975 original [Springer, New York; MR0385836 (52
  \#6695)].

\bibitem{weibel}
C.~A. Weibel.
\newblock {\em An introduction to homological algebra}, volume~38 of {\em
  Cambridge Studies in Advanced Mathematics}.
\newblock Cambridge University Press, Cambridge, 1994.

\bibitem{bochner}
K.~Yano and S.~Bochner.
\newblock {\em Curvature and {B}etti numbers}, volume No. 32 of {\em Annals of
  Mathematics Studies}.
\newblock Princeton University Press, Princeton, NJ, 1953.

\end{thebibliography}
\end{document}